\documentclass{amsart}

\usepackage{amssymb,cleveref,nicematrix,tikz}
\usepackage{marvosym}
\usetikzlibrary{cd}

\usepackage[T1]{fontenc}
\usepackage[utf8]{inputenc}
\usepackage[english]{babel}

\newtheorem{theorem}{Theorem}[section]
\newtheorem{proposition}[theorem]{Proposition}
\newtheorem{corollary}[theorem]{Corollary}
\newtheorem{lemma}[theorem]{Lemma}

\theoremstyle{definition}
\newtheorem{definition}[theorem]{Definition}
\theoremstyle{remark}
\newtheorem{remark}[theorem]{Remark}

\DeclareMathOperator{\stab}{Stab}
\DeclareMathOperator{\slg}{SL}
\DeclareMathOperator{\psl}{PSL}
\DeclareMathOperator{\supp}{Supp}
\DeclareMathOperator{\aut}{Aut}
\DeclareMathOperator{\inn}{Inn}
\DeclareMathOperator{\out}{Out}
\DeclareMathOperator{\res}{res}
\DeclareMathOperator{\tr}{tr}

\DeclareMathOperator{\diag}{diag}

\newcommand{\gr}[1]{\Gamma^{(#1)}}
\DeclareMathOperator{\C}{C}
\DeclareMathOperator{\Pa}{P}
\DeclareMathOperator{\IC}{I_C}
\DeclareMathOperator{\IP}{I_P}
\newcommand{\IM}[1]{\operatorname{I}_{\mathrm{M}}^{#1}}
\newcommand{\III}{\mathrm{III}}
\DeclareMathOperator{\IIP}{II_P}
\DeclareMathOperator{\IIC}{II_C}
\DeclareMathOperator{\IIM}{II_M}

\newcommand{\FI}{1}
\newcommand{\FII}{2}
\newcommand{\FIII}{3}
\newcommand{\FIV}{4}
\newcommand{\FV}{5}

\title{Group gradings on exceptional simple Lie superalgebras}
\author{Sebastiano Argenti}
\address{Dipartimento di Matematica, Informatica ed Economia, Università degli Studi della Basilicata, Via dell'Ateneo Lucano 10, 85100 Potenza, Italy }
\email{sebastiano.argenti@unibas.it}
\author{Mikhail Kochetov}
\address{Department of Mathematics and Statistics, Memorial University of Newfoundland, St. John's, NL, A1C5S7, Canada.}
\email{mikhail@mun.ca}
\author{Felipe Yasumura}
\address{Department of Mathematics, Instituto de Matem\'atica e Estat\'istica, Universidade de S\~ao Paulo, SP, Brazil}
\email{fyyasumura@ime.usp.br}
\thanks{S.A. is supported by Università degli Studi della Basilicata and Università del Salento}
\thanks{M.K.~is supported by Discovery Grant 2018-04883 of the Natural Sciences and Engineering Research Council (NSERC) of Canada.}
\thanks{F.Y.~is supported by FAPESP, grant 2023/03922-8 and 2018/23690-6.}
\subjclass[2020]{17B70, 17B25}
\keywords{Exceptional Lie superalgebra, grading, classification}

\begin{document}
\begin{abstract}
We classify up to isomorphism the gradings by arbitrary groups on the exceptional classical simple Lie superalgebras $G(3)$, $F(4)$ and $D(2,1;\alpha)$ over an algebraically closed field of characteristic $0$. To achieve this, we apply the recent method developed by A.~Elduque and M.~Kochetov to the known classification of fine gradings up to equivalence on the same superalgebras, which was obtained by C.~Draper et al.~in 2011. We also classify gradings on the simple Lie superalgebra $A(1,1)$, whose automorphism group is different from the other members of the $A$ series.
\end{abstract}
\maketitle

\section{Introduction}
The classification of gradings by arbitrary groups on a given algebra has been a topic of significant interest, particularly after the groundbreaking works by Patera and Zassenhaus \cite{PZ} and Bahturin, Sehgal and Zaicev \cite{BSZ}. Such classifications have been obtained for many important classes of algebras, including simple finite-dimensional associative, Lie and Jordan algebras (see the monograph \cite{EK13} and the references therein). 

In this paper we focus on the exceptional simple Lie superalgebras. Recall that finite-dimensional simple Lie superalgebras over an algebraically closed field of characteristic $0$ were classified by Kac in \cite{Kac} and are of two types: classical and Cartan type. A simple Lie superalgebra $\mathcal{L}=\mathcal{L}_0\oplus\mathcal{L}_1$ is \emph{classical} if its odd component $\mathcal{L}_1$ is a semisimple module over the even component $\mathcal{L}_0$. There are six infinite series of such Lie superalgebras: four of them, denoted $A$, $B$, $C$, and $D$, are similar to the corresponding series of simple Lie algebras, while two more, denoted $P$ and $Q$, are known as the \emph{strange Lie superalgebras}. In addition to these, there are the \emph{exceptional} cases: $G(3)$, $F(4)$, and the family $D(2,1;\alpha)$.

All group gradings for the series $A$, $B$, $C$, and $D$ of simple Lie algebras were classified in \cite{BK10} up to isomorphism, and the so-called fine gradings were classified in \cite{Eld10} up to equivalence (the terminology is reviewed in Section~\ref{sec:preliminaries} below).
Moving on to superalgebras, classifications of gradings were first obtained for the series $Q$ in \cite{BHSK} and then for the series $P$ and partially $A$ in \cite{HSK}. The approach in these works was based on considering $\mathcal{L}_1$ as a graded $\mathcal{L}_0$-module. A different approach, similar to the works on simple Lie algebras cited above, was taken in the Ph.D. dissertation of Hornhardt \cite{HPhD}: transferring the problem to associative superalgebras with superinvolution and studying gradings on the latter. This resulted in a classification of all gradings up to isomorphism for all six series of classical Lie superalgebras, with the exception of  $A(1,1)=\mathfrak{psl}(2\,|\,2)$, which admits automorphisms and gradings that do not have counterparts in the associative case (a situation reminiscent of the Lie algebra $D_4$).

The main goal of this paper is the classification of group gradings up to isomorphism on the exceptional simple Lie superalgebras $G(3)$, $F(4)$ and $D(2,1;\alpha)$, as well as $A(1,1)$. To achieve this, we start with the classification of fine gradings up to equivalence obtained in \cite{DEM} (excluding $A(1,1)$, which we handle separately) and apply a method developed in \cite{EK23}. This method involves obtaining the so-called almost fine gradings from the fine gradings and computing their respective Weyl groups, which are of independent interest. Then the action of the Weyl group on the admissible homomorphisms from the universal grading group to a given group $G$ will give the isomorphism classes of $G$-gradings without any redundancy.

The paper is structured as follows. To begin with, Section~\ref{sec:preliminaries} provides an overview of the basic notions and notations that are essential for this work. It also includes some preliminary results. Then each of the following sections gives the classification of group gradings for one case: Section~\ref{sec:g3} for $G(3)$ (see \Cref{classG3}), Section~\ref{sec:f4} for $F(4)$ (see \Cref{thm:F4}), Section~\ref{sec:d} for the family  $D(2,1;\alpha)$ (see \Cref{thm:D1,thm:D2,thm:D3}), and finally Section~\ref{sec:a} for $A(1,1)$ (see \Cref{thm:psl}). The last one is obtained as a consequence of the classification of gradings for the non-simple Lie superalgebra $D(2,1;-1)$.

\section{Preliminaries\label{sec:preliminaries}}
 
We assume that $\mathbb{F}$ is an algebraically closed field of characteristic zero. Every vector space and algebra will be over $\mathbb{F}$. We use multiplicative notation for a group $G$ and write $G_{[2]}=\{g\in G\mid g^2=e\}$ and $G^{[2]}=\{g^2\mid g\in G\}$. We shall use additive notation for the infinite and finite cyclic groups $\mathbb{Z}$ and $\mathbb{Z}_n$, but whenever multiplicative notation is more convenient, we will write $C_n$ instead of $\mathbb{Z}_n$. 

\subsection{Gradings}
Let $\mathcal{A}$ be a vector space. A \emph{grading} $\Gamma$ on $\mathcal{A}$ is a family of subspaces $\{\mathcal{A}_s\}_{s\in S}$, called the homogeneous components, such that $\mathcal{A}=\bigoplus_{s\in S}\mathcal{A}_s$. It will be convenient for us to assume that $\mathcal{A}_s\ne0$ for every $s\in S$. If $\mathcal{A}$ is an algebra with a set of multilinear operations $\Omega$ (an $\Omega$-algebra), then we require the following condition: for each $n$-ary operation $\omega$ and all $s_1$, $s_2$, \dots, $s_n\in S$ such that $\omega(\mathcal{A}_{s_1},\ldots,\mathcal{A}_{s_n})\ne0$, there exists $s\in S$ such that $\omega(\mathcal{A}_{s_1},\ldots,\mathcal{A}_{s_n})\subseteq\mathcal{A}_s$. We say that $\Gamma$ is a \emph{group grading} if it can be realized as a $G$-grading for some group $G$, that is, there exists an injective map $S\to G$ such that the above element $s$ equals the product $s_1\cdots s_n$ in $G$. In other words, $\omega(\mathcal{A}_{s_1},\ldots,\mathcal{A}_{s_n})\subseteq\mathcal{A}_{s_1\cdots s_n}$, for each $n$-ary operation $\omega$ and all $s_1$, \dots, $s_n\in S$, where we set $\mathcal{A}_g=0$ if $g\notin S$. In this case, we say that  $(\mathcal{A},\Gamma)$, or just $\mathcal{A}$ if $\Gamma$ is clear from the context, is a $G$-graded $\Omega$-algebra. Also, we call $S$ the \emph{support} of $\Gamma$, or of $\mathcal{A}$, and write $S=\mathrm{Supp}(\Gamma)$ or $S=\mathrm{Supp}(\mathcal{A})$. The nonzero elements of $\mathcal{A}_s$ are called \emph{homogeneous} of degree $s$. If $0\ne x\in\mathcal{A}_s$, we write $\deg x=s$. A subspace $\mathcal{V}\subseteq\mathcal{A}$ is said to be \emph{graded} if $\mathcal{V}=\bigoplus_{s\in S}\mathcal{V}\cap\mathcal{A}_s$.

Among the groups that may realize a grading $\Gamma$, there is one which satisfies a universal property. It can be constructed as follows. Let $U=U(\Gamma)$ be the group generated by $S=\mathrm{Supp}(\Gamma)$, subject to the relations $s=s_1\cdots s_n$ whenever $0\ne\omega(\mathcal{A}_{s_1},\ldots,\mathcal{A}_{s_n})\subseteq\mathcal{A}_s$ for some $n$-ary operation $\omega$. The canonical map $S\to U(\Gamma)$ is injective if and only if $\Gamma$ is a group grading. In this case, $\Gamma$ is realized as a $U(\Gamma)$-grading, and this realization satisfies the following universal property: for any realization of $\Gamma$ as a $G$-grading, there exists a unique group homomorphism $U(\Gamma)\to G$ that is the identity on $S$. We call $U(\Gamma)$ the \emph{universal group} of $\Gamma$. 

Let $U_\mathrm{ab}(\Gamma)$ be the abelianization of $U(\Gamma)$. Then $\Gamma$ can be realized as a $G$-grading with an abelian group $G$ if and only if $S\to U_\mathrm{ab}(\Gamma)$ is injective. If this is the case, then we say that $\Gamma$ is an \emph{abelian group grading}. It is known that $U_\mathrm{ab}(\Gamma)=U(\Gamma)$ for any group grading $\Gamma$ on a simple Lie (super)algebra (see, e.g., \cite[Corollary 1.21]{EK13}). Hence, in this paper, we will only consider abelian group gradings.

\subsubsection{Isomorphism and equivalence} Consider two $G$-graded algebras $\mathcal{A}=\bigoplus_{g\in G}\mathcal{A}_g$ and $\mathcal{B}=\bigoplus_{g\in G}\mathcal{B}_g$. We say that $\mathcal{A}$ and $\mathcal{B}$ are \emph{isomorphic as graded algebras} if there exists an isomorphism of algebras $\varphi:\mathcal{A}\to\mathcal{B}$ such that $\varphi(\mathcal{A}_g)=\mathcal{B}_g$, for each $g\in G$. Given two gradings $\Gamma$ and $\Gamma'$ on the same algebra $\mathcal{A}$, we say that $\Gamma$ and $\Gamma'$ are \emph{isomorphic} and write $\Gamma\cong\Gamma'$ if $(\mathcal{A},\Gamma)$ and $(\mathcal{A},\Gamma')$ are isomorphic as graded algebras.

Now, let $\Gamma$ be a $G$-grading on $\mathcal{A}$ and $\Gamma'$ be an $H$-grading on $\mathcal{B}$. We say that the graded algebras $(\mathcal{A},\Gamma)$ and $(\mathcal{B},\Gamma')$ are \emph{equivalent} if there exists a map $\alpha:\mathrm{Supp}(\Gamma)\to\mathrm{Supp}(\Gamma')$ and an algebra isomorphism $\varphi:\mathcal{A}\to\mathcal{B}$ such that $\varphi(\mathcal{A}_g)=\mathcal{B}_{\alpha(g)}$, for each $g\in\mathrm{Supp}(\Gamma)$. In general, the map $\alpha$ is just a bijection of sets. However, if $G=U(\Gamma)$ and $H=U(\Gamma')$, then $\alpha$ extends to a group isomorphism $G\to H$. Given two gradings $\Gamma$ and $\Gamma'$ on the same algebra $\mathcal{A}$, we say that $\Gamma$ and $\Gamma'$ are \emph{equivalent} if $(\mathcal{A},\Gamma)$ and $(\mathcal{A},\Gamma')$ are equivalent.

\subsubsection{Refinement and coarsening}
Let $\Gamma:\mathcal{A}=\bigoplus_{g\in G}\mathcal{A}_g$ and $\Gamma':\mathcal{A}=\bigoplus_{h\in H}\mathcal{A}_h'$ be gradings on $\mathcal{A}$. We say that $\Gamma$ is a \emph{refinement} of $\Gamma'$ (or $\Gamma'$ is a \emph{coarsening} of $\Gamma$) if for each $g\in\mathrm{Supp}(\Gamma)$, there exists (necessarily unique) $h\in\mathrm{Supp}(\Gamma')$ such that $\mathcal{A}_g\subseteq\mathcal{A}_{h}'$. A refinement is said to be \emph{proper}, if for at least one $g\in G$, we have $\mathcal{A}_g\ne\mathcal{A}_{h}'$. A grading $\Gamma$ is said to be \emph{fine} if it admits no proper refinement.

There is a special construction that gives a coarsening of a grading. Let $\Gamma$ be a $G$-grading on $\mathcal{A}$, and let $\alpha:G\to H$ be a group homomorphism. Then $\alpha$ induces an $H$-grading on $\mathcal{A}$, denoted by $^\alpha\Gamma$, whose components are  $\mathcal{A}_h'=\bigoplus_{g\in\alpha^{-1}(h)}\mathcal{A}_g$, for each $h\in H$. If $G=U(\Gamma)$, then every coarsening of $\Gamma$ is obtained in this way.

\subsubsection{Automorphisms and the Weyl group of a grading} 

Following \cite{PZ} (see also \cite[\S 1.1]{EK13}), we can associate three important subgroups of $\aut(\mathcal{A})$ to any grading $\Gamma$ of $\mathcal{A}$. First, the \emph{diagonal group} of $\Gamma$ is
\[
\mathrm{Diag}(\Gamma)=\{\varphi\in\aut(\mathcal{A})\mid\forall s\in \mathrm{Supp}(\Gamma),\,\exists\lambda_s\in\mathbb{F}^\times\text{ such that }\varphi|_{\mathcal{A}_s}=\lambda_s\mathrm{id}_{\mathcal{A}_s}\}.
\]
This is a central subgroup of the \emph{stabilizer} of $\Gamma$:
\[
\mathrm{Stab}(\Gamma)=\{\varphi\in\aut(\mathcal{A})\mid\forall s\in \mathrm{Supp}(\Gamma),\,\varphi(\mathcal{A}_s)=\mathcal{A}_s\}.
\]
In its turn, $\mathrm{Stab}(\Gamma)$ is a normal subgroup of the \emph{automorphism group} of $\Gamma$:
\[
\mathrm{Aut}(\Gamma)=\{\varphi\in\aut(\mathcal{A})\mid\forall s\in\mathrm{Supp}(\Gamma),\,\exists t\in\mathrm{Supp}(\Gamma)\text{ such that }\varphi(\mathcal{A}_{s})=\mathcal{A}_{t}\}.
\]
Note that, if $\Gamma$ is a $G$-grading, then $\mathrm{Stab}(\Gamma)$ is the group of automorphisms and $\mathrm{Aut}(\Gamma)$ is the group of self-equivalences of the graded algebra $(\mathcal{A},\Gamma)$. 

The following result will be useful:

\begin{lemma}[{\cite[Corollary 3.7]{EK23}}]\label{l:DiagStab}
Let $\Gamma$ be a fine abelian group grading on a finite-dimensional algebra $\mathcal{A}$ over an algebraically closed field of characteristic zero. If $\aut(\mathcal{A})$ is reductive, then $\mathrm{Stab}(\Gamma)=\mathrm{Diag}(\Gamma)$.
\end{lemma}

Finally, the Weyl group of $\Gamma$ is defined as the quotient
\[
W(\Gamma)=\aut(\Gamma)/\mathrm{Stab}(\Gamma).
\]
By definition, $W(\Gamma)$ can be identified with a group of permutations of the set $\mathrm{Supp}(\Gamma)$. If $\Gamma$ is a group grading, these permutations extend to automorphisms of the universal group $U(\Gamma)$ and we may regard $W(\Gamma)$ as a subgroup of $\aut(U(\Gamma))$.

\subsubsection{Duality}

There is a duality between gradings and actions, which we review in the case of abelian groups. Recall that we assume that the base field $\mathbb{F}$ is algebraically closed of characteristic zero. Let $\Gamma$ be a $G$-grading on a finite-dimensional $\Omega$-algebra $\mathcal{A}$, where $G$ is a finitely generated abelian group. Then the automorphism group $\aut(\mathcal{A})$ and the character group $\widehat{G}=\{\chi:G\to\mathbb{F}^\times\text{ homomorphism}\}$ have a canonical structure of (affine) algebraic groups, and we obtain a homomorphism of algebraic groups $\eta_\Gamma:\widehat{G}\to\aut(\mathcal{A})$ by setting
\[
\eta_\Gamma(\chi)(a)=\chi(g)a,\quad \forall \chi\in\widehat{G},\,g\in G,\,a\in\mathcal{A}_g.
\]
Conversely, any homomorphism of algebraic groups $\eta:\widehat{G}\to\aut(\mathcal{A})$ is a representation of $\widehat{G}$, compatible with the operations of $\mathcal{A}$. Since $\widehat{G}$ is a \emph{quasitorus}, i.e., the direct product of an algebraic torus and a finite abelian group, it is known that all representations of $\widehat{G}$ are diagonalizable. Thus we obtain an eigenspace decomposition of $\mathcal{A}$ indexed by the character group $\mathfrak{X}(\widehat{G})\cong G$: $\mathcal{A}=\bigoplus_{g\in G}\mathcal{A}_g$, and this is easily checked to be a $G$-grading on $\mathcal{A}$.

If $Q=\eta_\Gamma(\widehat{G})$, then $\mathrm{Diag}(\Gamma)$ contains $Q$, $\mathrm{Stab}(\Gamma)$ is the centralizer  $C(Q)$, and $\mathrm{Aut}(\Gamma)$ contains the normalizer $N(Q)$ in the group $\aut(\mathcal{A})$. 

Recall that $\Gamma$ can be realized as a grading by $U_\mathrm{ab}(\Gamma)$. It follows from the construction of $U_\mathrm{ab}(\Gamma)$ that its character group can be identified with $\mathrm{Diag}(\Gamma)$ and, hence, in the case $G=U_\mathrm{ab}(\Gamma)$, we have $\mathrm{Diag}(\Gamma)=Q$ and $\mathrm{Aut}(\Gamma)=N(Q)$. Therefore, 
\[
W(\Gamma)=N(\mathrm{Diag}(\Gamma))/C(\mathrm{Diag}(\Gamma)).
\]

Finally, it is worth mentioning that $\Gamma$ is a fine grading if and only if $\mathrm{Diag}(\Gamma)$ is a maximal quasitorus in $\aut(\mathcal{A})$. This gives a bijection between the equivalence classes of fine abelian group gradings on $\mathcal{A}$ and the conjugacy classes of maximal quasitori in $\aut(\mathcal{A})$.

\subsubsection{Gradings on Lie superalgebras} Recall that a \emph{Lie superalgebra} (in characteristic different from $2,3$) is a vector space $\mathcal{L}$ endowed with a bilinear operation $[\,,]$ and a $\mathbb{Z}_2$-grading $\mathcal{L}=\mathcal{L}_0\oplus \mathcal{L}_1$ satisfying $[\mathcal{L}_i,\mathcal{L}_j]\subseteq \mathcal{L}_{i+j}$, for $i,j\in\mathbb{Z}_2$, and the super versions of antisymmetry and Jacobi identities:
\begin{enumerate}
\renewcommand{\labelenumi}{(\roman{enumi})}
\item $[x,y]=(-1)^{ij}[y,x]$, for all $x\in \mathcal{L}_i$, $y\in \mathcal{L}_j$,
\item $[[x,y],z]=(-1)^{jk}[[x,z],y] + [x,[y,z]]$, for all $x\in \mathcal{L}_i$, $y\in \mathcal{L}_j$, $z\in \mathcal{L}_k$.
\end{enumerate}
It is convenient to denote $|x|=i$ if $0\ne x\in \mathcal{L}_i$ and refer to it as the \emph{parity} of $x$. 
The linear map $\delta:\mathcal{L}\to\mathcal{L}$ defined by $\delta(x)=(-1)^{|x|}x$ is the automorphism of the bracket $[\,,]$ corresponding to the $\mathbb{Z}_2$-grading via duality. It can be used to interpret a Lie superalgebra as an $\Omega$-algebra, where $\Omega$ consists of the binary operation $[\,,]$ and the unary operation $\delta$. The automorphisms, respectively gradings, of a Lie superalgebra $\mathcal{L}$ can be defined as the automorphisms, respectively gradings, of this $\Omega$-algebra. In other words, an automorphism must preserve $[\,,]$ and commute with $\delta$, i.e., leave both $\mathcal{L}_0$ and $\mathcal{L}_1$ invariant. For example, $\delta$ itself has this property; we will call it the \emph{parity automorphism}. A $G$-grading $\Gamma$ on $\mathcal{L}$ is a vector space decomposition $\mathcal{L}=\bigoplus_{g\in G}\mathcal{L}_g$ such that $[\mathcal{L}_g,\mathcal{L}_h]\subseteq \mathcal{L}_{gh}$ and $\delta(\mathcal{L}_g)\subseteq\mathcal{L}_g$, for all $g,h\in G$. The latter condition means that the grading $\Gamma$ and the canonical $\mathbb{Z}_2$-grading are compatible in the sense that both subspaces $\mathcal{L}_0$ and $\mathcal{L}_1$ are graded with respect to $\Gamma$ or, equivalently, each homogeneous component $\mathcal{L}_g$, $g\in G$, is graded with respect to the canonical $\mathbb{Z}_2$-grading. We denote by $\Gamma_0$ and $\Gamma_1$ the restriction of $\Gamma$ to the subspaces $\mathcal{L}_0$ and $\mathcal{L}_1$, respectively.

It follows from the definition that a $G$-grading on $\mathcal{L}$ can be refined to a $\mathbb{Z}_2\times G$-grading with homogeneous components $\mathcal{L}_{(i,g)}=\mathcal{L}_i\cap \mathcal{L}_g$. In particular, if $\Gamma$ is fine, then this refinement must be equivalent to $\Gamma$, which means $\mathrm{Supp}(\Gamma_0)\cap\mathrm{Supp}(\Gamma_1)=\emptyset$. Whenever this latter condition holds, we have a well-defined \emph{parity homomorphism} $p:U(\Gamma)\to\mathbb{Z}_2$ that maps $\mathrm{Supp}(\Gamma_0)$ to $0$ and $\mathrm{Supp}(\Gamma_1)$ to $1$.

\subsection{Examples of gradings}\label{s:examplesgr}

The following two examples are well known (see, e.g., \cite{EK13}), but we review them here to fix notation, because they will be used as building blocks throughout the paper. Let $G$ be an abelian group.

\begin{definition}\label{examples}
        On the matrix algebra $M_2(\mathbb{F})$, we have two types of $G$-gradings:
        \begin{itemize}
        \item For any $g\in G$, $\gr{\C}_{M_2}(g)$ is defined by assigning degrees to the basis elements
        \[
        I=\begin{pmatrix}1 & 0 \\ 0 & 1 \end{pmatrix},\, H=\begin{pmatrix}1 & 0 \\ 0 & -1 \end{pmatrix},\, E=\begin{pmatrix}0 & 1 \\ 0 & 0 \end{pmatrix},\, F=\begin{pmatrix}0 & 0 \\ 1 & 0 \end{pmatrix}
        \]
        as follows: $\deg(I)=\deg(H)=e$, $\deg(E)=g$, $\deg(F)=g^{-1}$. If $g^2\ne e$, this grading has type $(2,1)$, i.e., two components of dimension $1$ and one of dimension $2$. If $g^2=e\ne g$, it has type $(0,2)$.
        \item For $a,b\in G$ satisfying $\langle a,b\rangle=\{e,a,b,c\}\cong \mathbb{Z}_2^2$, $\gr{\Pa}_{M_2}(a,b)$ is defined by assigning degrees to the basis elements
        \[
        I=\begin{pmatrix}1 & 0 \\ 0 & 1 \end{pmatrix},\, A=\begin{pmatrix}i & 0 \\ 0 & -i \end{pmatrix},\,
        B=\begin{pmatrix}0 & -1 \\ 1 & 0 \end{pmatrix},\, C=\begin{pmatrix}0 & -i \\ -i & 0 \end{pmatrix}
        \]
        as follows: $\deg(I)=e$, $\deg(A)=a$, $\deg(B)=b$, $\deg(C)=ab=c$ ($i^2=-1$).
        \end{itemize}
        The restrictions of these gradings to $\mathfrak{sl}_2$ will be denoted by $\gr{\C}_{\mathfrak{sl}_2}(g)$ and $\gr{\Pa}_{\mathfrak{sl}_2}(a,b)$, respectively.
        \end{definition}
        
    Note that the matrices $A,B,C$ are scalar multiples of the Pauli matrices, hence the notation $\gr{\Pa}_{M_2}(a,b)$. The scalars are chosen so that $A,B,C\in\slg_2(\mathbb{F})$. This choice does not affect the grading, but will be convenient in Section~\ref{sec:d}.

    As to $\gr{\C}_{M_2}(g)$, it is induced by the group homomorphism $\mathbb{Z}\to G$, $1\mapsto g$, from the so-called \emph{Cartan grading} of $M_2(\mathbb{F})$, i.e., the grading corresponding to a maximal torus in the automorphism group.

    It is well known (see, e.g., \cite[Theorems 2.27 and 3.49]{EK13}) that any $G$-grading on $M_2(\mathbb{F})$, respectively $\mathfrak{sl}_2$, is isomorphic to $\gr{\C}_{M_2}(g)$ or $\gr{\Pa}_{M_2}(a,b)$, respectively $\gr{\C}_{\mathfrak{sl}_2}(g)$ or $\gr{\Pa}_{\mathfrak{sl}_2}(a,b)$.
    
    The Weyl group of the Cartan grading on $M_2(\mathbb{F})$ is cyclic of order $2$, generated by the inversion on its universal group $\mathbb{Z}$. The Weyl group of the Pauli grading is the full automorphism group of the universal group $\mathbb{Z}_2^2$, isomorphic to the symmetric group $S_3$ on the set $\{a,b,c\}$. It follows that the isomorphism class of $\gr{\Pa}_{\mathfrak{sl}_2}(a,b)$ depends only on the subgroup $\langle a,b\rangle$.

    Both $\gr{\C}_{M_2}(g)$ and $\gr{\Pa}_{M_2}(a,b)$ can be generalized to $M_n(\mathbb{F})$ as follows. 
    
\begin{definition}\label{def:elemM}
Given any $n$-tuple $\gamma=(g_1,\ldots,g_n)$ of elements of $G$, we obtain a $G$-grading $\Gamma_{M_n}(\gamma)$ on $M_n(\mathbb{F})$ by assigning degree $g_j g_k^{-1}$ to the matrix unit $E_{jk}$. Such gradings on $M_n(\mathbb{F})$ are called \emph{elementary}. 
\end{definition}

For example, $\gr{\C}_{M_2}(g)=\Gamma_{M_2}(\gamma)$ for $\gamma=(e,g^{-1})$. It is important to note that, if we define a $G$-grading $\Gamma_V(\gamma)$ on the vector space $V=\mathbb{F}^n$ by assigning degree $g_j$ to the $j$-th standard basis vector, then $\Gamma_{M_n}(\gamma)$ is induced on $M_n(\mathbb{F})=\mathrm{End}(V)$ in the following sense: a nonzero operator $T\in\mathrm{End}(V)$ is homogeneous of degree $g$ if and only if $T(V_h)\subseteq V_{gh}$ for all $h\in G$. In particular, if $V$ is a super vector space, i.e., it is equipped with a $\mathbb{Z}_2$-grading $V=V_0\oplus V_1$, then the associative algebra $\mathrm{End}(V)$ becomes $\mathbb{Z}_2$-graded, too, so we can define the supercommutator:
\[
[x,y]=xy-(-1)^{|x|\,|y|}yx.
\]
This operation turns $\mathrm{End}(V)$ into a Lie superalgebra, denoted by $\mathfrak{gl}(V)$, or $\mathfrak{gl}(n_0\,|\,n_1)$ if $\dim V_0=n_0$ and $\dim V_1=n_1$. All six series of classical simple Lie superalgebras can be defined as subquotients of $\mathfrak{gl}(n_0\,|\,n_1)$, including $\mathfrak{psl}(n_0\,|\,n_1)$ for $A$ and  $\mathfrak{osp}(n_0\,|\,n_1)$ for $B$, $C$ and $D$.

\begin{definition}\label{def:PauliM}
If $\zeta\in\mathbb{F}$ is a primitive $n$-th root of unity, then the diagonal matrix $X=\mathrm{diag}(1,\zeta,\ldots\zeta^{n-1})$ and the permutation matrix $Y$ corresponding to the cycle $(1,2,\ldots,n)$ commute up to $\zeta$ and satisfy $X^n=Y^n=I$. Therefore, $M_n(\mathbb{F})=\bigoplus_{j,k=0}^{n-1}\mathbb{F}X^jY^k$ is a grading by $\mathbb{Z}_n\times\mathbb{Z}_n$, called a \emph{generalized Pauli} grading. 
\end{definition}

The gradings $\gr{\C}_{M_2}(g)$ and $\gr{\Pa}_{M_2}(a,b)$ can be generalized in a different direction. Recall that $M_2(\mathbb{F})$ is the split quaternion algebra and can be obtained by two steps of the Cayley-Dickson doubling process starting from $\mathbb{F}$. A third step produces the (split) Cayley algebra $\mathcal{C}$.
All group gradings on $\mathcal{C}$ were described in \cite{Eld} (see also \cite[\S 4.2]{EK13}). The following two families cover all isomorphism classes. 

\begin{definition}\label{def:cayleyI}
Let $\{e_1,e_2,u_1,u_2,u_3,v_1,v_2,v_3\}$ be a `good basis' of $\mathcal{C}$. Given a triple $(h_1,h_2,h_3)$ of elements of $G$ satisfying $h_1h_2h_3=e$, we obtain a $G$-grading $\gr{\C}_\mathcal{C}(h_1,h_2,h_3)$ by assigning degrees as follows:
\begin{align*}
&\deg e_1=\deg e_2=e,\\
&\deg u_i=h_i,\,\deg v_i=h_i^{-1}\quad (i=1,2,3).
\end{align*}
\end{definition}

\begin{definition}\label{def:cayleyII}
 For $a_1,a_2,a_3\in G$ satisfying $T:=\langle a_1,a_2,a_3\rangle\cong\mathbb{Z}_2^3$, we obtain a $G$-grading on $\mathcal{C}$ by assigning degree $a_i$ to the generator added at the $i$-th step of the Cayley-Dickson doubling process. Since the isomorphism class of this grading depends only on the subgroup $T$, we will use notation $\gr{\Pa}_\mathcal{C}(T)$.
\end{definition}

\subsection{Almost fine gradings and admissible homomorphisms}
We now describe the method introduced in \cite{EK23} to obtain, for a given finite-dimensional $\Omega$-algebra $\mathcal{A}$ over an algebraically closed field and any abelian group $G$, a classification of $G$-gradings on $\mathcal{A}$ up to isomorphism if the fine abelian group gradings on $\mathcal{A}$ are known up to equivalence. We shall state the results in a simpler form, since we assume that the base field has characteristic $0$.

Let $\mathcal{A}$ be a finite-dimensional algebra. Then any $G$-grading $\Gamma$ on $\mathcal{A}$ is a coarsening of a fine grading $\Delta$, so $\Gamma$ is induced by a homomorphism $\alpha:U\to G$ where $U=U_\mathrm{ab}(\Delta)$. However, neither $\Delta$ nor $\alpha$ are uniquely determined. The main idea of \cite{EK23} is to force uniqueness (up to the action of the Weyl group of $\Delta$) by imposing a certain condition on $\alpha$ while slightly relaxing the fineness condition on $\Delta$. 

\begin{definition}[{\cite[Definition 3.2]{EK23}}]
A grading $\Gamma$ on $\mathcal{A}$ is \emph{almost fine} if the connected component of the identity  $\mathrm{Diag}(\Gamma)^\circ$ is a maximal torus in $\mathrm{Stab}(\Gamma)$.
\end{definition}

A $G$-grading on $\mathcal{A}$ induces a $G$-grading on the Lie algebra $\mathrm{Der}(\mathcal{A})$, which can be used to characterize almost fine gradings (in characteristic $0$): 
\begin{theorem}[{\cite[Theorem 5.4]{EK23}}]\label{thm:almostfine}
Assume that $\aut(\mathcal{A})$ is reductive, let $\Delta$ be a fine abelian group grading on $\mathcal{A}$, $U=U_\mathrm{ab}(\Delta)$, and let $\Delta'$ be an abelian group grading that is a coarsening of $\Delta$, so $U_\mathrm{ab}(\Delta')=U/E$, for some subgroup $E\subseteq U$. Then $\Delta'$ is almost fine if and only if $E\subseteq t(U)$ (the torsion subgroup of $U$) and $E\cap\Sigma\subseteq\{e\}$, where $\Sigma$ is the support of the induced grading on $\mathrm{Der}(\mathcal{A})$.
\end{theorem}

For example, there are two fine gradings on $M_2(\mathbb{F})$, up to equivalence: the Cartan grading with universal group $\mathbb{Z}$ and the Pauli grading with universal group $\mathbb{Z}_2^2$. The above theorem tells us that they do not have proper almost fine coarsenings.

\begin{definition}[{\cite[Definition 4.1]{EK23}}]\label{def:admissible}
Let $\Delta$ be an almost fine abelian group grading on $\mathcal{A}$, $U=U_\mathrm{ab}(\Delta)$, and let $\pi_\Delta:U\to U/t(U)$ be the natural homomorphism. A group homomorphism $\alpha:U\to G$ is said to be \emph{admissible} if the restriction of the homomorphism $(\alpha,\pi_\Delta):U\to G\times U/t(U)$ to the support of $\Delta$ is injective.
\end{definition}

Then a list of isomorphism classes of $G$-gradings on $\mathcal{A}$ can be obtained from a list of almost fine gradings up to equivalence:

\begin{theorem}[{\cite[Theorem 4.3]{EK23}}]\label{thm:gradings}
Let $\{\Gamma_i\}_{i\in I}$ be a set of representatives of the equivalence classes of almost fine abelian group gradings on $\mathcal{A}$. Then, for any abelian group $G$ and a $G$-grading $\Gamma$ on $\mathcal{A}$, there exists a unique $i\in I$ such that $\Gamma$ is isomorphic to the induced grading $^\alpha\Gamma_i$, for some admissible homomorphism $\alpha:U_\mathrm{ab}(\Gamma_i)\to G$. Moreover, two such homomorphisms, $\alpha$ and $\alpha'$, induce isomorphic $G$-gradings if and only if there exists $w\in W(\Gamma_i)$ such that $\alpha=\alpha'\circ w$.
\end{theorem}

For example, in the case of the Cartan grading on $M_2(\mathbb{F})$, any homomorphism $\alpha:\mathbb{Z}\to G$ is admissible, and it induces the grading $\gr{\C}_{M_2}(g)$ with $g=\alpha(1)$. For the Pauli grading, a homomorphism $\alpha:\mathbb{Z}_2^2\to G$ is admissible if and only if it is injective; it induces the grading $\gr{\Pa}_{M_2}(a,b)$ with $a=\alpha(1,0)$ and $b=\alpha(0,1)$.

\subsection{Some preliminary results}\label{s:prelim}
Let $\mathcal{L}$ be a finite-dimensional Lie superalgebra such that restricting automorphisms of $\mathcal{L}$ to those of $\mathcal{L}_0$ gives the exact sequence
\[
1\to\langle\delta\rangle\to\aut(\mathcal{L})\to\aut(\mathcal{L}_0).
\]
For example, by \cite[Lemma 1]{S}, this holds if $\mathcal{L}_1$ is a simple $\mathcal{L}_0$-module. Given a grading $\Gamma$ on $\mathcal{L}$, we shall investigate the kernel and the image of the homomorphism $W(\Gamma)\to W(\Gamma_0)$ induced by the restriction map $\aut(\mathcal{L})\to\aut(\mathcal{L}_0)$. Letting
\[
K=\{\psi\in\aut(\Gamma)\mid \psi|_{\mathcal{L}_0}\in\mathrm{Stab}(\Gamma_0)\}
\]
and composing the restriction map $\aut(\Gamma)\to\aut(\Gamma_0)$ with the quotient map $\aut(\Gamma_0)\to W(\Gamma_0)$, we obtain the exact sequence
\[
1\to K\to\aut(\Gamma)\to W(\Gamma_0).
\]
Moreover, denoting by $\bar{K}$ the image of $K$ under the quotient map $\aut(\Gamma)\to W(\Gamma)$, we obtain the exact sequence
\[
1\to\bar{K}\to W(\Gamma)\to W(\Gamma_0).
\]

\begin{lemma}\label{Misha's lemma}
Let $\mathcal{L}=\mathcal{L}_0\oplus \mathcal{L}_1$ be a finite-dimensional superalgebra such that the kernel of the restriction map $\res:\aut(\mathcal{L})\to\aut(\mathcal{L}_0)$ is generated by the parity automorphism $\delta$. Let $\Gamma$ be a grading on $\mathcal{L}$ and $\Gamma_0$ its restriction to $\mathcal{L}_0$. Then the  restriction map induces a homomorphism $W(\Gamma)\rightarrow W(\Gamma_0)$, whose kernel $\bar{K}=\res^{-1}(\stab(\Gamma_0))/\stab(\Gamma)$ can be embedded in the subgroup $U_{[2]}$ of $U=U_\mathrm{ab}(\Gamma)$. Moreover, if $\mathrm{Supp}(\mathcal{L}_0)\cap\mathrm{Supp}(\mathcal{L}_1)\ne\emptyset$, then $\bar{K}$ is trivial, otherwise the parity homomorphism $p:U\to\mathbb{Z}_2$ is well defined and the image of the embedding $\bar{K}\hookrightarrow U_{[2]}$ acts on $U$ as follows:
  \[
  k.u=\begin{cases}%
  u,&\text{ if $p(u)=0$},\\%
  ku,&\text{ if $p(u)=1$}.
  \end{cases}
  \]
\end{lemma}
\begin{proof}
Let $\varphi\in K$, that is, $\varphi|_{\mathcal{L}_0}\in\stab(\Gamma_0)$. If $\eta\in\mathrm{Diag}(\Gamma)$, then $[\varphi,\eta]|_{\mathcal{L}_0}=[\varphi|_{\mathcal{L}_0},\eta|_{\mathcal{L}_0}]=1$, therefore $[\varphi,\eta]\in\langle\delta\rangle$. So the map $\chi_\varphi:\mathrm{Diag}(\Gamma)\to\langle\delta\rangle$, $\eta\mapsto [\varphi,\eta]$, is well defined. For any $\eta$, $\mu\in\mathrm{Diag}(\Gamma)$, we have 
\begin{align*}
\chi_\varphi(\eta\mu)&=[\varphi,\eta\mu]= \varphi\eta\mu\varphi^{-1}(\eta\mu)^{-1}= 
\varphi\eta\varphi^{-1}\eta^{-1}\eta\varphi\mu\varphi^{-1}\mu^{-1}\eta^{-1}\\
&=[\varphi,\eta]\eta[\varphi,\mu]\eta^{-1}=
\chi_\varphi(\eta)\chi_\varphi(\mu),
\end{align*}
so $\chi_\varphi$ is a group homomorphism. Since $\langle\delta\rangle\cong C_2$, we can regard $\chi_\varphi$ as a character of $\mathrm{Diag}(\Gamma)$. The group of characters of $\mathrm{Diag}(\Gamma)$ is isomorphic to $U$. Moreover, $\chi_\varphi^2=1$, that is, $\chi_\varphi\in U_{[2]}$. We claim that the map $K\to U_{[2]}$, $\varphi\mapsto\chi_\varphi$, is a group homomorphism. Indeed, for any $\varphi$, $\psi\in K$ and $\eta\in\mathrm{Diag}(\Gamma)$, we have 
\begin{align*}
\chi_{\varphi\psi}(\eta)&=[\varphi\psi,\eta]=\varphi\psi\eta\psi^{-1}\varphi^{-1}\eta^{-1}\\ 
&=\varphi\chi_{\psi}(\eta)\eta\varphi^{-1}\eta^{-1}=\chi_{\varphi}(\eta)\chi_{\psi}(\eta).
\end{align*}
Since the elements of $\mathrm{Stab}(\Gamma)$ commute with every element of $\mathrm{Diag}(\Gamma)$, this  homomorphism factors through $\bar{K}\to U_{[2]}$. Finally, suppose $\chi_\varphi=1$ for some $\varphi\in K$. This means that $[\varphi,\eta]=1$ for each $\eta\in\mathrm{Diag}(\Gamma)$, hence $\varphi\in\stab(\Gamma)$. Therefore, the homomorphism $\bar{K}\to U_{[2]}$ is injective.

Now, pick $\psi\in K$ and consider the corresponding element $k\in U_{[2]}$. By definition of $k$, for any $\eta\in\mathrm{Diag}(\Gamma)$, we have $\psi\eta\psi^{-1}=\varepsilon(\eta(k))\eta$, where $\varepsilon$ is the unique isomorphism $\{\pm1\}\to\langle\delta\rangle$. 
Then, for any nonzero homogeneous $x\in\mathcal{L}$, we have
\[
\eta(\deg x)\psi(x)=\psi\eta\psi^{-1}\psi(x)=\varepsilon(\eta(k))(\eta\circ\psi(x))=\begin{cases}%
  \eta(\deg\psi(x))\psi(x),&\text{if $x\in \mathcal{L}_0$},\\%
  \eta(k)\eta(\deg\psi(x))\psi(x),&\text{if $x\in \mathcal{L}_1$}.
  \end{cases}
\]
If $x\in \mathcal{L}_0$, we get $\eta(\deg\psi(x))=\eta(\deg x)$ for all $\eta\in\mathrm{Diag}(\Gamma)$, hence $\deg\psi(x)=\deg x$ (which also follows from the definition of $K$). If $x\in \mathcal{L}_1$, we get 
$\eta(\deg\psi(x))=\eta(k\deg x)$ for all $\eta\in\mathrm{Diag}(\Gamma)$, hence $\deg\psi(x)=k\deg x$.
\end{proof}

We note that, even if the restriction map $\aut(\mathcal{L})\to\aut(\mathcal{L}_0)$ is surjective, the induced homomorphism $W(\Gamma)\to W(\Gamma_0)$ is not necessarily so, because the restriction of $\mathrm{Diag}(\Gamma)$ to $\mathcal{L}_0$ may fail to be the entire $\mathrm{Diag}(\Gamma_0)$. We will consider a slightly more general situation. Let $Q\subseteq\aut(\mathcal{A})$ be a quasitorus and $\Gamma$ the grading on $\mathcal{A}$ defined by $Q$. Recall that $C(Q)=\mathrm{Stab}(\Gamma)$, and $N(Q)\subseteq\aut(\Gamma)$. We set
$$
W(Q):=N(Q)/C(Q)=\mathrm{Stab}_{W(\Gamma)}(Q).
$$

\begin{lemma}\label{l:weyl}
Let $\mathcal{A}$ be a finite-dimensional algebra, $\mathcal{A}_0$ a subalgebra invariant under the action of $\aut(\mathcal{A})$ such that the restriction map $\mathrm{res}:\aut(\mathcal{A})\to\aut(\mathcal{A}_0)$ is surjective. Let $\Gamma$ be an abelian group grading on $\mathcal{A}$ and $\Gamma_0$ its restriction to $\mathcal{A}_0$. Denote $Q=\mathrm{Diag}(\Gamma)$ and $\tilde{Q}=\mathrm{res}(Q)$. If the kernel of $\mathrm{res}$ is contained in $Q$, then the image of the homomorphism $W(\Gamma)\to W(\Gamma_0)$ induced by $\mathrm{res}$ is $W(\tilde{Q})$.
\end{lemma}
\begin{proof}
Since the homomorphism $\mathrm{res}:\aut(\mathcal{A})\to\aut(\mathcal{A}_0)$ is surjective, so is its restriction to the normalizers $N_{\aut(\mathcal{A})}(Q)\to N_{\aut(\mathcal{A}_0)}(\tilde{Q})$, because $Q$ is the inverse image of $\tilde{Q}$. Since $\Gamma$ is an abelian group grading and $Q=\mathrm{Diag}(\Gamma)$, we have $N_{\aut(\mathcal{A})}(Q)=\aut(\Gamma)$. The result follows.
\end{proof}

\begin{lemma}\label{l:ext_aut}
Under the hypotheses of \Cref{Misha's lemma}, assume that $\Gamma$ is an abelian group grading and $\mathrm{Supp}(\mathcal{L}_0)\cap\mathrm{Supp}(\mathcal{L}_1)=\emptyset$.
If $\psi_0\in\mathrm{Stab}(\Gamma_0)$ admits an extension $\psi\in\aut(\mathcal{L})$, then $\psi\in\aut(\Gamma)$.
\end{lemma}
\begin{proof}
By the assumption on the supports, $\delta\in\mathrm{Diag}(\Gamma)$. Also, for any $\eta\in\mathrm{Diag}(\Gamma)$, we have $[\psi,\eta]|_{\mathcal{L}_0}=[\psi_0,\eta|_{\mathcal{L}_0}]=1$. Thus $[\psi,\eta]\in\langle\delta\rangle$, which implies $\psi\eta\psi^{-1}\in\{\eta,\eta\delta\}\subseteq\mathrm{Diag}(\Gamma)$. Since it holds for all $\eta\in\mathrm{Diag}(\Gamma)$, we conclude that $\psi\in N(\mathrm{Diag}(\Gamma))=\mathrm{Aut}(\Gamma)$.
\end{proof}

\begin{lemma}\label{l:exact_seq}
Under the hypotheses of 
\Cref{l:ext_aut}, assume further that $\mathrm{res}:\aut(\mathcal{L})\to\aut(\mathcal{L}_0)$ is surjective. Then there exists an exact sequence
\[
1\to\langle\delta\rangle\to\mathrm{Stab}(\Gamma)\to\mathrm{Stab}(\Gamma_0)\to\bar{K}\to 1.
\]
\end{lemma}
\begin{proof}
The homomorphism $\mathrm{Stab}(\Gamma_0)\to\bar{K}$ is defined as follows. If $\psi_0\in\mathrm{Stab}(\Gamma_0)$, then, by hypothesis, there exists an extension $\psi\in\aut(\mathcal{L})$. By \Cref{l:ext_aut}, we have $\psi\in\aut(\Gamma)$ and, hence, $\bar{\psi}:=\psi\,\mathrm{Stab}(\Gamma)\in W(\Gamma)$. Two distinct extensions of $\psi_0$ differ by $\delta\in\mathrm{Stab}(\Gamma)$, so the map $\psi_0\mapsto \bar{\psi}$ is well defined. Moreover, the restriction of $\psi$ to $\mathcal{L}_0$ is $\psi_0\in\mathrm{Stab}(\Gamma_0)$, so $\psi\in K$ and,  hence, $\bar{\psi}\in\bar{K}$. This gives the desired homomorphism $\mathrm{Stab}(\Gamma_0)\to\bar{K}$, which is surjective by definition of $\bar{K}$.

The homomorphism $\mathrm{Stab}(\Gamma)\to\mathrm{Stab}(\Gamma_0)$ is given by restriction, so its kernel is $\langle\delta\rangle$. For any $\psi\in\mathrm{Stab}(\Gamma)$, we have $\bar{\psi}=1$, so $\psi_0=\psi|_{\mathcal{L}_0}$ is in the kernel of the homomorphism constructed above. Conversely, if $\psi_0\in\mathrm{Stab}(\Gamma_0)$ is in the kernel, then its extension $\psi$ satisfies $\psi\in\mathrm{Stab}(\Gamma)$, which means that $\psi_0$ is in the image of $\mathrm{Stab}(\Gamma)\to\mathrm{Stab}(\Gamma_0)$. 
\end{proof}

\section{Gradings on $G(3)$\label{sec:g3}}
\subsection{The Lie superalgebra $G(3)$}

Recall that $G(3)$ is the simple Lie superalgebra $\mathcal{L}$ with $\mathcal{L}_0=\mathfrak{sl}_2\oplus\mathfrak{g}_2$ and $\mathcal{L}_1=V\otimes\mathcal{C}^0$, where $V$ is the $2$-dimensional irreducible $\mathfrak{sl}_2$-module and $\mathcal{C}^0$ is the $7$-dimensional irreducible $\mathfrak{g}_2$-module (see below).

Recall also that $\mathfrak{g}_2\cong\mathrm{Der}(\mathcal{C})$, where $\mathcal{C}$ is the Cayley algebra. Moreover, $\mathrm{Ad}:\aut(\mathcal{C})\to\aut(\mathrm{Der}(\mathcal{C}))$ is an isomorphism of algebraic groups and, hence, $\mathcal{C}$ and $\mathfrak{g}_2$ have the same classification of (abelian) group gradings (see e.g. \cite[\S 4.4, 4.5]{EK13}). Explicitly, the bijection between gradings (respectively, fine gradings) on the two algebras is as follows. A $G$-grading $\Gamma_{\mathcal{C}}$ on the algebra $\mathcal{C}$ induces a $G$-grading on the associative algebra $\mathrm{End}(\mathcal{C})$, whose Lie subalgebra $\mathrm{Der}(\mathcal{C})$ is easily seen to be graded, so we denote the resulting $G$-grading on $\mathrm{Der}(\mathcal{C})$ by $\Gamma_{\mathfrak{g}_2}$. The correspondence $\Gamma_{\mathcal{C}}\mapsto\Gamma_{\mathfrak{g}_2}$ defines a bijection between the isomorphism classes of gradings  (respectively, equivalence classes of fine gradings) on the two algebras. Furthermore, $\mathfrak{g}_2\cong\mathrm{Der}(\mathcal{C})$ acts naturally on $\mathcal{C}$, and the subspace $\mathcal{C}^0$ of zero-trace elements is the $7$-dimensional irreducible $\mathfrak{g}_2$-module. For any $\Gamma_{\mathcal{C}}$, this subspace is graded. 

The simply connected algebraic group corresponding to $\mathcal{L}_0$ is $\mathrm{SL}_2(\mathbb{F})\times\mathrm{G}_2(\mathbb{F})$, which acts on $\mathcal{L}_0$, via $\mathrm{Ad}$ (given by conjugation in $M_2(\mathbb{F})$ and $\mathrm{End}(\mathcal{C})$), and also on $\mathcal{L}_1$, via the natural action on each factor. This gives a group isomorphism $\mathrm{SL}_2(\mathbb{F})\times\mathrm{G}_2(\mathbb{F})\to\aut(G(3))$ (see \cite{S} or \cite[\S 6.8.1]{GP}). 

\subsection{Group gradings on $G(3)$}
Let $G$ be an abelian group. For any $G$-grading on $\mathcal{L}$, the simple ideals $\mathfrak{sl}_2$ and $\mathfrak{g}_2$ of $\mathcal{L}_0$ are graded. Moreover, both the $G$-grading of $\mathfrak{sl}_2$ and the action of $\mathfrak{sl}_2$ on $\mathcal{L}_1$ come by restriction from $\mathcal{Q}=M_2(\mathbb{F})$, so $\mathcal{L}_1$ is a graded $\mathcal{Q}$-module. As mentioned above, any grading on $\mathfrak{g}_2$ is induced by a grading on its module $\mathcal{C}^0$. An important observation made in \cite{DEM} is that, in this case, the grading on $\mathfrak{sl}_2$ is also induced from its module $V$. Indeed, a Pauli grading on $\mathcal{Q}$ would make it a \emph{graded-division algebra} (i.e., its nonzero homogeneous elements would be invertible), and this would force $\mathcal{L}_1$ to be a free $\mathcal{Q}$-module, which is not the case. For gradings on $V$ and $\mathcal{C}$, we will use the notation introduced in \Cref{s:examplesgr}. The tensor product of gradings is defined by setting $\deg(v\otimes w)=\deg(v)\deg(w)$.

\begin{definition}\label{def:gradG3}
 Let $g\in G$ and $\Gamma_{\mathcal{C}}$ be a $G$-grading on $\mathcal{C}$. Let $\Gamma_{\mathfrak{g}_2}$ and $\Gamma_{\mathcal{C}^0}$ be the induced gradings on $\mathfrak{g}_2$ and $\mathcal{C}^0$, respectively. We write $\Gamma_{G(3)}(g,\Gamma_{\mathcal{C}})$ to denote the $G$-grading on $G(3)$ such that
\begin{eqnarray*}
\mathcal{L}_0&=&(\mathfrak{sl}_2,\Gamma_{\mathfrak{sl}_2}(g,g^{-1}))\oplus(\mathfrak{g}_2,\Gamma_{\mathfrak{g}_2}),\\%
\mathcal{L}_1&=&(V\otimes\mathcal{C}^0,\Gamma_V(g,g^{-1})\otimes\Gamma_{\mathcal{C}^0}).
\end{eqnarray*}
If $\Gamma_\mathcal{C}=\gr{\Pa}_\mathcal{C}(T)$ or $\Gamma_\mathcal{C}=\gr{\C}_{\mathcal{C}}(h_1,h_2,h_3)$ (following \Cref{def:cayleyI,def:cayleyII}), then we shall denote the respective gradings on $G(3)$ by $\gr{\Pa}_{G(3)}(g,T)$ or $\gr{\C}_{G(3)}(g,h_1,h_2,h_3)$.
\end{definition}

We start by computing the Weyl group of the fine gradings on $G(3)$. According to \cite[Theorem 5.1]{DEM}, we have two fine gradings, up to equivalence: the Cartan grading, of type $(28,0,1)$, which in our notation is the $\mathbb{Z}^3$-grading $\gr{\C}_{G(3)} (b_0,b_1,b_2,b_3)$, where
\[
\mathbb{Z}^3\cong\langle b_0,b_1,b_2,b_3\mid b_1+b_2+b_3=0\rangle,
\]
and the $\mathbb{Z}\times\mathbb{Z}_2^3$-grading $\gr{\Pa}_{G(3)}((1,0),\{0\}\times\mathbb{Z}_2^3)$, of type $(17,7)$. Clearly, these two gradings induce, respectively, $\gr{\C}_{G(3)}(g,h_1,h_2,h_3)$ and $\gr{\Pa}_{G(3)}(g,T)$, via the homomorphisms $\mathbb{Z}^3\to G$, $b_0\mapsto g$, $b_i\mapsto h_i$ for $i=1,2,3$, and $\mathbb{Z}\times\mathbb{Z}_2^3\to G$, $(1,0)\mapsto g$, $\mathbb{Z}_2^3\overset{\sim}{\rightarrow}T$.

\begin{proposition}\label{prop:WeylG}
Let $\Gamma$ be a fine grading on $G(3)$. Then $W(\Gamma)\cong W(\Gamma_0)$:
\begin{itemize}
\item If $U(\Gamma)\cong\mathbb{Z}\times\mathbb{Z}^2$, then $W(\Gamma)\cong C_2^2\times S_3$. In matrix form,
\[
W(\Gamma)=\left\{\left(\begin{array}{cc}\varepsilon&0\\0&g\end{array}\right)\mid\varepsilon\in\{\pm1\},\, g\in C_2\times S_3\right\},
\]
where each matrix multiplies on the left the elements of $\mathbb{Z}\times\mathbb{Z}^2$ (written as columns), with $C_2$ acting on $\mathbb{Z}^2\cong\langle b_1,b_2,b_3\mid b_1+b_2+b_3=0\rangle$ by inversion and $S_3$ by permuting $b_1$, $b_2$ and $b_3$.
\item If $U(\Gamma)=\mathbb{Z}\times\mathbb{Z}_2^3$, then $W(\Gamma)\cong C_2\times\mathrm{GL}_3(2)$. In matrix form,
\[
W(\Gamma)=\left\{\left(\begin{array}{cc}\varepsilon&0\\0&A\end{array}\right)\mid\varepsilon\in\{\pm1\},\, A\in\mathrm{GL}_3(2)\right\},
\]
where each matrix multiplies the elements of $\mathbb{Z}\times\mathbb{Z}_2^3$ on the left.
\end{itemize}
\end{proposition}
\begin{proof}
We are going to use the results and notation from \Cref{s:prelim}.
First, we shall prove that $\bar{K}$ is trivial. Write $U=\mathbb{Z}\times U_{\mathfrak{g}_2}$, where $U_{\mathfrak{g}_2}$ is either $\mathbb{Z}^2$ or $\mathbb{Z}_2^3$. Note that, in both cases, the restriction $\Gamma_0$ is fine as well. Thus, \Cref{l:DiagStab} implies that
$$
\mathrm{Stab}(\Gamma)=\mathrm{Diag}(\Gamma)\cong\mathbb{F}^\times\times\widehat{U_{\mathfrak{g}_2}}\cong\mathrm{Diag}(\Gamma_0)=\mathrm{Stab}(\Gamma_0).
$$
The restriction map $\mathrm{Stab}(\Gamma)\to\mathrm{Stab}(\Gamma_0)$ is $(x,u)\mapsto(x^2,u)$. Since $\mathbb{F}$ is algebraically closed, this map is surjective. From \Cref{l:exact_seq}, we have the exact sequence
$$
\mathrm{Stab}(\Gamma)\to\mathrm{Stab}(\Gamma_0)\to\bar{K}\to1.
$$
Since the first map is surjective, we obtain $\bar{K}=1$. Next, applying \Cref{l:weyl}, we obtain $W(\Gamma)\cong W(\Gamma_0)$.

We have $W(\Gamma_0)\cong W(\Gamma_{\mathfrak{sl}_2})\times W(\Gamma_{\mathfrak{g}_2})$,  $W(\Gamma_{\mathfrak{sl}_2})\cong C_2$, and $W(\Gamma_{\mathfrak{g}_2})\cong W(\Gamma_{\mathcal{C}})$ if $\Gamma_{\mathfrak{g}_2}$ is induced by $\Gamma_{\mathcal{C}}$. The Weyl groups $W(\Gamma_{\mathcal{C}})$ were computed in \cite{EK_Weyl} (see also \cite{EK13}). In the case of the Cartan grading, $W(\Gamma_{\mathfrak{g}_2})\cong C_2\times S_3$ (the automorphism group of the root system of type $G_2$) by \cite[Theorem 4.17]{EK13}. In the case of the grading induced by the Cayley-Dickson process, $W(\Gamma_{\mathfrak{g}_2})\cong\mathrm{GL}_3(2)$ by \cite[Theorem 4.19]{EK13}.
\end{proof}

Now we can obtain a classification of group gradings on $G(3)$. First, note that, by \Cref{thm:almostfine}, neither fine grading on $G(3)$ admits almost fine proper coarsenings. This is clear for the Cartan grading, since in this case $U(\Gamma)$ has trivial torsion. As for the Cayley-Dickson grading, the algebra of even derivations of $\mathcal{L}$ is $\mathrm{ad}(\mathcal{L}_0)\cong\mathcal{L}_0$, so its support contains all nontrivial torsion elements of $U(\Gamma)$, because they belong to $\mathrm{Supp}(\Gamma_{\mathfrak{g}_2})$, whereas $e$ belongs to $\mathrm{Supp}(\Gamma_{\mathfrak{sl}_2})$. This fact also implies, in view of \Cref{def:admissible}, that every admissible homomorphism $U(\Gamma)\to G$ must be injective on the torsion subgroup $\mathbb{Z}_2^3$. Hence, \Cref{thm:gradings} and \Cref{prop:WeylG} give us the classification of the isomorphism classes of group gradings on $G(3)$: 

\begin{corollary}\label{classG3}
Let $\Gamma$ be a $G$-grading on $G(3)$. Then $\Gamma$ is isomorphic to $\gr{\Pa}_{G(3)}(g,T)$ or $\gr{\C}_{G(3)}(g,h_1,h_2,h_3)$ of \Cref{def:gradG3}. Moreover, gradings from different families are not isomorphic to each other, and we have:
\begin{itemize}
\item $\gr{\Pa}_{G(3)}(g,T)\cong\gr{\Pa}_{G(3)}(g',T')$ if and only if $g'\in\{g,g^{-1}\}$ and $T=T'$,
\item $\gr{\C}_{G(3)}(g,h_1,h_2,h_3)\cong\gr{\C}_{G(3)}(g',h_1',h_2',h_3')$ if and only if $g'\in\{g,g^{-1}\}$ and there exists $\sigma\in S_3$ such that either $h_1'=h_{\sigma(1)}$, $h_2'=h_{\sigma(2)}$, $h_3'=h_{\sigma(3)}$, or $h_1'=h_{\sigma(1)}^{-1}$, $h_2'=h_{\sigma(2)}^{-1}$, $h_3'=h_{\sigma(3)}^{-1}$.\qed
\end{itemize}
\end{corollary}

To state the isomorphism conditions more concisely, it is convenient to introduce the \emph{multiset} $\{h_1,h_2,h_3\}$ and define its \emph{inverse} as $\{h_1^{-1},h_2^{-1},h_3^{-1}\}$. Then the $G$-gradings are classified up to isomorphism by the following parameters: a subgroup of $G$ of a certain isomorphism type (possibly trivial) and a certain configuration of elements, considered up to an equivalence relation. Table \ref{tab:G(3)} summarizes our result for $G(3)$, and we will see this pattern repeated in the following sections.

\begin{table}
\begin{tabular}{|c|c|c|c|}
\hline
Family & Subgroups & Combinatorial data & Equivalence\\ \hline 
$\gr{\C}_{G(3)}$ 
 & trivial & \vtop{\hbox{\strut $g\in G$ and multiset}\hbox{\strut $\{h_1,h_2,h_3\}$ in $G$}\hbox{\strut with $h_1h_2h_3=e$}}  & \vtop{\hbox{\strut independent}\hbox{\strut inversion of $g$,}\hbox{\strut $\{h_1,h_2,h_3\}$}}\\ \hline 
$\gr{\Pa}_{G(3)}$ 
 & $T\cong\mathbb{Z}_2^3$ & $g\in G$ & inversion\\
\hline 
\end{tabular}
\vspace{5pt} 
\caption{Classification of $G$-gradings on the Lie superalgebra $G(3)$}\label{tab:G(3)}
\end{table}

\section{Gradings on $F(4)$\label{sec:f4}}

\subsection{The Lie superalgebra $F(4)$ and its gradings}
Recall that $F(4)$ is the simple Lie superalgebra with
$\mathcal{L}_0=\mathfrak{sl}_2\oplus\mathfrak{so}_7$ and 
$\mathcal{L}_1=V\otimes W$,
where $V$ is the $2$-dimensional irreducible $\mathfrak{sl}_2$-module and $W$ is the $8$-dimensional irreducible spin $\mathfrak{so}_7$-module.

There is a surjective group homomorphism $\mathrm{SL}_2(\mathbb{F})\times\mathrm{Spin}_7(\mathbb{F})\to\aut(F(4))$ (see \cite{S} or \cite[\S 6.8.2]{GP}), via the adjoint action on $\mathcal{L}_0$ and the natural action on each factor of $\mathcal{L}_1$. The kernel of this homomorphism is generated by the element $(-1,-1)$.

For any $G$-grading on $\mathcal{L}$, the simple ideals $\mathfrak{sl}_2$ and $\mathfrak{so}_7$ of $\mathcal{L}_0$ are graded, the $G$-grading on $\mathfrak{sl}_2$ comes from $\mathcal{Q}=M_2(\mathbb{F})$, and $\mathcal{L}_1$ becomes a graded $\mathcal{Q}$-module. But, unlike in the case of $G(3)$, the grading on $\mathcal{Q}$ need not be elementary.

According to \cite[Theorem 4.1]{DEM}, there exist five fine gradings on $F(4)$, up to equivalence: the $\mathbb{Z}^4$-grading (Cartan grading), a $\mathbb{Z}\times\mathbb{Z}_2^3$-grading, $\mathbb{Z}^2\times\mathbb{Z}_2^2$ and $\mathbb{Z}\times\mathbb{Z}_2^3$-gradings coming from the Tits-Kantor-Koecher (TKK) construction, and a $\mathbb{Z}_2^3\times\mathbb{Z}_4$-grading. Their types are, respectively, $(36,0,0,1)$, $(19,0,7)$, $(32,4)$, $(31,0,3)$ and $(24,6,0,1)$.
To compute their Weyl groups, we will use the following specialization of \Cref{l:weyl} and \Cref{l:exact_seq} to $F(4)$, combined with \Cref{l:DiagStab}:

\begin{lemma}\label{l:seq_f(4)}
Let $\Gamma$ be a fine grading on $F(4)$ and $Q=\mathrm{Diag}(\Gamma)$. Denote by $\Gamma_0$ and $\tilde{Q}$ their restrictions to $\mathcal{L}_0$, and by $\Gamma_{\mathfrak{sl}_2}$ and $\Gamma_{\mathfrak{so}_7}$ the restrictions of $\Gamma_0$ to the simple ideals of $\mathcal{L}_0$.
Then we have the exact sequences
\begin{gather}
1\to\bar{K}\to W(\Gamma)\to W(\tilde{Q})\to 1,\label{seq_f(4)_1} \\
1\to\langle\delta\rangle\to\mathrm{Stab}(\Gamma)\to\mathrm{Stab}(\Gamma_0)\to\bar{K}\to 1.\label{seq_f(4)_2}
\end{gather}
Moreover, $\mathrm{Stab}(\Gamma)=\mathrm{Diag}(\Gamma)$ and $\mathrm{Stab}(\Gamma_0)\cong\mathrm{Stab}(\Gamma_{\mathfrak{sl}_2})\times\mathrm{Stab}(\Gamma_{\mathfrak{so}_7})$.\qed
\end{lemma}

We shall see that, for every fine grading on $F(4)$, we have $\bar{K}\cong U_{[2]}$, and the first exact sequence is always split. 
The Weyl groups are computed in \Cref{prop:WeylI,prop:WeylII,prop:WeylIII,prop:WeylIV,prop:WeylV}. There are no almost fine gradings apart from the fine ones, hence the constructions given in \Cref{def:I,def:II,def:III,def:IV,def:VF(4)} provide all group gradings on $F(4)$. Moreover, the isomorphism classes in each of these five families are described in \Cref{cor:I,cor:II,cor:III,cor:IV,cor:V}.
For convenience, we summarize the resulting classification here:
\begin{theorem}\label{thm:F4}
Let $G$ be an abelian group and $\Gamma$ a $G$-grading on $F(4)$. Then $\Gamma$ is isomorphic to $\gr{\FI}_{F(4)}(g,h_1,h_2,h_3,h_4)$, $\gr{\FII}_{F(4)}(g,T)$, $\gr{\FIII}_{F(4)}(g_1,g_2,a,b)$, $\gr{\FIV}_{F(4)}(g,a,b,h)$, or $\gr{\FV}_{F(4)}(g,a,b,h)$. Moreover, gradings from different families are not isomorphic, and
the isomorphism classes within each family are parametrized by certain subgroups of $G$ and equivalence classes of certain multisets as indicated in Table \ref{tab:F(4)}.


\end{theorem}

\begin{table}
\begin{tabular}{|c|c|c|c|}
\hline
Family & Subgroups & Combinatorial data & Equivalence\\ \hline 
$\gr{\FI}_{F(4)}$ 
 & trivial & \vtop{\hbox{\strut $g\in G$ and multiset}\hbox{\strut $\{h_1,\ldots,h_4\}$ in $G$}\hbox{\strut with $h_1h_2h_3h_4=e$}}  & \vtop{\hbox{\strut independent}\hbox{\strut inversion of $g$,}\hbox{\strut $\{h_1,\ldots,h_4\}$}}\\ \hline 
$\gr{\FII}_{F(4)}$ 
 & $T\cong\mathbb{Z}_2^3$ & $gT\in G/T$ & inversion\\ \hline 
$\gr{\FIII}_{F(4)}$ 
 & $T=\langle a,b\rangle\cong\mathbb{Z}_2^2$ & $\{g_1^{\pm 1},g_2^{\pm 1}\}$ in $G$ & $T$-translation\\ \hline 
$\gr{\FIV}_{F(4)}$ 
 & \vtop{\hbox{\strut $H=\langle a,b,h\rangle\cong\mathbb{Z}_2^3$,}\hbox{\strut $T=\langle a,b\rangle\leq H$}} & $gH\in G/H$ & inversion\\ \hline
$\gr{\FV}_{F(4)}$ 
 & \vtop{\hbox{\strut $H=\langle g\rangle\times\langle a,b,h\rangle\cong\mathbb{Z}_4\times\mathbb{Z}_2^3$,}\hbox{\strut $S=\langle ah,b,g^2\rangle$, $T=\langle h,g^2\rangle\leq H$}} & & \\ \hline
\end{tabular}
\vspace{5pt} 
\caption{Classification of $G$-gradings on the Lie superalgebra $F(4)$}\label{tab:F(4)}
\end{table}

Before considering each of the five families, we make some general remarks. First, we recall gradings on $\mathfrak{so}_n$ for an odd $n\in\mathbb{N}$ (see e.g. \cite[\S 3.4]{EK13}).
\begin{definition}\label{def:typeB}
Let $G$ be an abelian group, and consider a sequence
\[
\gamma=(g_1,\ldots,g_q;\,g_1',g_1'',\ldots,g_s',g_s'')
\]
of elements of $G$ satisfying
\[
g_1^2=\cdots=g_q^2=g_1'g_1''=\cdots=g_s'g_s'',\quad q+2s=n.
\]
This $n$-tuple determines the elementary $G$-grading $\Gamma_{M_n}(\gamma)$ on $M_n(\mathbb{F})$, whose components are invariant under the involution $\varphi(X)=\Phi^{-1}X^T\Phi$, where
\[
\Phi=\mathrm{id}_q\oplus\sum_{i=1}^s\begin{pmatrix}0&1\\1&0\end{pmatrix}.
\]
In other words, $\varphi$ is the adjunction with respect to the bilinear form on $\mathbb{F}^n$ represented by the matrix $\Phi$. Thus, we obtain a $G$-grading $\Gamma_{\mathfrak{so}_n}(\gamma)$ on $\mathfrak{so}_n$ if we put
\[
\left(\mathfrak{so}_n\right)_g=M_n(\mathbb{F})_g\cap\mathcal{K}(M_n,\varphi),\quad g\in G,
\]
where $\mathcal{K}(M_n,\varphi):=\{X\in M_n(\mathbb{F})\mid\varphi(X)=-X\}\cong\mathfrak{so}_n$.
In particular, if we take the entries of $\gamma$ as group generators subject only to the relations stated above, then we obtain a fine grading on $\mathfrak{so}_n$ with universal group $\mathbb{Z}_2^{q-1}\times\mathbb{Z}^s$.
\end{definition}

Now, if we have a $G$-grading on $F(4)$, the action of the simple ideal $\mathfrak{so}_7$ of $\mathcal{L}_0$ on $\mathcal{L}_1$ gives a homomorphism of graded Lie algebras $\mathfrak{so}_7\hookrightarrow\mathrm{End}_{\mathcal{Q}}(\mathcal{L}_1)$. To find the septuple $\gamma$, we need to identify a $G$-graded vector space $\mathcal{U}$ with a homogeneous symmetric bilinear form so that $\mathrm{End}_{\mathcal{Q}}(\mathcal{L}_1)$ becomes isomorphic to the even part of the Clifford algebra of $\mathcal{U}$. For the first two families, the gradings on $\mathcal{Q}=\mathrm{End}(V)$ and $\mathrm{End}_{\mathcal{Q}}(\mathcal{L}_1)\cong\mathrm{End}(W)$ are elementary, and it is convenient to take for $W$ the underlying space of the Cayley algebra $\mathcal{C}$ to describe a grading on $W$ which induces the one on the algebra $\mathrm{End}(W)$, which in this case is isomorphic to the even Clifford algebra of $\mathcal{C}^0$. For the remaining three families, the gradings are not elementary, so we will use different models.

\subsection{Gradings from Cayley algebra}
It is well known that every element $x\in\mathcal{C}$ satisfies the equation
\[
x^2-n(x,1)x+n(x)1=0,
\]
where $n:\mathcal{C}\to\mathbb{F}$ is a nondegenerate quadratic form, called the \emph{norm}, and $n(x,y):=n(x+y)-n(x)-n(y)$ is its polarization. In particular, $n(x,1)$ is the \emph{trace} of $x$, hence any $x\in\mathcal{C}^0$ satisfies $x^2=-n(x)1$. Since $\mathcal{C}$ is an alternative algebra, the operators of left multiplication satisfy the identity $\ell_x^2=\ell_{x^2}$ and, in particular, $\ell_x^2=-n(x)\mathrm{id}$ for any $x\in\mathcal{C}^0$.
It follows that the mapping $x\mapsto\ell_x$ defines a homomorphism from the Clifford algebra $\mathfrak{Cl}(\mathcal{C}^0,-n)$ to $\mathrm{End}(\mathcal{C})$, which restricts to an isomorphism on the even part of the Clifford algebra, $\mathfrak{Cl}_0(\mathcal{C}^0,-n)\to\mathrm{End}(\mathcal{C})$. This allows us to identify the spin $\mathfrak{so}_7$-module $W$ with $\mathcal{C}$, by letting $\mathfrak{so}(\mathcal{C}^0,-n)$ act via the composition of the standard embedding $\mathfrak{so}(\mathcal{C}^0,-n)\to\mathfrak{Cl}_0(\mathcal{C}^0,-n)$, sending $yn(x,\cdot)-xn(y,\cdot)\mapsto\frac12[x,y]$, and the above isomorphism.

\subsubsection{The Cartan grading}
The first family of group gradings are coarsenings of the Cartan grading, which has universal group $\mathbb{Z}^4$. Since this group has trivial torsion, it follows that there is no almost fine proper coarsening, and every homomorphism from the universal group is admissible. 

Let $\{e_1,e_2,u_1,u_2,u_3,v_1,v_2,v_3\}$ be a `good basis' of $\mathcal{C}$. Then the Cartan grading on $\mathcal{L}$ can be obtained as the eigenspace decomposition defined by the adjoint action of the Cartan subalgebra of $\mathcal{L}_0$ spanned by the standard Cartan subalgebra of $\mathfrak{sl}_2$ and the elements $d_i:=u_i n(v_i,\cdot)-v_i n(u_i,\cdot)$ of $\mathfrak{so}(\mathcal{C}^0,n)=\mathfrak{so}(\mathcal{C}^0,-n)$, $i=1,2,3$, which are represented by diagonal matrices with respect to the basis $\{u_1,u_2,u_3,v_1,v_2,v_3,e_1-e_2\}$, namely, $1$ in position $i$, $-1$ in position $3+i$, and $0$ elsewhere. The element $d_i$ acts by $-\frac12[\ell_{u_i},\ell_{v_i}]$ on the spin module $\mathcal{C}$, and a straightforward computation shows that the elements of the `good basis' are eigenvectors: $e_1$ of weight $\frac12(\varepsilon_1+\varepsilon_2+\varepsilon_3)$, $e_2$ of weight $-\frac12(\varepsilon_1+\varepsilon_2+\varepsilon_3)$, $u_1$ of weight $\frac12(\varepsilon_1-\varepsilon_2-\varepsilon_3)$, $v_1$ of weight $\frac12(-\varepsilon_1+\varepsilon_2+\varepsilon_3)$, etc., where $\{\varepsilon_1,\varepsilon_2,\varepsilon_3\}$ is the dual basis of $\{d_1,d_2,d_3\}$. Note that these weights are the vertices of a cube on which the automorphism group $C_2^3\rtimes S_3$ of the root system of type $B_3$ acts as the full group of isometries of the cube. Hence, it is convenient to present
\[
\mathbb{Z}^4\cong\mathbb{Z}\times\langle b_1,b_2,b_3,b_4\mid b_1+b_2+b_3+b_4=0\rangle,
\]
which leads to the following:
\begin{definition}\label{def:I}
Given $g\in G$ and a quadruple $(h_1,h_2,h_3,h_4)$ of elements of $G$ satisfying $h_1h_2h_3h_4=e$, we define a $G$-grading $\Gamma_W(h_1,h_2,h_3,h_4)$ on the vector space $W=\mathcal{C}$ by assigning degrees as follows:
\begin{align*}
&\deg e_1=h_4,\,\deg e_2=h_4^{-1},\\
&\deg u_i=h_i,\,\deg v_i=h_i^{-1}\quad (i=1,2,3).
\end{align*}
Taking $\gamma=(e;\,h_1h_4,h_1^{-1}h_4^{-1},h_2h_4,h_2^{-1}h_4^{-1},h_3h_4,h_3^{-1}h_4^{-1})$ in \Cref{def:typeB}, we obtain a $G$-grading $\Gamma_{\mathfrak{so}_7}(\gamma)$ on the algebra $\mathfrak{so}(\mathcal{C}^0,n)\cong\mathfrak{so}_7$, relative to the basis 
\[
\{e_1-e_2,u_1,v_1,u_2,v_2,u_3,v_3\}.
\]
Using these two, we define the grading $\gr{\FI}_{F(4)}(g,h_1,h_2,h_3,h_4)$ on $F(4)$ as follows:
\begin{eqnarray*}
\mathcal{L}_0&=&(\mathfrak{sl}_2,\Gamma_{\mathfrak{sl}_2}(g,g^{-1}))\oplus(\mathfrak{so}_7,\Gamma_{\mathfrak{so}_7}(\gamma)),\\%
\mathcal{L}_1&=&(V\otimes W,\Gamma_V(g,g^{-1})\otimes\Gamma_W(h_1,h_2,h_3,h_4)).
\end{eqnarray*}
\end{definition}

\begin{proposition}\label{prop:WeylI}
Let $\Gamma$ be the fine grading on $F(4)$ with universal group $\mathbb{Z}^4$ (the Cartan grading). Then $W(\Gamma)\cong C_2^2\times S_4$. In matrix form,
\[
W(\Gamma)=\left\{\left(\begin{array}{cc}\varepsilon&0\\0&g\end{array}\right)\mid\varepsilon\in\{\pm1\}, g\in C_2\times S_4\right\},
\]
where each matrix multiplies on the left the elements of $\mathbb{Z}\times\mathbb{Z}^3$, and $C_2\times S_4$ acts on $\mathbb{Z}^3\cong\langle b_1,b_2,b_3,b_4\mid b_1+b_2+b_3+b_4=0\rangle$ similarly to \Cref{prop:WeylG}.
\end{proposition}
\begin{proof}
Using the notation of \Cref{l:seq_f(4)}, we have $\bar{K}=1$ by \Cref{Misha's lemma} and, hence, $W(\Gamma)\cong W(\tilde{Q})$ and the restriction map $\mathrm{Stab}(\Gamma)\to\mathrm{Stab}(\Gamma_0)$ is surjective. In fact, the restriction $\Gamma_0$ of $\Gamma$ is the Cartan grading on $\mathcal{L}_0$, so we have $\mathrm{Stab}(\Gamma)=\mathrm{Diag}(\Gamma)\cong(\mathbb{F}^{\times})^4\cong\mathrm{Diag}(\Gamma_0)=\mathrm{Stab}(\Gamma_0)$.  Thus, $\tilde{Q}=\mathrm{Diag}(\Gamma_0)$, so we get $W(\Gamma)\cong W(\Gamma_0)$. Finally, $W(\Gamma_0)\cong W(\Gamma_{\mathfrak{sl}_2})\times W(\Gamma_{\mathfrak{so}_7})\cong C_2\times(C_2\times S_4)$. 
\end{proof}

As a consequence, we obtain the isomorphism classes of group gradings in the family given by \Cref{def:I}.
\begin{corollary}\label{cor:I}
$\gr{\FI}_{F(4)}(g,h_1,h_2,h_3)\cong\gr{\FI}_{F(4)}(g',h_1',h_2',h_3')$ if and only if $g'\in\{g,g^{-1}\}$ and there is $\sigma\in S_4$ such that either $h_i'=h_{\sigma(i)}$ for all $i$, or $h_i'=h_{\sigma(i)}^{-1}$ for all $i$.
\qed
\end{corollary}

\subsubsection{The grading from Cayley-Dickson doubling process}

We will again realize the spin module $W$ on the vector space $\mathcal{C}$ via the algebra homomorphism $\mathfrak{Cl}(\mathcal{C}^0,-n)\to\mathrm{End}(\mathcal{C})$ defined by the mapping $x\mapsto\ell_x$ for all $x\in\mathcal{C}^0$. One checks (see \cite[\S 2.3]{DEM}) that this is a homomorphism of algebras with involution if we take the bar involution on the Clifford algebra and the adjunction with respect to $n$ on $\mathrm{End}(\mathcal{C})$. Now, any algebra automorphism $\psi$ of $\mathcal{C}$ preserves $n$ and satisfies $\psi\ell_x\psi^{-1}=\ell_{\psi(x)}$. It follows that the corresponding element of $\mathfrak{Cl}_0(\mathcal{C}^0,-n)$ belongs to the spin group. Thus, we obtain an embedding $\aut(\mathcal{C})\to\mathrm{Spin}(\mathcal{C}^0,-n)$ that, together with the restriction map $\aut(\mathcal{C})\to\mathrm{SO}(\mathcal{C}^0,-n)$, fits the following commutative diagram:
\[
\begin{tikzcd}
&\mathrm{Spin}(\mathcal{C}^0,-n)\arrow[d, two heads]\\%
\mathrm{Aut}(\mathcal{C}) \arrow[ru,hook]\arrow[r,hook]&\mathrm{SO}(\mathcal{C}^0,-n)
\end{tikzcd}
\]
Therefore, if $\Gamma_\mathcal{C}$ is a $G$-grading on the algebra $\mathcal{C}$ and $\gr{\mathcal{C}}_{\mathfrak{so}_7}$ is the grading of $\mathfrak{so}(\mathcal{C}^0,n)$ induced by the restriction of $\Gamma_{\mathcal{C}}$ to $\mathcal{C}^0$, then $W=\mathcal{C}$ becomes a graded $\mathfrak{so}_7$-module. 

\begin{definition}\label{def:II}
Given $g\in G$ and a subgroup $T\subseteq G$ isomorphic to $\mathbb{Z}_2^3$, consider a fine $G$-grading $\gr{\Pa}_\mathcal{C}(T)$ with support $T$ on $\mathcal{C}$ obtained from the Cayley-Dickson process (defined up to isomorphism, see \Cref{def:cayleyII}). Let $\gr{\mathcal{C}}_{\mathfrak{so}_7}(T)$ be the induced grading on $\mathfrak{so}(\mathcal{C}^0,n)$. Using these two, we define the grading $\gr{\FII}_{F(4)}(g,T)$ on $F(4)$ as follows:
\begin{eqnarray*}
\mathcal{L}_0&=&(\mathfrak{sl}_2,\Gamma_{\mathfrak{sl}_2}(g,g^{-1}))\oplus(\mathfrak{so}_7,\gr{\mathcal{C}}_{\mathfrak{so}_7}(T))\\%
\mathcal{L}_1&=&(V\otimes\mathcal{C},\Gamma_V(g,g^{-1})\otimes\gr{\mathrm{P}}_\mathcal{C}(T)).
\end{eqnarray*}
\end{definition}

From \cite[\S 3]{DEM}, the grading $\Gamma=\gr{\FII}_{F(4)}((1,0),\mathbb{Z}_2^3)$ is fine with universal group $U=\mathbb{Z}\times\mathbb{Z}_2^3$. Moreover, all torsion elements of $U$ belong to the support of $\mathrm{Der}(\mathcal{L})_0\cong\mathcal{L}_0$. It follows that there is no almost fine proper coarsening and also that every admissible homomorphism $U\to G$ must be injective on the torsion subgroup.  
To compute $W(\Gamma)$, consider the exact sequence \eqref{seq_f(4)_2} of \Cref{l:seq_f(4)}:
\[
1\to\langle\delta\rangle\to\mathrm{Stab}(\Gamma)\to\mathrm{Stab}(\Gamma_0)\to\bar{K}\to 1.
\]
Since $\Gamma$ is fine, we have $\mathrm{Stab}(\Gamma)=\mathrm{Diag}(\Gamma)\cong\mathbb{F}^\times\times C_2^3$. However, $\Gamma_0$ is not fine, because the $\mathbb{Z}_2^3$-grading $\Gamma_{\mathfrak{so}_7}$ admits a fine refinement with universal group  $\mathbb{Z}_2^6$ (take $q=7$, $s=0$ in \Cref{def:typeB}). Hence, $\mathrm{Stab}(\Gamma_{\mathfrak{so}_7})$ contains the maximal quasitorus corresponding to the fine refinement: $C_2^6\subseteq\mathrm{Stab}(\Gamma_{\mathfrak{so}_7})$. 
In view of the exact sequence above, $\bar{K}$ contains the cokernel of the homomorphism 
$
\mathbb{F}^\times\times C_2^3\to\mathbb{F}^\times\times C_2^6
$
and, hence, $|\bar{K}|\ge 2^3$. Then, \Cref{Misha's lemma} implies that $\bar{K}\cong U_{[2]}$.

\begin{proposition}\label{prop:WeylII}
Let $\Gamma$ be the fine grading on $F(4)$ with universal group $U=\mathbb{Z}\times\mathbb{Z}_2^3$ induced from the $\mathbb{Z}_2^3$-grading $\Gamma_{\mathcal{C}}$ on the Cayley algebra. Then
\[
W(\Gamma)=\aut(U)\cong C_2\times\mathrm{AGL}_3(2)\cong C_2\times\left(U_{[2]}\rtimes\mathrm{GL}_3(2)\right).
\]
In matrix form,
\[
W(\Gamma)=\left\{\left(\begin{array}{cc}\varepsilon&0\\k&A\end{array}\right)\mid\varepsilon\in\{\pm1\},A\in\mathrm{GL}_3(2),k\in U_{[2]}\right\},
\]
where each matrix multiplies on the left the elements of $\mathbb{Z}\times\mathbb{Z}_2^3$.
\end{proposition}
\begin{proof}
Recall that $W(\Gamma_\mathcal{C})=\aut(\mathbb{Z}_2^3)\cong\mathrm{GL}_3(2)$ and that  $\aut(\mathcal{C})\hookrightarrow\mathrm{Spin}(\mathcal{C}^0,-n)$ via $\mathfrak{Cl}_0(\mathcal{C}^0,-n)\cong\mathrm{End}(\mathcal{C})$. So $W(\Gamma)$ contains $\aut(\mathbb{Z}_2^3)$. Also, we have the automorphism of $\mathcal{L}$ given by the element
\[
\left(\left(\begin{array}{cc}0&1\\-1&0\end{array}\right),1\right)\in\mathrm{SL}_2(\mathbb{F})\times\mathrm{Spin}_7(\mathbb{F}).
\]
This automorphism swaps the two homogeneous components of $V$ in $\mathcal{L}_1$, so the corresponding element of $W(\Gamma)$ inverts the generator of $\mathbb{Z}$ and fixes the elements $\mathbb{Z}_2^3$. Finally, since the kernel of the parity homomorphism $p:U\to\mathbb{Z}_2$ is $2\mathbb{Z}\times\mathbb{Z}_2^3$, $\bar{K}$ gives all automorphisms of $\mathbb{Z}\times\mathbb{Z}_2^3$ that add an element of $\mathbb{Z}_2^3$ to the generator of $\mathbb{Z}$ and fix all elements of $\mathbb{Z}_2^3$.
\end{proof}

Hence, the isomorphism classes of group gradings given by \Cref{def:II} are as follows:
\begin{corollary}\label{cor:II}
$\gr{\FII}_{F(4)}(g,T)\cong\gr{\FII}_{F(4)}(g',T')$ if and only if $T=T'$ and $g'\in gT\cup g^{-1}T$.\qed
\end{corollary}

\subsection{Gradings from TKK construction}
In \cite[\S 3.2]{DEM}, the authors obtain two fine gradings on $F(4)$ using the model of $F(4)$ given by the TKK construction applied to the Kac's ten-dimensional simple Jordan superalgebra $\mathcal{K}_{10}$:
\[
\mathcal{L}=\mathfrak{tkk}(\mathcal{K}_{10})=(\mathfrak{sl}_2\otimes\mathcal{K}_{10})\oplus\mathrm{Der}(\mathcal{K}_{10}).
\]
The bracket is given by
\begin{align*}
&[a\otimes x,b\otimes y]=[a,b]\otimes xy-2n(a,b)[\ell_x,\ell_y],\\%
&[d,a\otimes x]=a\otimes d(x),
\end{align*}
for any $a,b\in\mathfrak{sl}_2$, $x,y\in\mathcal{K}_{10}$, and $d\in\mathrm{Der}(\mathcal{K}_{10})$, where $n$ is the polarization of the quadratic form $n(a)=\det(a)$, which is the norm of $M_2(\mathbb{F})$ as a quaternion algebra.

Recall that $\mathcal{K}_{10}$ can be described in terms of the smaller Kaplansky's superalgebra $\mathcal{K}_3$, which is the three-dimensional Jordan superalgebra
\[
\mathcal{K}_3=\mathbb{F}e\oplus\mathcal{V},
\]
with the even part $\mathbb{F}e$, the odd part $\mathcal{V}$ a two-dimensional vector space endowed with a nonzero symplectic form $(\cdot\,|\,\cdot)$, and the multiplication defined as follows:
$$
e^2=e,\quad ex=xe=\frac12x,\quad xy=(x|y)e,\quad x,y\in V.
$$
The symplectic form extends to a supersymmetric bilinear form on $\mathcal{K}_3$ by means of $(e|e)=\frac12$ and $(e|\mathcal{V})=0$. Then $\mathrm{Der}(\mathcal{K}_3)$ is the orthosymplectic Lie superalgebra of the super vector space $\mathcal{K}_3$ relative to this form, that is, $\mathrm{Der}(\mathcal{K}_3)=\mathfrak{sp}(\mathcal{V})\oplus\mathcal{V}$.
Now, 
\[
\mathcal{K}_{10}=\mathbb{F}1\oplus(\mathcal{K}_3\otimes\mathcal{K}_3),
\]
with unity element $1$ and with product determined by
\[
(x\otimes y)(x'\otimes y')=(-1)^{|y||x'|}\left(xx'\otimes yy'-\frac34(x|x')(y|y')1\right),
\]
for homogeneous elements $x,x',y,y'\in\mathcal{K}_3$. The Lie superalgebra of derivations is
\[
\mathrm{Der}(\mathcal{K}_{10})=(\mathrm{Der}(\mathcal{K}_3)\otimes\mathrm{id})\oplus(\mathrm{id}\otimes\mathrm{Der}(\mathcal{K}_3))\cong (\mathfrak{sp}(\mathcal{V})\oplus\mathfrak{sp}(\mathcal{V}))\oplus(\mathcal{V}\oplus\mathcal{V}),
\]
where the action on $1$ is trivial. 

The algebra $(\mathcal{K}_{10})_0$ is a direct sum of two simple ideals: a copy of $\mathbb{F}$ and the Jordan algebra $\mathcal{J}=\mathbb{F}\oplus(\mathcal{V}\otimes\mathcal{V})$ of the symmetric form on $\mathcal{V}\otimes\mathcal{V}$ defined by 
\[
(v_1\otimes v_2,v_1'\otimes v_2')=\frac12(v_1|v_1')(v_2|v_2').
\]
Accordingly, $\mathcal{L}_0$ is the direct sum of the ideals $\mathfrak{sl}_2\otimes\mathbb{F}\cong\mathfrak{sl}_2$ and $(\mathfrak{sl}_2\otimes\mathcal{J})\oplus\mathrm{Der}(\mathcal{K}_{10})_0$. By \cite[Lemma 3.1]{DEM}, this latter is isomorphic to $\mathfrak{so}(\mathcal{U},q)$, where 
\[
\mathcal{U}=\mathfrak{sl}_2\oplus(\mathcal{V}\otimes\mathcal{V}),\quad q(a+w)=n(a)+(w,w),\quad a\in\mathfrak{sl}_2,\,w\in\mathcal{V}\otimes\mathcal{V}.
\]
In summary, the Lie superalgebra $F(4)$ is identified with
\begin{align*}
&\mathcal{L}_0=(\mathfrak{sl}_2\otimes(\mathcal{K}_{10})_0)\oplus\mathrm{Der}(\mathcal{K}_{10})_0\cong(\mathfrak{sl}_2\otimes\mathbb{F})\oplus\big((\mathfrak{sl}_2\otimes\mathcal{J})\oplus(\mathfrak{sp}(\mathcal{V})\oplus \mathfrak{sp}(\mathcal{V}))\big),\\%
&\mathcal{L}_1=(\mathfrak{sl}_2\otimes(\mathcal{K}_{10})_1)\oplus\mathrm{Der}(\mathcal{K}_{10})_1\cong M_2(\mathbb{F})\otimes(\mathcal{V}\oplus\mathcal{V}).
\end{align*}

Now, let $G$ be an abelian group. Any $G$-grading $\Gamma_{\mathcal{K}_{10}}$ on the superalgebra $\mathcal{K}_{10}$ restricts to the simple ideal $\mathcal{J}$ of its even part and, hence, to $\mathcal{V}\otimes\mathcal{V}$ (see e.g. \cite[Theorem 5.42]{EK13}). It also induces a $G$-grading on $\mathrm{Der}(\mathcal{K}_{10})$. If we have, in addition, any $G$-grading $\Gamma_{\mathfrak{sl}_2}$ on $\mathfrak{sl}_2$, then the tensor product $\mathfrak{sl}_2\otimes\mathcal{K}_{10}$ becomes $G$-graded and, hence, we get a $G$-grading on the superalgebra $\mathcal{L}=\mathfrak{tkk}(\mathcal{K}_{10})$. Moreover, the isomorphism $(\mathfrak{sl}_2\otimes\mathcal{J})\oplus\mathrm{Der}(\mathcal{K}_{10})_0\cong\mathfrak{so}(\mathcal{U},q)$ of \cite[Lemma 3.1]{DEM} preserves degrees.

Note that $\Gamma_{\mathfrak{sl}_2}$ is responsible not only for the grading on the simple ideal $\mathfrak{sl}_2$ of $\mathcal{L}_0$, but also the other simple ideal, $\mathfrak{so}_7$, contains a copy of $\mathfrak{sl}_2$ with the same grading, namely, $\mathfrak{sl}_2\otimes 1_{\mathcal{J}}$. This fact will have a bearing on the quasitorus $\tilde{Q}$ in \Cref{l:seq_f(4)}: it will not be the direct product of its restrictions to $\mathfrak{sl}_2$ and $\mathfrak{so}_7$, as was the case for the two fine gradings considered before. The following result will be applied to $\mathcal{L}_0=\mathfrak{sl}_2\oplus\mathfrak{so}_7$ to find $W(\tilde{Q})$, but it may be of independent interest.

\begin{lemma}\label{l:TKKmain}
Let $\mathcal{A}$ be a finite-dimensional algebra that is the direct sum of subalgebras $\mathcal{A}_1$ and $\mathcal{A}_2$ such that $\aut(\mathcal{A})=\aut(\mathcal{A}_1)\times\aut(\mathcal{A}_2)$ and suppose $\mathcal{A}$ is graded by an abelian group $G$. Let $Q$ be the image of $\widehat{G}$ in $\aut(\mathcal{A})$, $Q_i$ its projection to $\aut(\mathcal{A}_i)$, and $G_i$ the subgroup of $G$ generated by the support of the corresponding grading on $\mathcal{A}_i$, $i=1,2$. Then $W(Q_i)$ is embedded in the automorphism group of $G_i=\mathfrak{X}(Q_i)$ and 
\[
W(Q)=\mathrm{Stab}_{W(Q_1)}(G_0)\times_{\aut(G_0)}\mathrm{Stab}_{W(Q_2)}(G_0),
\]
where $G_0=G_1\cap G_2$ and the fiber product is taken with respect to the restriction maps $\mathrm{Stab}_{W(Q_i)}(G_0)\to\aut(G_0)$. 
In particular, if $G_1\subseteq G_2$ then
\[
W(Q)\cong\{\alpha\in\mathrm{Stab}_{W(Q_2)}(G_1)\mid\alpha|_{G_1}\in W(Q_1)\}.
\]
\end{lemma}
\begin{proof}
Without loss of generality, we may assume that $G$ is generated by the support of the grading on $\mathcal{A}$, so we may identify $G$ with $\mathfrak{X}(Q)$. Then $G$ is the push-forward of the inclusion maps $G_0\to G_i$ and, hence, $Q$ is the pull-back $Q_1\times_{Q_0}Q_2$ of the restriction maps $Q_i\to Q_0$. The normalizer of $Q_i$ in $\aut(\mathcal{A}_i)$ acts on $Q_i$ and, hence, on its dual $G_i$, which gives us a homomorphism $\rho_i:N(Q_i)\to\aut(G_i)$ with kernel $C(Q_i)$ and, hence, an embedding $W(Q_i)\hookrightarrow\aut(G_i)$.

Now, the normalizer of $Q$ in $\aut(\mathcal{A})=\aut(\mathcal{A}_1)\times\aut(\mathcal{A}_2)$ projects to $N(Q_i)$, so it is a subgroup of $N(Q_1)\times N(Q_2)$. Since $Q=\{(\chi_1,\chi_2)\in Q_1\times Q_2\mid\chi_1|_{G_0}=\chi_2|_{G_0}\}$ and, by definition of $\rho_i$, we have $\chi_i(\rho_i(
\psi_i)(g))=(\psi_i^{-1}\chi_i\psi_i)(g)$ for all $\chi_i\in Q_i=\widehat{G}_i$, $\psi_i\in N(Q_i)$, and $g\in G_0$, it is easy to check that 
\[
N(Q)=\{(\psi_1,\psi_2)\in N(Q_1)\times N(Q_2)\mid
\rho_1(\psi_1)|_{G_0}=\rho_2(\psi_2)|_{G_0}\}.
\]
The condition on $(\psi_1,\psi_2)$ implies  that $\rho_i(\psi_i)$ leaves $G_0$ invariant. Passing modulo $C(Q)=C(Q_1)\times C(Q_2)$, we obtain the result.
%
%
\end{proof}

The fine gradings on $\mathcal{K}_{10}$ were found in \cite{CDM}: there are two, up to equivalence, with universal groups $\mathbb{Z}^2$ and $\mathbb{Z}\times\mathbb{Z}_2$. When combined with the fine $\mathbb{Z}_2^2$-grading on $\mathfrak{sl}_2$, they give rise to two non-equivalent fine gradings on $F(4)$. We shall consider them in the following two subsections.

\subsubsection{The $\mathbb{Z}^2\times\mathbb{Z}_2^2$-grading\label{sec:tkk_gr1}}
First we define a $G$-grading on $\mathcal{K}_3$ by letting its even part to have degree $e$ and putting a $G$-grading on the odd part $\mathcal{V}$ such that $(\cdot\,|\,\cdot)$ becomes homogeneous of degree $e$. Up to isomorphism, the grading on $\mathcal{V}=\mathbb{F}^2$ has the form $\Gamma_\mathcal{V}(g,g^{-1})$ for some $g\in G$, and we will denote the resulting grading on $\mathcal{K}_3$ by $\Gamma_{\mathcal{K}_3}(g)$. 
(Since $\aut(\mathcal{K}_3)\cong\mathrm{Sp}(\mathcal{V})=\mathrm{SL}(\mathcal{V})$, it is clear that $\Gamma_{\mathcal{K}_3}(g)\cong\Gamma_{\mathcal{K}_3}(g')$ if and only if $g'\in\{g,g^{-1}\}$.)

\begin{definition}\label{def:III}
Let $g_1$, $g_2$, $a$, $b\in G$ be such that $\langle a,b\rangle\cong\mathbb{Z}_2^2$. Then $\Gamma_{\mathcal{K}_3}(g_1)$ and $\Gamma_{\mathcal{K}_3}(g_2)$ are two $G$-gradings on $\mathcal{K}_3$, whose tensor product gives a $G$-grading $\Gamma_{\mathcal{K}_{10}}(g_1,g_2)$ on $\mathcal{K}_{10}$ (with $1$ having degree $e$). Combined with $\gr{\Pa}_{\mathfrak{sl}_2}(a,b)$ on $\mathfrak{sl}_2$, it gives rise to a $G$-grading on $\mathcal{L}=\mathfrak{tkk}(\mathcal{K}_{10})$, which we denote by $\gr{\FIII}_{F(4)}(g_1,g_2,a,b)$. 
\end{definition}

\begin{remark}\label{rm:TKK1}
On the algebra $\mathcal{L}_0$ and vector space $\mathcal{L}_1$, this grading looks as follows:
\begin{eqnarray*}
\mathcal{L}_0&\cong&(\mathfrak{sl}_2,\gr{\Pa}_{\mathfrak{sl}_2}(a,b))\oplus(\mathfrak{so}_7,\Gamma_{\mathfrak{so}_7}(\gamma)),\\%
\mathcal{L}_1&\cong&(M_2,\gr{\Pa}_{M_2}(a,b))\otimes\big((\mathcal{V},\Gamma_\mathcal{V}(g_1,g_1^{-1}))\oplus(\mathcal{V},\Gamma_\mathcal{V}(g_2,g_2^{-1}))\big),
\end{eqnarray*}
where $\gamma=(a,b,ab;\,g_1g_2,g_1^{-1}g_2^{-1},g_1g_2^{-1},g_1^{-1}g_2)$.
\end{remark}

Taking for $g_1,g_2$ and $a,b$ the standard bases of $\mathbb{Z}^2$ and $\mathbb{Z}_2^2$, we obtain a fine grading $\Gamma$ with universal group $U=\mathbb{Z}^2\times\mathbb{Z}_2^2$. Moreover, the support of $\mathcal{L}_0$ contains all torsion elements of $U$, since they all belong to $\mathrm{Supp}(\mathfrak{so}_7)$. Hence, the fine grading $\Gamma$ admits no almost fine proper coarsening, and every admissible homomorphism $U\to G$ is injective on the torsion subgroup. 

\begin{proposition}\label{prop:WeylIII}
Let $\Gamma$ be the fine grading on $F(4)$ with universal group $U=\mathbb{Z}^2\times\mathbb{Z}_2^2$. Then
$W(\Gamma)\cong D_4\times\mathrm{AGL}_2(2)\cong D_4\times\left(U_{[2]}\rtimes\mathrm{GL}_2(2)\right)$. In matrix form,
\[
W(\Gamma)=\left\{\left(\begin{array}{cc}
g&0\\k&A
\end{array}\right)\mid g\in D_4, A\in\mathrm{GL}_2(2),k\in U_{[2]}\right\},
\]
where each matrix multiplies on the left the elements of $\mathbb{Z}^2\times\mathbb{Z}_2^2$, and $D_4\cong C_2^2\rtimes C_2$ acts on $\mathbb{Z}^2$ by inversion and permutation of the standard basis elements.
\end{proposition}
\begin{proof}
Since in this case both $\Gamma_{\mathfrak{sl}_2}$ and $\Gamma_{\mathfrak{so}_7}$ are fine (see \Cref{rm:TKK1}), with universal groups $\mathbb{Z}_2^2$ and $\mathbb{Z}_2^2\times\mathbb{Z}^2$, the exact sequence \eqref{seq_f(4)_2} of \Cref{l:seq_f(4)} gives
\[
(\mathbb{F}^\times)^2\times C_2^2\to C_2^2\times((\mathbb{F}^\times)^2\times C_2^2)\to\bar{K}\to 1,
\]
so $|\bar{K}|\ge2^2$. Then, \Cref{Misha's lemma} implies that $\bar{K}\cong U_{[2]}$. The kernel of the parity homomorphism $U\to\mathbb{Z}_2$ is
\[
\{(x,y)\in\mathbb{Z}^2\mid x+y\equiv0\!\!\!\!\pmod{2}\}\times\mathbb{Z}_2^2,
\]
so $\bar{K}$ acts on $U$ as indicated.

Now, from \Cref{l:TKKmain} and the exact sequence \eqref{seq_f(4)_1} of \Cref{l:seq_f(4)}, we get 
\[
1\to\bar{K}\to W(\Gamma)\to W(\Gamma_{\mathfrak{so}_7}),
\]
where the last map is induced by restriction. Since $\Gamma_{\mathfrak{so}_7}$ is fine, we know that $W(\Gamma_{\mathfrak{so}_7})\cong S_3\times D_4$ (see e.g. \cite[\S 3.4]{EK13}), where $S_3$ permutes the first three entries of the septuple $\gamma$ and $D_4\cong C_2^2\rtimes C_2$ permutes the last four entries while keeping their pairing. Finally, any permutation in $S_3$ is realized by some $\psi\in\aut(\Gamma_{\mathfrak{sl}_2})$, which lifts canonically to $\mathfrak{tkk}(\mathcal{K}_{10})$. On the other hand, any permutation in $D_4$ is realized by some $\varphi\in\aut(\mathcal{K}_{10})\cong\mathrm{Sp}(\mathcal{V})^2\rtimes C_2$, which also lifts canonically to $\mathfrak{tkk}(\mathcal{K}_{10})$. The result follows.
\end{proof}

This solves the isomorphism problem for the family of gradings in \Cref{def:III}:
\begin{corollary}\label{cor:III}
$\gr{\FIII}_{F(4)}(g_1,g_2,a,b)\cong\gr{\FIII}_{F(4)}(g_1',g_2',a',b')$ if and only if $\langle a,b\rangle=\langle a',b'\rangle$ and there exist $\sigma\in S_2$ and $t\in\langle a,b\rangle$ such that $g_{i}'\in\{g_{\sigma(i)}t, g_{\sigma(i)}^{-1}t\}$ for $i=1,2$. \qed
\end{corollary}

\begin{remark}
$\Gamma_{\mathcal{K}_{10}}(g_1,g_2)\cong\Gamma_{\mathcal{K}_{10}}(g_1',g_2')$ if and only if there exists $\sigma\in S_2$ such that $g_{i}'\in\{g_{\sigma(i)}, g_{\sigma(i)}^{-1}\}$ for $i=1,2$.
\end{remark}

\subsubsection{The $\mathbb{Z}\times\mathbb{Z}_2^3$-grading} 
The super vector space $\mathcal{K}_3\otimes\mathcal{K}_3$ has an involution $\tau$ defined by
\[
\tau(x\otimes y)=(-1)^{|x||y|}y\otimes x,
\]
which gives an automorphism of $\mathcal{K}_{10}$ by setting $\tau(1)=1$.

\begin{definition}\label{def:IV}
Let $a$, $b$, $g$, $h\in G$ such that $\langle a,b,h\rangle\cong\mathbb{Z}_2^3$. We can refine the $G$-grading $\Gamma_{\mathcal{K}_{10}}(g,g)$ using $\tau$ and the element $h$ as follows: we keep the degree of $\tau$-symmetric homogeneous elements as is and shift the degree of $\tau$-skew-symmetric ones by $h$. We will denote by $\Gamma^{\tau}_{\mathcal{K}_{10}}(g,h)$ the resulting $G$-grading on $\mathcal{K}_{10}$ and by $\gr{\FIV}_{F(4)}(g,h,a,b)$ the $G$-grading on $\mathcal{L}=\mathfrak{tkk}(\mathcal{K}_{10})$ obtained by combining it with $\gr{\Pa}_{\mathfrak{sl}_2}(a,b)$. 
\end{definition}

\begin{remark}\label{rm:TKK2}
On the algebra $\mathcal{L}_0$ and vector space $\mathcal{L}_1$, this grading looks as follows:
\begin{eqnarray*}
\mathcal{L}_0&=&(\mathfrak{sl}_2,\gr{\Pa}_{\mathfrak{sl}_2}(a,b))\oplus(\mathfrak{so}_7,\Gamma_{\mathfrak{so}_7}(\gamma)),\\%
\mathcal{L}_1&=&(M_2,\gr{\Pa}_{M_2}(a,b))\otimes(\mathcal{V}\oplus\mathcal{V},\Gamma_{\mathcal{V}\oplus\mathcal{V}}(g,h)),
\end{eqnarray*}
where $\gamma=(a,b,ab,e,h;\,g^2h,g^{-2}h)$ and $\Gamma_{\mathcal{V}\oplus\mathcal{V}}(g,h)$ is the refinement of $\Gamma_{\mathcal{V}}(g,g^{-1})\oplus\Gamma_{\mathcal{V}}(g,g^{-1})$ using the swap $(v,w)\mapsto(w,v)$ and $h$. 
\end{remark}

Taking $1\in\mathbb{Z}$ for $g$ and the standard basis of $\mathbb{Z}_2^3$ for $h,a,b$, we obtain a fine grading $\Gamma$ with universal group $U=\mathbb{Z}\times\mathbb{Z}_2^3$. Once again, all torsion elements of $U$ belong to  $\mathrm{Supp}(\Gamma_{\mathfrak{so}_7})$. Hence, there is no almost fine proper coarsening of $\Gamma$, and every admissible homomorphism $U\to G$ is injective on the torsion subgroup. 

\begin{proposition}\label{prop:WeylIV}
Let $\Gamma$ be the fine grading on $F(4)$ with universal group $U=\mathbb{Z}\times\mathbb{Z}_2^3$ obtained from the TKK construction. 
Then $W(\Gamma)\cong C_2\times(U_{[2]}\rtimes\mathrm{AGL}_2(2))$. In matrix form,
\[
W(\Gamma)=\left\{\begin{pNiceArray}{c|cc}%
\varepsilon&0&0\\\hline%
\Block{2-1}{k}&1&0\\%
&t&A%
\end{pNiceArray}\mid\varepsilon\in\{\pm1\}, A\in\mathrm{GL}_2(2),k\in U_{[2]}=\mathbb{Z}_2^2\times\mathbb{Z}_2,t\in\mathbb{Z}_2^2\right\},
\]
where each matrix multiplies on the left the elements of $\mathbb{Z}\times(\mathbb{Z}_2\times\mathbb{Z}_2^2)$.
\end{proposition}

\begin{proof}
In this case $\Gamma_{\mathfrak{so}_7}$ is not fine and admits a fine refinement with universal group  $\mathbb{Z}_2^4\times\mathbb{Z}$ (see \Cref{rm:TKK2}). It means that $\mathrm{Stab}(\Gamma_{\mathfrak{so}_7})$ contains a copy of $\mathbb{F}^\times\times C_2^4$. Hence, the exact sequence \eqref{seq_f(4)_2} of 
\Cref{l:seq_f(4)} gives
\[
\mathbb{F}^\times\times C_2^3\to C_2^2\times(\mathbb{F}^\times\times C_2^4)\to\bar{K},
\]
so $|\bar{K}|\ge2^3$, which implies $\bar{K}\cong U_{[2]}$. The kernel of the parity homomorphism is $2\mathbb{Z}\times\mathbb{Z}_2^3$, so $\bar{K}$ acts on $U$ as indicated.

Now, consider the projections $Q_{\mathfrak{sl}_2}$ and $Q_{\mathfrak{s0}_7}$ of $\tilde{Q}$, and 
let $T$ be the subgroup generated by $\mathrm{Supp}(\Gamma_{\mathfrak{sl}_2})$, i.e., $T=\langle a,b\rangle$ with $a=(0,0,1,0)$ and $b=(0,0,0,1)$. Since $W(Q_{\mathfrak{sl}_2})$ is the entire group of automorphisms of $T$, \Cref{l:TKKmain} tells us that $W(\tilde{Q})\cong\mathrm{Stab}_{W(Q_{\mathfrak{so}_7})}(T)$. Combining this with the exact sequence \eqref{seq_f(4)_1} of \Cref{l:seq_f(4)}, we obtain 
\[
1\to\bar{K}\to W(\Gamma)\to\mathrm{Stab}_{W(Q_{\mathfrak{so}_7})}(T)\to 1.
\]
Since the entries of the septuple $\gamma$ defining $\Gamma_{\mathfrak{so}_7}$ are all distinct, any automorphism of this grading is the conjugation by a monomial matrix, which must keep the pairing of the last two entries of $\gamma$ (see e.g. \cite[p.~99]{EK13}). It is easy to check that if the conjugation by such a matrix leaves invariant the sum of the homogeneous components of $\Gamma_{\mathfrak{so}_7}$ with degrees in $T$, then the corresponding permutation fixes position $5$ (which is occupied by $h=(0,1,0,0)$) and, conversely, for any permutation fixing $5$ and leaving $\{6,7\}$ invariant, the conjugation by the corresponding permutation matrix moves the components of $\Gamma_{\mathfrak{so}_7}$ according to an automorphism of the subgroup generated by $\mathrm{Supp}(\Gamma_{\mathfrak{so}_7})$ that leaves $T$ invariant. The result follows.
\end{proof}

This solves the isomorphism problem for the family of gradings in \Cref{def:IV}:
\begin{corollary}\label{cor:IV}
$\gr{\FIV}_{F(4)}(g,a,b,h)\cong\gr{\FIV}_{F(4)}(g',a',b',h')$ if and only if $\langle a,b\rangle=\langle a',b'\rangle$, $\langle a,b,h\rangle=\langle a',b',h'\rangle$, and $g'\in gH\cup g^{-1}H$, where $H=\langle a,b,h\rangle$.\qed
\end{corollary}

\begin{remark}
$\Gamma^\tau_{\mathcal{K}_{10}}(g,h)\cong\Gamma^\tau_{\mathcal{K}_{10}}(g',h')$ if and only if $h=h'$ and $g'\in g\langle h\rangle\cup g^{-1}\langle h\rangle$. 
\end{remark}

\subsection{Gradings from the fine $\mathbb{Z}_4\times\mathbb{Z}_2^3$-grading}

We now turn to the fifth and final fine grading in \cite[Theorem 4.1]{DEM}, which is constructed as follows (see \cite[\S 2.4, 4.3]{DEM} for more details).

Let $a,b,h\in G$ be elements of order $2$ and $g\in G$ an element of order $4$ that generate a subgroup isomorphic to $\mathbb{Z}_4\times\mathbb{Z}_2^3$. We will again regard $\mathcal{Q}=M_2(\mathbb{F})$ as a quaternion algebra, endowed with its standard involution: $\bar{x}=n(x,1)1-x$ (of symplectic type). We equip $\mathcal{Q}$ with the grading $\gr{\Pa}_{M_2}(h,g^2)$ and let $\mathcal{Q}'$ be a copy of $\mathcal{Q}$, but equipped with the grading $\gr{\Pa}_{M_2}(a,b)$. Then $\mathcal{Q}'\otimes\mathcal{Q}$ becomes a graded right $\mathcal{Q}$-module via $(x\otimes y)z=x\otimes(yz)$. Also, the following map is an isomorphism of graded algebras:
\begin{align*}
\Phi:\mathcal{Q}'\otimes\mathcal{Q}'\otimes\mathcal{Q}&\to\mathrm{End}_{\mathcal{Q}}(\mathcal{Q}'\otimes\mathcal{Q})\\
q_1\otimes q_2\otimes q_3&\mapsto\Phi_{q_1\otimes q_2\otimes q_3}:x\otimes y\mapsto(q_1x\bar{q}_2)\otimes(q_3y).
\end{align*}
One can find a subspace $\mathcal{U}$ in $\mathcal{Q}'\otimes\mathcal{Q}'\otimes\mathcal{Q}$ of dimension $7$ that has a basis consisting of homogeneous mutually anticommuting elements whose squares are nonzero scalars. Using the notation of \Cref{examples}, we can take:
\[
\mathcal{U}=\mathrm{span}\{A\otimes I\otimes I, 
C\otimes I\otimes I, 
B\otimes I\otimes A, 
B\otimes C\otimes B, 
B\otimes A\otimes B, 
B\otimes I\otimes C, 
B\otimes B\otimes B 
\}.
\]
If we define a quadratic form $q$ on $\mathcal{U}$ by declaring this basis orthogonal and $b^2=q(b)1$ for each basis element, then the Clifford algebra $\mathfrak{Cl}(\mathcal{U},q)$ becomes graded, and we get a homomorphism of graded algebras $\mathfrak{Cl}(\mathcal{U},q)\to\mathcal{Q}'\otimes\mathcal{Q}'\otimes\mathcal{Q}$, which restricts to an isomorphism on $\mathfrak{Cl}_0(\mathcal{U},q)$. Composing the standard embedding $\mathfrak{so}(\mathcal{U},q)\hookrightarrow\mathfrak{Cl}_0(\mathcal{U},q)$ with this isomorphism and the isomorphism $\Phi$, we make $\mathcal{Q}'\otimes\mathcal{Q}$ a graded $\mathfrak{so}_7$-module, which also has a compatible (left) action of $\mathfrak{sl}_2\subset\mathcal{Q}$ via $z(x\otimes y):=x\otimes (y\bar{z})=-x\otimes (yz)$ for all $z\in\mathfrak{sl_2}$. Thus, we can take $\mathcal{Q}'\otimes\mathcal{Q}$ as the graded $\mathcal{L}_0$-module $\mathcal{L}_1$, but there is one further consideration. The bracket $\mathcal{L}_1\otimes\mathcal{L}_1\to\mathcal{L}_0$ is determined up to a nonzero scalar, so it is automatically homogeneous. It turns out that in this case the degree is $g^2$, so we need to shift the grading on $\mathcal{L}_1$ by $g$ to make $\mathcal{L}$ a graded superalgebra.

\begin{definition}\label{def:VF(4)}
Let $g,h,a,b\in G$ be elements as above: $\langle g^2,h,a,b\rangle\cong\mathbb{Z}_2^4$. Then, taking 
$\mathcal{L}_0=\mathcal{Q}^0\oplus\mathfrak{so}(\mathcal{U},q)$ and $\mathcal{L}_1=(\mathcal{Q}'\otimes\mathcal{Q})^{[g]}$, we obtain a $G$-grading on $F(4)$, which we denote by $\gr{\FV}_{F(4)}(g,a,b,h)$. 
\end{definition}

\begin{remark}\label{rem:Z4Z23}
On the algebra $\mathcal{L}_0$ and vector space $\mathcal{L}_1$, this grading looks as follows:
\begin{eqnarray*}
\mathcal{L}_0&=&(\mathfrak{sl}_2,\gr{\Pa}_{\mathfrak{sl}_2}(h,g^2))\oplus(\mathfrak{so}_7,\Gamma_{\mathfrak{so}_7}(\gamma)),\\
\mathcal{L}_1&=&\left(M_2\otimes M_2,\gr{\Pa}_{M_2}(a,b)\otimes\gr{\Pa}_{M_2}(h,g^2)\right)^{[g]},
\end{eqnarray*}
where $\gamma=(a,ab,bh,ag^2,abg^2,bhg^2,g^2)$.
\end{remark}

Taking $1\in\mathbb{Z}_4$ for $g$ and the standard basis of $\mathbb{Z}_2^3$ for $h$, $a$, $b$, we get a fine grading $\Gamma$ with universal group $U=\mathbb{Z}_4\times\mathbb{Z}_2^3$. The support of $\Gamma_{\mathfrak{sl}_2}$ consists of the nontrivial elements of the subgroup $T=\langle h,g^2\rangle$. To see the support of $\Gamma_{\mathfrak{so}_7}$, it is convenient to partition the set $\{1,\ldots,7\}$ into blocks $\{1,4\}$, $\{2,5\}$, $\{3,6\}$, and $\{7\}$. Then the degree of the basis element $E_{ij}-E_{ji}$ ($i<j$) is $g^2$ if $i$ and $j$ are in the same block, an element of the subgroup $S=\langle ah,b,g^2\rangle$ different from $e$ and $g^2$ (each occurring twice) if $i,j\ne 7$ belong to different blocks, and otherwise an element of the set $\{a,ab,bh\}\langle g^2\rangle$ (each occurring once). 
Thus, the subgroups $S$ and $T$ are special for this grading: we have $S\cap T=\langle g^2\rangle$, the component of $\Gamma_0$ with degree $g^2$ has dimension $4$ ($1$ in $\mathfrak{sl}_2$ and $3$ in $\mathfrak{so}_7$), the components with degree in $S\smallsetminus T$ have dimension $2$ (contained in $\mathfrak{so}_7$), and the components with degrees in $U_{[2]}\smallsetminus S$ have dimension $1$. 
In particular, the nontrivial elements of $U_{[2]}$ constitute the support of $\mathcal{L}_0$. It follows that $\Gamma$ has no almost fine proper coarsening and that every admissible homomorphism $U\to G$ must be injective. 

\begin{proposition}\label{prop:WeylV}
Let $\Gamma$ be the fine grading on $F(4)$ with universal group $U=\mathbb{Z}_4\times\mathbb{Z}_2^3$ and let $T=\langle h,g^2\rangle$ and $S=\langle ah,b,g^2\rangle$. Then 
$
W(\Gamma)=\mathrm{Stab}_{\aut(U)}(S,T)\cong U_{[2]}\rtimes \mathrm{Stab}_{\aut(U_{[2]})}(S,T).
$
In matrix form, we have
\[
W(\Gamma)=\left\{\begin{pNiceArray}{c|cc}%
\varepsilon&\Block{1-2}{\ell}\\\hline%
\Block{2-1}{k}&1&0\\%
&0&A%
\end{pNiceArray}\mid\varepsilon\in\{\pm1\},k\in\mathbb{Z}_2^3,\ell\in\left(2\mathbb{Z}_4\right)^3,A\in\mathrm{GL}_2(2)\right\},
\]
where each matrix multiplies on the left the elements of $\mathbb{Z}_4\times\mathbb{Z}_2^3$, with the elements of $\mathbb{Z}_2^3$ being represented relative to the ordered basis $\{h,ah,b\}$.
\end{proposition}

\begin{proof}
$\Gamma_{\mathfrak{so}_7}$ admits a fine refinement with universal group $\mathbb{Z}_2^6$, so $\mathrm{Stab}(\Gamma_{\mathfrak{so}_7})$ contains a copy of $C_2^6$. Then, the exact sequence \eqref{seq_f(4)_2} of \Cref{l:seq_f(4)} gives 
\[
1\to C_2\to C_2^3\times C_4\to C_2^2\times C_2^6\to\bar{K},
\]
so $|\bar{K}|\ge2^4$, which implies $\bar{K}\cong U_{[2]}$. The kernel of the parity homomorphism is $U_{[2]}$, 
so the action of $\bar{K}$ gives both $\varepsilon$ and $k$ in the matrix above. It follows that $W(\Gamma)\cong\bar{K}\rtimes\mathrm{Stab}_{W(\Gamma)}(g)$ and, hence, the exact sequence \eqref{seq_f(4)_1} of \Cref{l:seq_f(4)},
\[
1\to\bar{K}\to W(\Gamma)\to W(\tilde{Q})\to 1,
\]
is split. From the discussion above, it is clear that $W(\tilde{Q})\subset\mathrm{Stab}_{\aut(U_{[2]})}(S,T)$. From the matrix representation relative to the basis $\{g^2,h,ah,b\}$, it is clear that this stabilizer has order $2^3\cdot 6$. 
It remains to observe that, for any permutation that moves the blocks $\{1,4\}$, $\{2,5\}$, $\{3,6\}$ and fixes $7$, the conjugation by the corresponding permutation matrix moves the components of $\Gamma_{\mathfrak{so}_7}$ according to an automorphism of $U_{[2]}$ that leaves $T$ invariant and, hence, yields an element of $W(\tilde{Q})$ by \Cref{l:TKKmain}. Since there are $2^3\cdot 6$ such permutations, the result follows.
\end{proof}

\begin{corollary}\label{cor:V}
$\gr{\FV}_{F(4)}(g,a,b,h)\cong\gr{\FV}_{F(4)}(g',a',b',h')$ if and only if $\langle h,g^2\rangle=\langle h',g'^2\rangle$ and $\langle ah,b,g^2\rangle=\langle a'h',b',g'^2\rangle$.\qed
\end{corollary}

\section{Gradings on $D(\alpha)$\label{sec:d}}

\subsection{The Lie superalgebras $D(\alpha)=D(2,1;\alpha)$}

We let $\mathcal{L}=\mathcal{L}_0\oplus\mathcal{L}_1$ with $\mathcal{L}_0=\mathfrak{s}_1\oplus\mathfrak{s}_2\oplus\mathfrak{s}_3$ and $\mathcal{L}_1=V_1\otimes V_2\otimes V_3$, where $\mathfrak{s}_i$ are copies of $\mathfrak{sl}_2$ and $V_i$ are copies of the irreducible 2-dimensional $\mathfrak{sl}_2$-module $V$. Fix an alternating (invariant) bilinear form $\psi:V\times V \rightarrow \mathbb{F}$ and a symplectic basis $\{u,v\}$ of $V$ (i.e., $\psi(u,v)=1$). Consider the invariant symmetric bilinear map  $p_i:V_i\times V_i\rightarrow \mathfrak{s}_i$ defined by $p_i(x,y)= x\psi_i(y,\cdot)+y\psi_i(x,\cdot)$, where $\psi_i$ is a copy of $\psi$. For any $\sigma_1,\sigma_2,\sigma_3\in \mathbb{F}$, we can define a symmetric $\mathcal{L}_0$-invariant bilinear map $[\,,]:\mathcal{L}_1\otimes\mathcal{L}_1\to\mathcal{L}_0$ by 
	\[ [\,,]=\sigma_1p_1\otimes\psi_2\otimes\psi_3+ \sigma_2\psi_1\otimes p_2\otimes\psi_3 + \sigma_3\psi_1\otimes\psi_2\otimes p_3.
        \]
	This bracket satisfies the Jacobi identity if and only if $\sigma_1+\sigma_2+\sigma_3=0$. We will temporarily denote the resulting Lie superalgebra by $D(\sigma_1,\sigma_2,\sigma_3)$. For each $\pi\in S_3$ and $\lambda\in \mathbb{F}^\times$, we have the isomorphism $D(\sigma_1,\sigma_2,\sigma_3)\cong D(\lambda\sigma_{\pi^{-1}(1)},\lambda\sigma_{\pi^{-1}(2)},\lambda\sigma_{\pi^{-1}(3)})$ given by \begin{gather*}
		(x_1,x_2,x_3)\in\mathcal{L}_0 \mapsto (x_{\pi(1)},x_{\pi(2)},x_{\pi(3)}), \\
		x_1\otimes x_2\otimes x_3 \in \mathcal{L}_1 \mapsto \lambda^{-\frac{1}{2}}x_{\pi(1)}\otimes x_{\pi(2)} \otimes x_{\pi(3)}. 
	\end{gather*}
	Thus, excluding the trivial case $\sigma_1=\sigma_2=\sigma_3=0$, it is sufficient to consider the algebras $D(\alpha)=D(1,\alpha,-1-\alpha)$, where $\alpha=\sigma_2/\sigma_1$ is an element of the projective line $\mathbb{P}^1(\mathbb{F})=\mathbb{F}\cup\{\infty\}$. Moreover, we have $D(\alpha)\cong D(\alpha')$ if and only if $\alpha'=\pi.\alpha$ for some $\pi\in S_3$, where the action of $S_3$ is defined by $(12).\alpha = \alpha^{-1}$ and $(123).\alpha=-(1+\alpha)^{-1}$. 
 
	Let $\stab(\alpha)=\{\pi\in S_3\mid\pi.\alpha=\alpha\}$. The only orbits with fewer than six elements are $\{\omega,\omega^2\}$, with $\omega$ a primitive cube root of unity (so $|\stab(\alpha)|=3$), and the two orbits with $|\stab(\alpha)|=2$: $\{1,-2,-\frac{1}{2}\}$ and $\{-1,0,\infty\}$. Note that $D(-1)$ is not simple; its unique nontrivial ideal is isomorphic to $\mathfrak{psl}(2\,|\,2)$ and will be used in Section \ref{sec:a}. Moreover, we have $D(1)\cong \mathfrak{osp}(4\,|\,2)$, which is the Lie superalgebra of type $D(2,1)$. For this reason the algebras $D(\alpha)$ are usually denoted by $D(2,1;\alpha)$ and considered as deformations of the orthosymplectic algebra $\mathfrak{osp}(4\,|\,2)$. Here we use the shorter notation, since there is no risk of ambiguity.
 
    Let $\mathcal{L}=D(\alpha)$ and define a homomorphism $\iota: \slg_2(\mathbb{F})^3\rightarrow \aut(D(\alpha))$ as follows:
    \begin{equation} \begin{split}        
            \iota(f,g,h) :\mathcal{L} &\to \mathcal{L}, \\
            (x_1,x_2,x_3)\in \mathcal{L}_0 &\mapsto (fx_1f^{-1},gx_2g^{-1},hx_3h^{-1}), \\
            x_1\otimes x_2\otimes x_3\in \mathcal{L}_1 &\mapsto  fx_1\otimes gx_2\otimes hx_3.
    \end{split}\end{equation}
    Note that $\ker(\iota)=\{(\varepsilon_1,\varepsilon_2,\varepsilon_3)\in\slg_2(\mathbb{F})\mid\varepsilon_i\in\{\pm1\},\,\varepsilon_1\varepsilon_2\varepsilon_3=1\}$.
    By \cite[Theorem~1]{S} (see also \cite[Theorem~4.1]{GP}), the image of $\iota$ is the subgroup of the inner automorphisms $\inn(D(\alpha))$. Moreover, the group of outer automorphisms $\out(D(\alpha))=\aut(D(\alpha))/\inn(D(\alpha))$ is isomorphic to $\stab(\alpha)$.

	To simplify the exposition, we introduce the following notation. For any $x\in\mathfrak{sl}_2$, we set 
    \[ 
    x_1=(x,0,0)\in \mathfrak{s}_1\subseteq \mathcal{L}_0,\ x_2=(0,x,0)\in\mathfrak{s}_2\subseteq \mathcal{L}_0,\ x_3=(0,0,x)\in\mathfrak{s}_3\subseteq \mathcal{L}_0. 
    \]
    We will also use the following basis of $\mathcal{L}_1$:  
    \begin{gather*}
    w_0= v\otimes v\otimes v,\quad w_1= v\otimes u\otimes u,\quad w_2= u\otimes v\otimes u,\quad w_3= u\otimes u\otimes v, \\
    w'_0= u\otimes u\otimes u,\quad w'_1= u\otimes v\otimes v,\quad w'_2= v\otimes u\otimes v,\quad w'_3= v\otimes v\otimes u.
    \end{gather*}

    To give a description of the Weyl groups of fine gradings, we fix some matrices in $\slg_2(\mathbb{F})$ that give generators for the Weyl groups of the two fine grading on $\mathfrak{sl}_2$ (see \Cref{examples}). The Weyl group of the Cartan grading is generated by the conjugation by $C$. For each permutation $\pi$ of the three elements $\{a,b,c\}$ in the support of the Pauli grading, we fix a matrix $Y_\pi$ so that the conjugation by it acts as $\pi$ on the homogeneous components. For example, we can take
    \[ 
    Y_{(a c)}=\frac{1}{i\sqrt{2}}\begin{pmatrix}1 & 1 \\ 1 & -1 \end{pmatrix},\quad Y_{(a b c)}=\frac{1+i}{2} \begin{pmatrix}1 & 1 \\ i & -i \end{pmatrix},\quad\text{where }i^2=-1. 
    \]
 
\subsection{Generic orbits}

    If $|\stab(\alpha)|=1$ then there are no outer automorphisms, so $\iota$ is surjective and describes all automorphisms of $D(\alpha)$.

    By \cite[Theorem~8.5]{DEM}, there are five fine gradings in this case (up to equivalence): one with universal group $\mathbb{Z}^3$ and type $(14,0,1)$ (the Cartan grading), one with universal group $\mathbb{Z}_4\times\mathbb{Z}_2^2$ and type $(14,0,1)$, and three with universal group $\mathbb{Z}\times\mathbb{Z}_2^2$ and type $(11,3)$. For each of these gradings, we will show that there are no almost fine proper coarsenings, determine the admissible homomorphisms $U\to G$ for any abelian group $G$, and compute the Weyl group.

    \begin{definition}[Gradings of type $\IC$] \label{def:IC}
        Let $h_1,h_2,h_3,h_4\in G$ be elements satisfying $h_1h_2h_3h_4=e$. We define $\gr{\IC}_D(h_1,h_2,h_3,h_4)$ to be the grading which is $\gr{\C}_{\mathfrak{sl}_2}(h_2h_3)$ on $\mathfrak{s}_1$, $\gr{\C}_{\mathfrak{sl}_2}(h_1h_3)$ on $\mathfrak{s}_2$, $\gr{\C}_{\mathfrak{sl}_2}(h_1h_2)$ on $\mathfrak{s}_3$, while on $\mathcal{L}_1$ it is defined by $\deg(w_0)=h_4$, $\deg(w_1)=h_1$, $\deg(w_2)=h_2$, $\deg(w_3)=h_3$ and $\deg(w'_i)=\deg(w_i)^{-1}$ for each $i=0,1,2,3$.
    \end{definition}

    \begin{proposition} \label{prop:IC_fine}
        A fine grading on $D(\alpha)$ with universal group $\mathbb{Z}^3$ is equivalent to $\gr{\IC}_D(e_1,e_2,e_3,e_4)$ where $e_1=(1,0,0)$, $e_2=(0,1,0)$, $e_3=(0,0,1)$, and $e_4=(-1,-1,-1)$. It has no almost fine proper coarsenings. Moreover, the $G$-gradings of type $\IC$ are precisely the gradings induced by admissible homomorphisms $\mathbb{Z}^3\rightarrow G$.
    \end{proposition}
    
    \begin{proof}
    Any fine $\mathbb{Z}^3$-grading on $D(\alpha)$ is given by the eigenspace decomposition of the adjoint action of a Cartan subalgebra of $\mathcal{L}_0$. Without loss of generality, we can choose the one spanned by $H_1$,$H_2$,$H_3$. Then the homogeneous elements of Definition~\ref{def:IC} are common eigenvectors for these three elements. Fix $x\in\mathcal{L}$ in a common eigenspace of $H_1$,$H_2$ and $H_3$, and say $[H_i,x]=\lambda_ix$ for some $\lambda_i\in\mathbb{Z}$, $i=1,2,3$. Then, it is easy to check that $\lambda_1 \equiv \lambda_2\equiv \lambda_3 \pmod{2}$, so $n_1=(\lambda_2+\lambda_3)/2$, $n_2=(\lambda_1+\lambda_3)/2$, $n_3=(\lambda_1+\lambda_2)/2$ are integers, and setting $\deg(x)=(n_1,n_2,n_3)$ produces exactly the grading $\gr{\IC}_D(e_1,e_2,e_3,e_4)$.
    The universal group has no torsion, so there are no almost fine proper coarsenings, and any homomorphism $\alpha:\mathbb{Z}^3\rightarrow G$ is admissible. 
    The grading $\gr{\IC}_D(h_1,h_2,h_3,h_4)$ is induced by setting $\alpha(e_i)=h_i$, $i=1,2,3$.
    \end{proof}

    \begin{proposition}\label{prop:IC_weyl}
        If $|\stab(\alpha)|=1$, then the Weyl group of a fine grading on $D(\alpha)$ with universal group $\mathbb{Z}^3$ is isomorphic to $C_2^3$ and acts on $\mathbb{Z}^3$ by changing signs of the elements $\alpha_1=e_2+e_3$, $\alpha_2=e_1+e_3$ and $\alpha_3=e_1+e_2$.
    \end{proposition}
    
    \begin{proof}
        Consider the exact sequence $1\rightarrow \bar{K}\rightarrow W(\Gamma)\rightarrow W(\Gamma_0)$ induced by the restriction map as in \Cref{Misha's lemma}. The Weyl group $W(\Gamma_0)$ is isomorphic to $C_2^3\rtimes S_3$ where the $S_3$ subgroup acts by permuting the three ideals. But the automorphisms of $\mathcal{L}$ cannot permute the ideals, so the image of the restriction of $W(\Gamma)$ is contained $C_2^3$. Moreover, $\bar{K}$ is trivial, since $\mathbb{Z}^3$ does not have elements of order $2$. Therefore, $W(\Gamma)$ is generated by the images of $\iota(C,1,1),\iota(1,C,1),\iota(1,1,C)\in\aut(\mathcal{L})$. Their action inverts each of the roots $\alpha_1,\alpha_2,\alpha_3$ of the three copies of $\mathfrak{sl}_2$, which determines the action on $\mathbb{Z}^3$.
    \end{proof}

    \begin{remark}
        Since $e_1=\frac12(-\alpha_1+\alpha_2+\alpha_3)$, $e_2=\frac12(\alpha_1-\alpha_2+\alpha_3)$, $e_3=\frac12(\alpha_1+\alpha_2-\alpha_3)$, the above action of $C_2^3$ on $\mathbb{Z}^3$, when restricted to the elements $\pm e_1,\pm e_2,\pm e_3,\pm e_4$, corresponds to the permutation action on their indices by the Klein group $V_4\triangleleft S_4$, combined with the simultaneous change of sign.
    \end{remark}

    \begin{definition}[Gradings of type $\IP$] \label{def:IP}
        Let $a,b\in G$ be elements of order $2$ and $s\in G$ an element of order $4$ such that $\langle s,a,b\rangle\cong\mathbb{Z}_4\times\mathbb{Z}_2^2$. We define $\gr{\IP}_D(s,a,b)$ to be the grading which is $\gr{\Pa}_{\mathfrak{sl}_2}(a,s^2a)$ on $\mathfrak{s}_1$, $\gr{\Pa}_{\mathfrak{sl}_2}(b,s^2b)$ on $\mathfrak{s}_2$, and $\gr{\Pa}_{\mathfrak{sl}_2}(ab,s^2ab)$ on $\mathfrak{s}_3$. To define it on $\mathcal{L}_1$, consider the subgroup $\langle a,b\rangle\cong\mathbb{Z}_2^2$ which has the four characters $\chi_0,\chi_1,\chi_2,\chi_3$, where $\chi_0$ is the trivial character, $\chi_1(a)=\chi_2(b)=1$, $\chi_1(b)=\chi_2(a)=-1$, and $\chi_3=\chi_1\chi_2$. Then, for each $h\in\langle a,b\rangle$, take as homogeneous elements the vectors $\sum_{i=0}^3 \chi_i(h)w_i$ of degree $sh$ and $\sum_{i=0}^3 \chi_i(h)w'_i$ of degree $s^{-1}h$.
    \end{definition}

    \begin{proposition} \label{prop:IP_fine}
        A fine grading on $D(\alpha)$ with universal group $\mathbb{Z}_4\times\mathbb{Z}_2^2$ is equivalent to $\gr{\IP}_D(e_0,e_1,e_2)$ where $e_0=(1,0,0)$, $e_1=(0,1,0)$, $e_2=(0,0,1)$. It has no almost fine proper coarsenings. Moreover, the $G$-grading of type $\IP$ are precisely the gradings induced by admissible homomorphisms $\mathbb{Z}_4\times\mathbb{Z}_2^2\rightarrow G$.
    \end{proposition}
    
    \begin{proof}
        Up to equivalence, the fine $\mathbb{Z}_4\times\mathbb{Z}_2^2$-grading corresponds to the maximal quasitorus $Q\subseteq \aut(D(\alpha))$ generated by $\iota(A,A,A),\iota(A,B,B),\iota(B,B,A)$ (see \cite[\S 7.1]{DEM}). Then, diagonalizing $Q$, we obtain as eigenvectors the homogeneous elements in Definition~\ref{def:IP} for the grading $\gr{\IP}_D(e_0,e_1,e_2)$. 
        Moreover, $\supp(\mathcal{L}_0)=(\{0,2\}\times \mathbb{Z}_2^2) \smallsetminus \{(0,0,0)\}$ and $\supp(\mathcal{L}_1)=\{1,3\}\times \mathbb{Z}_2^2$.
        For any nontrivial subgroup of $\mathbb{Z}_4\times\mathbb{Z}_2^2$, the intersection with $\supp(\mathcal{L}_0)$ is not contained in $\{e\}$, so there are no almost fine proper coarsenings.
        Finally, a homomorphism  $\mathbb{Z}_4\times\mathbb{Z}_2^2\rightarrow G$ is admissible if and only if it is injective. 
    \end{proof}

    \begin{proposition}\label{prop:IP_weyl}
        If $|\stab(\alpha)|=1$, then the Weyl group of a fine grading on $D(\alpha)$ with universal group $\mathbb{Z}_4\times\mathbb{Z}_2^2$ is isomorphic to $C_2^5$ and acts on $\mathbb{Z}_4\times\mathbb{Z}_2^2$ as the set of all automorphisms that preserve the three subgroups $\langle 2e_0, e_1\rangle$, $\langle 2e_0, e_2\rangle$, $\langle 2e_0, e_1+e_2\rangle$.
    \end{proposition}
    \begin{proof}
        The Weyl group must fix the degree, $2e_0$, of the unique 3-dimensional homogeneous component $\langle C_1,C_2,C_3\rangle$. Since automorphisms cannot permute the three ideals, the Weyl group must leave invariant their supports, which are exactly the three subgroups $\langle 2e_0, e_1\rangle$, $\langle 2e_0, e_2\rangle$, $\langle 2e_0, e_1+e_2\rangle$ with $0$ removed.
        The subgroup generated by the elements $\iota(A,1,1)$, $\iota(B,1,1)$ and $\iota(1,1,A)$ fix $e_1$ and $e_2$ and replace $e_0$ with an arbitrary element of order 4, while the automorphisms $\iota(Y_{ab},1,Y_{ab})$ and $\iota(1,Y_{ab},Y_{ab})$ fix $e_0$ and respectively swap $e_1$ and $2e_0+e_1$ or $e_2$ and $2e_0+e_2$.
        We conclude that $W(\Gamma)\cong C_2^3\times C_2^2\cong C_2^5$.
    \end{proof}

    \begin{definition}[Gradings of type $\IM{i}$] \label{def:IM}
        Let $g\in G$ and $T=\langle a,b\rangle\cong\mathbb{Z}_2^2$ be a subgroup of $G$. For each $i\in\{1,2,3\}$, we define the grading $\gr{\IM{i}}_D(g,T)$ as follows. The value of $i$ indicates which of the three simple ideals of $\mathcal{L}_0$ has an elementary grading. If, say, $i=3$, then the grading on the distinguished ideal $\mathfrak{s}_3$ is  $\gr{\mathrm{C}}_{\mathfrak{sl}_2}(g^2)$ and the grading on each of $\mathfrak{s}_1$ and $\mathfrak{s}_2$ is $\gr{\Pa}_{\mathfrak{sl}_2}(a,b)$. To define the grading on $\mathcal{L}_1$, note that, since $\psi$ gives an isomorphism $V\cong V^*$ as $\mathfrak{sl}_2$-modules, we also have $V_1\otimes V_2 \cong\mathrm{End}(V)$ as $\mathfrak{s}_1\oplus\mathfrak{s}_2$-modules, via $x\otimes y\mapsto x\psi(y,\cdot)$, where the action is given by $(X,Y).M=XM-MY$ for all $X,Y\in \mathfrak{sl}_2$ and $M\in \mathrm{End}(V)\cong M_2(\mathbb{F})$. Then the grading on $\mathcal{L}_1\cong M_2(\mathbb{F})\otimes V_3$ is the tensor product of $\gr{\Pa}_{M_2}(a,b)$ and $\Gamma_V(g,g^{-1})$ on $V_3$ (i.e., $\deg(u_3)=g$ and $\deg(v_3)=g^{-1}$).
    \end{definition}

    We remark that we slightly abused  notation by denoting the grading $\gr{\IM{i}}_D(g,T)$ instead of $\gr{\IM{i}}_D(g,a,b)$, since it depends on the choice of generators of $T$. However, we will see later that different choices give isomorphic gradings.

    \begin{proposition}\label{prop:IM_fine}
        A fine grading on $D(\alpha)$ with universal group $\mathbb{Z}\times\mathbb{Z}_2^2$ is of type $\IM{i}$ for some $i=1,2,3$, depending on which ideal of $\mathcal{L}_0$ has Cartan grading, and is equivalent to $\gr{\IM{i}}_D(e_0,E)$ for the appropriate $i$, where $E=\langle e_1,e_2\rangle$ and $e_0=(1,0,0)$, $e_1=(0,1,0)$, $e_2=(0,0,1)$. It has no almost fine proper coarsenings. Moreover, the $G$-gradings of type $\IM{i}$ are precisely the ones induced by admissible homomorphisms $\mathbb{Z}\times\mathbb{Z}_2^2\rightarrow G$.
    \end{proposition}
    \begin{proof}
        We prove the proposition for $\IM{3}$. The other two cases are analogous.

        Up to equivalence, the fine $\mathbb{Z}\times\mathbb{Z}_2^2$-grading corresponds to the maximal quasitorus $Q\subseteq \aut(D(\alpha))$ generated by $\{\iota(1,1,d_\mu)\mid\mu\in\mathbb{F}^{\times}\}\cup\{\iota(A,A,1),\iota(B,B,1)\}$ where $d_\mu = \diag(\mu,\mu^{-1})$ (see \cite[\S 7.4]{DEM}). Then, diagonalizing $Q$, we obtain as eigenvectors the homogeneous elements in Definition~\ref{def:IM} for the grading $\gr{\IM{3}}_D(e_0,E)$. 
        The supports are $\supp(\Gamma_0)=\{(\pm2,0,0)\}\cup(\{0\}\times\mathbb{Z}_2^2)$ and $\supp(\Gamma_1)=\{\pm1\}\times \mathbb{Z}_2^2$. Since the torsion subgroup is contained in $\supp(\Gamma_0)$, there are no almost fine proper coarsenings.
        Finally, a homomorphism $\mathbb{Z}\times\mathbb{Z}_2^2\rightarrow G$ is admissible if and only if its restriction to the torsion subgroup is injective. 
    \end{proof}

    \begin{proposition}\label{prop:IM_weyl}
        The Weyl group of a fine grading on $D(\alpha)$ with universal group $\mathbb{Z}\times\mathbb{Z}_2^2$ is $\aut(\mathbb{Z}\times\mathbb{Z}_2^2)\cong C_2\times(C_2^2 \rtimes S_3)$.
    \end{proposition}
    \begin{proof}
        Using the automorphisms $\iota(1,A,1)$, $\iota(1,B,1)$ and $\iota(1,1,C)$, we can move $e_0$ to any element whose image generates the quotient by the torsion subgroup. Also, the subgroup $\{\iota(Y_\sigma,Y_\sigma,1)\mid\sigma\in S_3\}$ acts as the full automorphism group of the torsion subgroup. It easy to check that the described elements generate the entire $\aut(\mathbb{Z}\times\mathbb{Z}_2^2)\cong C_2\times(C_2^2 \rtimes S_3)$.
    \end{proof}

    \begin{theorem}\label{thm:D1}
        Let $\mathcal{L}=D(\alpha)$ with $|\stab(\alpha)|=1$ and let $G$ be an abelian group. Then every $G$-grading on $\mathcal{L}$ is of one of the types $\IC,\IP,\IM{1},\IM{2},\IM{3}$ (see Definitions \ref{def:IC}, \ref{def:IP}, \ref{def:IM}). Gradings of different types are not isomorphic and, within each type, the isomorphism classes are as follows (summarized in Table \ref{tab:D(alpha)}):
		\begin{itemize}
			\item $\gr{\IC}_D(h_1,h_2,h_3,h_4)\cong \gr{\IC}_D(h'_1,h'_2,h'_3,h'_4)$ if and only if there exist  $\sigma\in V_4\triangleleft S_4$ and $\varepsilon\in\{\pm1\}$ such that $h'_i=h_{\sigma(i)}^\varepsilon$ for all $i$,
			\item $\gr{\IP}_D(s,a,b)\cong \gr{\IP}_D(s',a',b')$ if and only if $\langle s,a,b \rangle=\langle s',a',b'\rangle$, $a'\in a\langle s^2\rangle$ and $b'\in b\langle s^2\rangle$,
			\item $\gr{\IM{n}}_D(g,T)\cong \gr{\IM{n}}_D(g',T')$ 
            if and only if $T'=T$ and $g'\in gT\cup g^{-1}T$. 
		\end{itemize}
    \end{theorem}
    \begin{proof}
        Recall that, by \cite{DEM}, if $|\stab(\alpha)|=1$ then any fine grading on $D(\alpha)$ is equivalent to a $\mathbb{Z}^3$-grading, a $\mathbb{Z}_4\times\mathbb{Z}_2^2$-grading, or one of the three  $\mathbb{Z}\times\mathbb{Z}_2^2$-gradings. Therefore, by \Cref{thm:gradings} and Propositions \ref{prop:IC_fine}, \ref{prop:IP_fine} and \ref{prop:IM_fine}, any $G$-grading is isomorphic to one of the types $\IC,\IP,\IM{1},\IM{2},\IM{3}$, and gradings of different types are not isomorphic. By the same theorem and Propositions \ref{prop:IC_weyl}, \ref{prop:IP_weyl} and \ref{prop:IM_weyl}, the action of the Weyl group in each type leads to the stated conditions.
    \end{proof}

    \begin{table}
    \begin{tabular}{|c|c|c|c|}
    \hline
    Family & Subgroups & Combinatorial data & Equivalence\\ \hline 
    $\gr{\IC}_D$ 
    & trivial & \vtop{\hbox{\strut $(h_1,h_2,h_3,h_4)\in G^4$}\hbox{\strut with $h_1h_2h_3h_4=e$}} & \vtop{\hbox{\strut inversion and}\hbox{\strut permutation} \hbox{\strut by $V_4\triangleleft S_4$}} \\ \hline 
    $\gr{\IP}_D$ 
    &  $H=\langle s\rangle\times\langle a,b\rangle\cong\mathbb{Z}_4\times\mathbb{Z}_2^2$ & \vtop{\hbox{\strut $(aS,bS)\in (H_{[2]}/S)^2$}\hbox{\strut where $S=\langle s^2\rangle$}} & \\ \hline 
     $\gr{\IM{n}}_D$ 
    & $T\cong \mathbb{Z}_2^2$ & $gT\in G/T$ & inversion \\ $n=1,2,3$ & & & \\ \hline 
    \end{tabular}
    \vspace{5pt} 
    \caption{Classification of $G$-gradings on the Lie superalgebra $D(\alpha)$ with $\alpha\notin\{-2,-1,0,\frac{1}{2},1,\omega,\omega^2\}$}\label{tab:D(alpha)}
    \end{table}
	
\subsection{Orbit of size $2$}
	
    The only orbit of size $2$ is $\{\omega,\omega^2\}$, where $\omega$ a primitive cube root of unity, so without loss of generality we can assume $\alpha=\omega$. In this case, we have the outer automorphism $\varphi:\mathcal{L}\to \mathcal{L}$ defined by $(x_1,x_2,x_3)\in \mathcal{L}_0 \mapsto (x_2,x_3,x_1)$ and $x_1\otimes x_2\otimes x_3\in \mathcal{L}_1\mapsto \omega^{-1}x_2\otimes x_3\otimes x_1$. Note that $\varphi^3=1$.
    
    By \cite{S}, $\aut(D(\omega))\cong\inn(D(\omega))\rtimes C_3$, and $\langle\varphi\rangle\cong C_3$ gives a section of $\aut(\mathcal{L})\to\out(\mathcal{L})$. For any $f,g,h\in\slg_2(\mathbb{F})$, we have $\varphi\iota(f,g,h)\varphi^{-1}=\iota(g,h,f)$.

    We already described the inner gradings (i.e., the $G$-gradings for which the image of $\widehat{G}$ in $\aut(\mathcal{L})$ consists of inner automorphisms) in Definitions \ref{def:IC}, \ref{def:IP} and \ref{def:IM}, but the outer automorphisms can induce some additional isomorphisms between them. In fact, for the fine gradings of type $\IC$ and $\IP$ (Cartan and Pauli), we obtain larger Weyl groups, since in both cases $\varphi$ permutes the homogeneous components, so the Weyl group is enlarged by $\langle\varphi\rangle\cong C_3$. This gives the following two results:

    \begin{proposition}\label{prop:IC_weyl3}
        The Weyl group of a fine grading on $D(\omega)$ with universal group $\mathbb{Z}^3$ is isomorphic to $C_2^3\rtimes C_3$ and acts on $\mathbb{Z}^3$ by changing signs of the elements $\alpha_1=e_2+e_3$, $\alpha_2=e_1+e_3$ and $\alpha_3=e_1+e_2$ or permuting them cyclically.\qed 
    \end{proposition}

    \begin{remark}
        The above action, when restricted to the elements $\pm e_1, \pm e_2, \pm e_3, \pm e_4$, corresponds to the permutation action on their indices by the alternating group $A_4\triangleleft S_4$, combined with the  simultaneous change of sign.
    \end{remark}
    
    \begin{proposition}\label{prop:IP_weyl3}
        The Weyl group of a fine grading on $D(\omega)$ with universal group $\mathbb{Z}_4\times\mathbb{Z}_2^2$ is isomorphic to $C_2^5\rtimes C_3$ and acts on $\mathbb{Z}_4\times\mathbb{Z}_2^2$ as the set of all automorphisms that permute cyclically the three subgroups $\langle 2e_0, e_1\rangle$, $\langle 2e_0, e_2\rangle$, $\langle 2e_0, e_1+e_2\rangle$.\qed 
    \end{proposition}
    
    In the case of gradings of mixed type, $\varphi$ induces isomorphisms between the three fine gradings of Proposition~\ref{prop:IM_fine}. Therefore, there exists only one mixed fine grading up to equivalence, say $\IM{1}$. Its Weyl group is the same as described in Proposition~\ref{prop:IM_weyl}. We now turn to outer gradings. In the next definition, the Roman numeral III represents the size of the image of $\widehat{G}$ in $\out(\mathcal{L})$.
	
    \begin{definition}[Gradings of type $\III$] \label{def:IIIP}
        Let $g,h\in G$ with $h$ of order 3. We define the grading $\gr{\III}_D(g,h)$ in the following way. For each $X\in \mathfrak{sl}_2$ and $i=0,1,2$, we set $X^{(i)}=(X,\omega^iX,\omega^{2i}X)\in\mathcal{L}_0$. Then the grading on $\mathcal{L}_0$ is given by $\deg(H^{(i)})= h^i$, $\deg(E^{(i)})= g^2h^i$ and $\deg(F^{(i)})= g^{-2}h^i$. In $\mathcal{L}_1$, for each $i=0,1,2$, we define the vectors $w^{(i)}=w_1+\omega^iw_2+\omega^{2i}w_3$ and $w'^{(i)}=w'_1+\omega^iw'_2+\omega^{2i}w'_3$. Then the grading on $\mathcal{L}_1$ is given by $\deg(w_0)=g^{-3}h^2$, $\deg(w'_0)=g^{3}h^2$, $\deg(w^{(i)})=gh^{i-1}$ and $\deg(w'^{(i)})=g^{-1}h^{i-1}$.
    \end{definition}

    \begin{proposition} \label{prop:III_fine}
        A fine grading on $D(\omega)$ with universal group $\mathbb{Z}\times\mathbb{Z}_3$ is equivalent to $\gr{\III}_D(e_0,e_1)$, of type $(17)$, where $e_0=(1,0)$, $e_1=(0,1)$. It has no almost fine proper coarsenings. Moreover, the $G$-gradings of type $\III$ are precisely the ones induced by admissible homomorphisms $\mathbb{Z}\times\mathbb{Z}_3\rightarrow G$.
    \end{proposition}
    \begin{proof}
        Up to equivalence, the fine $\mathbb{Z}\times\mathbb{Z}_3$-grading corresponds to the maximal quasitorus $Q\subseteq \aut(D(\alpha))$ generated by $\{ \iota(d_\mu,d_\mu,d_\mu)\mid \mu\in\mathbb{F}^\times\}$ and $\varphi$ (see \cite[\S 8.1]{DEM}). Then, diagonalizing $Q$, we obtain as eigenvectors the homogeneous elements in Definition~\ref{def:III} for the grading $\gr{\III}_D(e_0,e_1)$. 
        Moreover, $\supp(\Gamma_0)=\{0,\pm2\}\times\mathbb{Z}_3$ and $\supp(\Gamma_1)=(\{\pm 1\}\times\mathbb{Z}_3)\cup\{(\pm3,2)\}$, so there are no almost fine proper coarsenings.
        A homomorphism $\mathbb{Z}\times\mathbb{Z}_3\rightarrow G$ is admissible if and only if its restriction to the torsion subgroup is injective.
    \end{proof}
    
    \begin{proposition}\label{prop:III_weyl}
        The Weyl group of a fine grading on $D(\omega)$ with universal group $\mathbb{Z}\times\mathbb{Z}_3$ is isomorphic to $C_3\rtimes C_2\cong S_3$ and acts on $\mathbb{Z}\times\mathbb{Z}_3$ as the set of all  automorphisms that fix the torsion subgroup point-wise.
    \end{proposition}
    \begin{proof}
        The degree $0$ component must be invariant under the automorphism group of the grading. Since it is spanned by $H^{(0)}$, the eigenspaces defined by the adjoint action of $H^{(0)}$ must be preserved or interchanged with their opposites. This gives a distinguished pair of components, namely, the ones spanned by $w_0$ and $w'_0$, which correspond to the extreme eigenvalues of $H^{(0)}$. 
        Since $[w_0,w'_0]$ is a nonzero element of degree $e_1=(0,1)$, we conclude that $e_1$ is fixed by the Weyl group. 

        Conversely, every automorphism of $\mathbb{Z}\times\mathbb{Z}_3$ fixing $e_1$ is determined by the image of $e_0=(1,0)$. The automorphism $\iota(I,d_{\omega^2},d_\omega)$ sends $e_0$ to $(1,1)$ and $\iota(S,S,S)$ sends $e_0$ to $(-1,0)$. These automorphisms generate the desired subgroup, since $e_1$ generates the torsion subgroup.
    \end{proof}

    \begin{theorem}\label{thm:D2}
        Let $\mathcal{L}=D(\omega)$ and let $G$ be an abelian group. Then every $G$-grading on $\mathcal{L}$ is of one of the types $\IC,\IP,\IM{1},\III$ (see Definitions \ref{def:IC}, \ref{def:IP}, \ref{def:IM}, \ref{def:IIIP}). Gradings of different types are not isomorphic and, within each type, the isomorphism classes are as follows (summarized in Table \ref{tab:D(omega)}):
	\begin{itemize}
            \item $\gr{\IC}_D(h_1,h_2,h_3,h_4)\cong \gr{\IC}_D(h'_1,h'_2,h'_3,h'_4)$ if and only if there exist $\sigma\in A_4\triangleleft S_4$ and $\varepsilon\in\{\pm1\}$ such that $h'_i=h_{\sigma(i)}^\varepsilon$ for all $i$,
			
            \item $\gr{\IP}_D(s,a,b)\cong \gr{\IP}_D(s',a',b')$ if and only if $\langle s,a,b \rangle=\langle s',a',b'\rangle$ and there exists a cyclic permutation $\sigma$ of the triple $(a,b,ab)$ such that $a'\in \sigma(a)\langle s^2\rangle$ and $b'\in\sigma(b)\langle s^2\rangle$, 
			
            \item $\gr{\IM{1}}_D(g,T)\cong \gr{\IM{1}}_D(g',T')$ if and only if $T'=T$ and $g'\in gT\cup g^{-1}T$, 
		
            \item $\gr{\III}_D(g,h)\cong \gr{\III}_D(g',h')$ if and only if $h'=h$ and $g'\in g\langle h\rangle\cup g^{-1}\langle h\rangle$.
	\end{itemize}
    \end{theorem}
    \begin{proof}
        By \cite{DEM}, any fine gradings on $D(\omega)$ is equivalent to one of the above fine gradings with universal groups $\mathbb{Z}^3$, $\mathbb{Z}_4\times\mathbb{Z}_2^2$, $\mathbb{Z}\times\mathbb{Z}_2^2$ and $\mathbb{Z}\times\mathbb{Z}_3$. Therefore, by \Cref{thm:gradings} and Propositions \ref{prop:IC_fine}, \ref{prop:IP_fine}, \ref{prop:IM_fine} and \ref{prop:III_fine}, any $G$-grading is isomorphic to one of the types $\IC,\IP,\IM{1},\III$, 
        and gradings of different types are not isomorphic. By the same theorem and Propositions \ref{prop:IC_weyl3}, \ref{prop:IP_weyl3}, \ref{prop:IM_weyl} and \ref{prop:III_weyl}, the action of the Weyl group in each type leads to the stated conditions.
    \end{proof}

    \begin{table}
    \begin{tabular}{|c|c|c|c|}
    \hline
    Family & Subgroups & Combinatorial data & Equivalence\\ \hline 
    $\gr{\IC}_D$ 
    & trivial & \vtop{\hbox{\strut $(h_1,h_2,h_3,h_4)\in G^4$}\hbox{\strut with $h_1h_2h_3h_4=e$}} & \vtop{\hbox{\strut inversion and}\hbox{\strut permutation} \hbox{\strut by $A_4\triangleleft S_4$}}  \\ \hline 
    $\gr{\IP}_D$ 
    &  $H=\langle s\rangle\times\langle a,b\rangle\cong\mathbb{Z}_4\times\mathbb{Z}_2^2$ & \vtop{\hbox{\strut cyclic order on the three}\hbox{\strut nontrivial elements of}\hbox{\strut $H_{[2]}/\langle s^2\rangle\cong\mathbb{Z}_2^2$}} &
    \\ \hline 
    $\gr{\IM{1}}_D$ 
    & $T\cong \mathbb{Z}_2^2$ & $gT\in G/T$ & inversion \\ \hline 
    $\gr{\III}_D$ 
    & $H=\langle h\rangle \cong \mathbb{Z}_3$ & $gH\in G/H$ & inversion \\ \hline
    \end{tabular}
    \vspace{5pt} 
    \caption{Classification of $G$-gradings on the Lie superalgebra $D(\omega)$}\label{tab:D(omega)}
    \end{table}
	
\subsection{Orbits of size $3$}

    There are two orbits of size $3$, $\{1,-2,-\frac{1}{2}\}$ and $\{-1,0,\infty\}$, so without loss of generality we may assume $\alpha=\pm1$. Then the group of outer automorphisms is isomorphic to $C_2$ and permutes the first two ideals of $\mathcal{L}_0$. As mentioned before, the Lie superalgebra $D(-1)$ is not simple, so we postpone its discussion until the next section. Here we focus on $D(1)\cong\mathfrak{osp}(4\,|\,2)$.
    
    In this case, we have the outer automorphism $\varphi:\mathcal{L}\to \mathcal{L}$ defined by $(x_1,x_2,x_3)\in \mathcal{L}_0 \mapsto (x_2,x_1,x_3)$ and $x_1\otimes x_2\otimes x_3\in \mathcal{L}_1\mapsto  x_2\otimes x_1\otimes x_3$. Note that $\varphi^2=1$, so $\langle\varphi\rangle\cong C_2$ gives a section of $\aut(\mathcal{L})\to\out(\mathcal{L})$. Also, $\varphi\iota(f,g,h)\varphi^{-1}=\iota(g,f,h)$ for any $f,g,h\in\slg_2(\mathbb{F})$.
    
    Again, the inner gradings are given in Definitions \ref{def:IC}, \ref{def:IP} and \ref{def:IM}, and we need to check the additional isomorphisms between them induced by the outer automorphisms. Since $\varphi$ permutes the homogeneous components for the fine gradings of type $\IC$ and $\IP$, we obtain the following Weyl groups.

    \begin{proposition}\label{prop:IC_weyl2}
        The Weyl group of a fine grading on $D(1)$ with universal group $\mathbb{Z}^3$ is isomorphic to $C_2^3\rtimes C_2$ and acts on $\mathbb{Z}^3$ by changing signs of the elements $\alpha_1=e_2+e_3$, $\alpha_2=e_1+e_3$ and $\alpha_3=e_1+e_2$ or interchanging $\alpha_1$ and $\alpha_2$.\qed 
    \end{proposition}

    \begin{remark}
        The above action, when restricted to the elements $\pm e_1, \pm e_2, \pm e_3, \pm e_4$, corresponds to the  permutation action by the dihedral group $D_4\leq S_4$ preserving the partition $\{\{1,2\},\{3,4\}\}$, combined with the simultaneous change of sign.
    \end{remark}
    
    \begin{proposition}\label{prop:IP_weyl2}
        The Weyl group of a fine grading on $D(1)$ with universal group $\mathbb{Z}_4\times\mathbb{Z}_2^2$ is isomorphic to $C_2^5\rtimes C_2$ and acts on $\mathbb{Z}_4\times\mathbb{Z}_2^2$ as the set of all automorphisms that either preserve the three subgroups $\langle 2e_0, e_1\rangle$, $\langle 2e_0, e_2\rangle$, $\langle 2e_0, e_1+e_2\rangle$ or interchange the first two.\qed 
    \end{proposition}

    \begin{proposition}\label{prop:IM_fine2}
        A fine grading on $D(1)$ with universal group $\mathbb{Z}\times\mathbb{Z}_2^2$ is equivalent to $\gr{\IM{1}}_D(e_0,E)$ where $E=\langle e_1,e_2\rangle$ and $e_0=(1,0,0)$, $e_1=(0,1,0)$, $e_2=(0,0,1)$. It has no almost fine proper coarsenings. Moreover, the $G$-gradings of type $\IM{1}$ are precisely the ones induced by admissible homomorphisms $\mathbb{Z}\times\mathbb{Z}_2^2\rightarrow G$.
    \end{proposition}
    \begin{proof}
        Follows from Proposition \ref{prop:IM_fine}, taking into account that $\varphi$ is an equivalence between the mixed gradings of types $\IM{1}$ and $\IM{2}$, while type $\IM{3}$ is not fine.
    \end{proof}

    The outer automorphism $\varphi$ refines $\gr{\IM{3}}_D(e_0,E)$, leading to the following family of gradings (in fact, the union of two families: outer for $h\ne e$ and inner for $h=e$):
    
    \begin{definition}[Gradings of type $\IIM$] \label{def:IIM}
        Let $T=\{e,a,b,c\}\cong\mathbb{Z}_2^2$ be a subgroup of $G$ and let $g,h\in G$ such that $h^2=e$ and $h\notin\{a,b,c\}$. For each $X\in \mathfrak{sl}_2$ and $i=0,1$, we set $X^{(i)}=(X,(-1)^iX)\in\mathfrak{sl}_2\oplus\mathfrak{sl}_2$. Then we define $\gr{\IIM}_D(g,h,T)$ as follows. On the distinguished ideal $\mathfrak{s}_3$ of $\mathcal{L}_0$, we put the grading $\gr{\mathrm{C}}_{\mathfrak{sl}_2}(g^2)$ while, on $\mathfrak{s}_1\oplus\mathfrak{s}_2$, we set $\deg(A^{(i)})=ah^i$, $\deg(B^{(i)})=bh^i$ and $\deg(C^{(i)})=ch^i$. To define the grading on $\mathcal{L}_1$, we use the isomorphism $\mathcal{L}_1\cong M_2(\mathbb{F})\otimes V_3$ described in Definition \ref{def:IM} and take the tensor product of $\Gamma_V(g,g^{-1})$ on $V_3$ 
        and the grading on $M_2(\mathbb{F})$ given by $\deg(I)=h$, $\deg(A)=a$, $\deg(B)=b$, $\deg(C)=c$. 
    \end{definition}

    \begin{proposition}\label{prop:IIM_fine}
        A fine grading on $D(1)$ with universal group $\mathbb{Z}\times\mathbb{Z}_2^3$ is equivalent to $\gr{\IIM}_D(e_0,e_1,E)$, of type $(17)$, where $E=\langle e_2,e_3\rangle$ and $e_0=(1,0,0,0)$, $e_1=(0,1,0,0)$, $e_2=(0,0,1,0)$, $e_3=(0,0,0,1)$.
        It has an almost fine proper coarsening isomorphic to $\gr{\IIM}_D(e_0,0,E) = \gr{\IM{3}}_D(e_0,E)$, with universal group $\mathbb{Z}\times\mathbb{Z}_2^2$.
        Moreover, the $G$-gradings of type $\IIM$ are precisely the ones induced by admissible homomorphisms $\mathbb{Z}\times\mathbb{Z}_2^3\rightarrow G$ and $\mathbb{Z}\times\mathbb{Z}_2^2\rightarrow G$.
    \end{proposition}
    \begin{proof}
        In \cite[\S 8.2]{DEM}, a fine grading with universal group $\mathbb{Z}\times\mathbb{Z}_2^3$ is obtained as a refinement of a grading of type $\IM{3}$. If we refine $\gr{\IM{3}}_D(e_0,E)$ according to the eigenspace decomposition of $\varphi$, we get exactly $\gr{\IIM}_D(e_0,e_1,E)$.

        The supports are $\supp(\Gamma_0)=\{(\pm2,0,0,0)\}\cup(\{0\}\times\mathbb{Z}_2^3)\smallsetminus \{(0,1,0,0)\}$ and $\supp(\Gamma_1)=\{\pm1\}\times \{(1,0,0),(0,1,0),(0,0,1),(0,1,1)\}$. Thus, the only nontrivial subgroup of the torsion part whose intersection with $\supp(\Gamma_0)$ is contained in $\{0\}$ is $\langle e_1\rangle$. Taking the quotient by this subgroup, we obtain $\gr{\IM{3}}_D(e_0,E)$ as the only proper almost fine coarsening. 

        A homomorphism $\alpha:\mathbb{Z}\times\mathbb{Z}_2^3\rightarrow G$ or $\bar{\alpha}:\mathbb{Z}\times\mathbb{Z}_2^2\rightarrow G$ is admissible if and only if its restriction to the torsion part is injective. The gradings $\gr{\IIM}_D(g,h,H)$ are induced by such homomorphisms: $\alpha$ if $h\ne e$ and $\bar{\alpha}$ if $h=e$.
    \end{proof}

    \begin{proposition}\label{prop:IIM_weyl}
        Let $W(\IIM)$ be the Weyl group of a fine grading on $D(1)$ with universal group $\mathbb{Z}\times\mathbb{Z}_2^3$ and let $W(\IM{3})$ be the Weyl group of its almost fine coarsening induced by the quotient map $\mathbb{Z}\times\mathbb{Z}_2^3\to\mathbb{Z}\times\mathbb{Z}_2^2$. 
        Then $W(\IIM)$ leaves the kernel of this map invariant, and the resulting map $W(\IIM)\rightarrow W(\IM{3})$ is an isomorphism. Moreover,  $W(\IM{3})$ acts on $\mathbb{Z}\times\mathbb{Z}_2^2$ as the full automorphism group, which is isomorphic to $C_2\times(C_2^2 \rtimes S_3)$.
    \end{proposition}
    \begin{proof}
        By Proposition \ref{prop:IM_weyl}, $W(\IM{3})$ is the full automorphism group of $\mathbb{Z}\times\mathbb{Z}_2^2$.
        An automorphism of the algebra stabilizing a grading also stabilizes all its coarsenings, so the map $W(\IIM)\rightarrow W(\IM{3})$ is well defined. It is surjective since the automorphisms described in Proposition \ref{prop:IM_weyl} are in $\aut(\IIM)$, too. It is injective since an automorphism is completely determined by its action on $\mathcal{L}_1$ and on this space the two gradings coincide.
    \end{proof}	

    \begin{definition}[Gradings of type $\IIC$] \label{def:IIC}
        Let $g_1,g_2,h\in G$ with $h$ of order $2$. For each $X\in \mathfrak{sl}_2$ and $i=0,1$, we set $X^{(i)}=(X,(-1)^iX)\in\mathfrak{sl}_2\oplus\mathfrak{sl}_2$. Then we define $\gr{\IIC}_D(g_1,g_2,h)$ as follows. On the distinguished ideal $\mathfrak{s}_3$ of $\mathcal{L}_0$, we put the grading $\gr{\mathrm{C}}_{\mathfrak{sl}_2}(g_1^2)$ while, on $\mathfrak{s}_1\oplus\mathfrak{s}_2$, we set $\deg(H^{(i)})=h^i$, $\deg(E^{(i)})=g_2h^i$ and $\deg(F^{(i)})=g_2^{-1}h^i$. To define the grading on $\mathcal{L}_1$, we use the isomorphism $\mathcal{L}_1\cong M_2(F)\otimes V_3$ described in Definition \ref{def:IM} and take the tensor product of $\Gamma_V(g_1,g_1^{-1})$ on $V_3$ 
        and the grading on $M_2(\mathbb{F})$ given by $\deg(I)=h$, $\deg(H)=e$, $\deg(E)=g_2$, $\deg(F)=g_2^{-1}$.
    \end{definition}

    \begin{proposition} \label{prop:IIC_fine}
        A fine grading on $D(1)$ with universal group $\mathbb{Z}^2\times\mathbb{Z}_2$ is equivalent to $\gr{\IIC}_D(e_1,e_2,e_3)$, of type $(15,1)$, where $e_1=(1,0,0)$, $e_2=(0,1,0)$, $e_3=(0,0,1)$. It has no almost fine proper coarsenings. Moreover, the $G$-gradings of type $\IIC$ are precisely the ones induced by admissible homomorphisms $\mathbb{Z}^2\times\mathbb{Z}_2\rightarrow G$.
    \end{proposition}
    \begin{proof}         
        Up to equivalence, the fine $\mathbb{Z}^2\times\mathbb{Z}_2$-grading corresponds to the maximal quasitorus $Q\subseteq \aut(D(1))$ generated by $\{ \iota(d_\lambda,d_\lambda,d_\mu)\mid \lambda,\mu\in\mathbb{F}^\times\}$ and $\varphi$ (see \cite[\S 8.2]{DEM}). Then, diagonalizing $Q$, we obtain as eigenvectors the homogeneous elements in Definition~\ref{def:IIC} for the grading $\gr{\IIC}_D(e_1,e_2,e_3)$.
        Moreover, $\supp(\Gamma_0)=\{(\pm2,0,0)\}\cup(\{(0,0),(0,\pm2)\}\times\mathbb{Z}_2)$ and $\supp(\Gamma_1)=\{\pm1\}\times \{(0,1),(0,0),(\pm1,0)\}$, so there are no almost fine proper coarsenings.
        A homomorphism $\mathbb{Z}^2\times\mathbb{Z}_2\rightarrow G$ is admissible if and only if its restriction to the torsion subgroup is injective.
    \end{proof}

    \begin{proposition}\label{prop:IIC_weyl}
        The Weyl group of a fine grading on $D(1)$ with universal group $\mathbb{Z}^2\times\mathbb{Z}_2$ is isomorphic to $C_2^3$ and acts on $\mathbb{Z}^2\times\mathbb{Z}_2$ as the set of all automorphisms preserving the sets $\{\pm 2e_1\}$ and $\{\pm e_1\pm e_2\}$. Explicitly, it can change the sign of the first two components or add their sum to the third one (modulo 2).
    \end{proposition}
    \begin{proof}
        The set $\{\pm 2e_1\}$ is preserved because the Weyl group must preserve the support of $\mathfrak{s}_3$.
        The degree $0$ component must be invariant, and it is spanned by $H^{(0)}=H_1+H_2\in\mathfrak{s}_1\oplus\mathfrak{s}_2$ and $H_3\in\mathfrak{s}_3$. Therefore, the eigenspaces defined by the adjoint action of each of these elements must be preserved or interchanged with their opposites. 
        It follows that the set $\{\pm e_1\pm e_2\}$ is preserved.

        The automorphisms $\iota(1,1,C)$ and $\iota(C,C,1)$ change the signs of $e_1$ and $e_2$, respectively, and $\iota(A,1,1)$ sends $e_1$ to $e_1+e_3$ and $e_2$ to $e_2+e_3$.
    \end{proof}	

    \begin{definition}[Gradings of type $\IIP$] \label{def:IIP}
        Let $a,b\in G$ with $a^4=b^4=e$ and $\langle a,b\rangle \cong \mathbb{Z}_4^2$. We define $\gr{\IIP}_D(a,b)$ in the following way. On $\mathcal{L}_0$, we set $\deg(A_3)=a^2$, $\deg(B_3)=\deg(B_1+B_2)=b^2$, $\deg(C_3)=\deg(B_1-B_2)=a^2b^2$, $\deg(A_1\pm iA_2)=a^{\pm 1}$ and $\deg(C_1\pm iC_2)=a^{\pm 1}b^2$, where $i^2=-1$. To define the grading on $\mathcal{L}_1$, consider the subgroup $\langle a\rangle\cong\mathbb{Z}_4$ and its four characters, $\chi_0, \chi_1, \chi_2, \chi_3$, defined by $\chi_k(a)=i^k$. Then, for each $h\in\langle a\rangle$, we set $\deg(\chi_0(h)w_0-\chi_2(h)w_3+\chi_1(h)w_1+\chi_3(h)w_2)=bh$ and $\deg(\chi_0(h)w'_0-\chi_2(h)w'_3+\chi_1(h)w'_1+\chi_3(h)w'_2)=b^{-1}h$.
    \end{definition}

    \begin{proposition} \label{prop:IIP_fine}
        A fine grading on $D(1)$ with universal group $\mathbb{Z}_4^2$ is equivalent to $\gr{\IIP}_D(e_0,e_1)$, of type $(13,2)$, where $e_0=(1,0)$ and $e_1=(0,1)$. It has no almost fine proper coarsenings. Moreover, the $G$-gradings of type $\IIP$ are precisely the ones induced by admissible homomorphisms $\mathbb{Z}_4^2\rightarrow G$.
    \end{proposition}
    \begin{proof}
         Up to equivalence, the fine $\mathbb{Z}_4^2$-grading corresponds to the maximal quasitorus $Q\subseteq \aut(D(1))$ generated by $\iota(A,A,A)$ and $\varphi\iota(I,B,B)$ (see \cite[\S 8.2]{DEM}). Diagonalizing $Q$, we obtain as eigenvectors the homogeneous elements of Definition~\ref{def:IIP} for the grading $\gr{\IIP}_D(e_0,e_1)$.
        Moreover, $\supp(\Gamma_0)=(\mathbb{Z}_4\times\{0,2\})\smallsetminus\{(0,0)\}$ and $\supp(\Gamma_1)=\mathbb{Z}_4\times\{1,3\}$, so there are no proper almost fine coarsenings.
        A homomorphism $\mathbb{Z}_4^2\rightarrow G$ is admissible if and only if it is injective.
    \end{proof}

    \begin{proposition}\label{prop:IIP_weyl}
        The Weyl group of a fine grading on $D(1)$ with universal group $\mathbb{Z}_4^2$ is isomorphic to $C_2^4\rtimes C_2$ and acts on $\mathbb{Z}_4^2$ as the set of all automorphisms fixing the element $2e_0=(2,0)$.
    \end{proposition}
    \begin{proof}
        The grading $\gr{\IIP}_D(e_0,e_1)$ has two components of dimension $2$, namely, the ones of degrees $(0,2)$ and $(2,2)$. Their bracket is nonzero and has degree $(2,0)$, so this element is fixed by the Weyl group.

        Now we show that the Weyl group can act as every automorphism fixing $(2,0)$. In fact, consider the sequence of subgroups $K_1\subseteq K_2 \subseteq K_3\subseteq \aut(\mathbb{Z}_4^2)$ defined by $K_1=\{\eta\in\aut(\mathbb{Z}_4^2)\mid \eta(e_0)=e_0,\,\eta(2e_1)=2e_1\}$, $K_2=\{\eta\in\aut(\mathbb{Z}_4^2)\mid \eta(2e_0)=2e_0,\,\eta(2e_1)=2e_1\}$ and $K_3=\{\eta\in\aut(\mathbb{Z}_4^2)\mid \eta(2e_0)=2e_0\}$. Then $|K_1|=4$, $|K_2/K_1|=4$, $|K_3/K_2|=2$ and we want to prove that the Weyl group corresponds to $K_3$. This follows from the fact that $\iota(I,I,A),\iota(I,I,B)\in K_1$, $\iota(B,I,I),\iota(Y_{ac},Y_{ac},I)\in K_2\smallsetminus K_1$, and $\iota(1,A,Y_{bc})\in K_3\smallsetminus K_2$.
    \end{proof}

    \begin{theorem}\label{thm:D3}
		Let $\mathcal{L}=D(1)$ and let $G$ be an abelian group. Then every $G$-gradings on $\mathcal{L}$ is of one of the types $\IC,\IP,\IM{1},\IIC,\IIP,\IIM$ (see Definitions \ref{def:IC}, \ref{def:IP}, \ref{def:IM}, \ref{def:IIC}, \ref{def:IIP}, \ref{def:IIM}). Gradings of different types are not isomorphic and, within each type, the isomorphism classes are as follows (summarized in Table \ref{tab:D(1)}):
		\begin{itemize}
			\item $\gr{\IC}_D(h_1,h_2,h_3,h_4)\cong \gr{\IC}_D(h'_1,h'_2,h'_3,h'_4)$ if and only if there exist $\sigma\in D_4\leq S_4$ and $\varepsilon\in\{\pm1\}$ such that $h'_i=h_{\sigma(i)}^\varepsilon$ for all $i$, where $D_4$ is the dihedral group preserving the partition $\{\{1,2\},\{3,4\}\}$,
			
			\item $\gr{\IP}_D(s,a,b)\cong \gr{\IP}_D(s',a',b')$ if and only if $\langle s,a,b \rangle=\langle s',a',b'\rangle$ and there exists a permutation $\sigma$ of the pair $(a,b)$ such that $a'\in\sigma(a)\langle s^2\rangle$ and $b'\in\sigma(b)\langle s^2\rangle$, 
			
			\item $\gr{\IM{1}}_D(g,T)\cong \gr{\IM{1}}_D(g',T')$ if and only if $T'=T$ and $g'\in gT\cup g^{-1}T$, 
			
			\item $\gr{\IIM}_D(g,h,T)\cong \gr{\IIM}_D(g',h',T')$ if and only if $T'=T$, $h'=h$ and $g'\in gT\cup g^{-1}T$,
			
			\item  $\gr{\IIC}_D(g_1,g_2,h)\cong\gr{\IIC}_D(g'_1,g'_2,h')$ if and only if $h'=h$ and there exists $k\in\langle h\rangle$ such that $g'_i\in\{g_ik,g_i^{-1}k\}$ for $i=1,2$,
			
			\item $\gr{\IIP}_D(a,b)\cong\gr{\IIP}_D(a,b)$ if and only if $\langle a,b\rangle=\langle a',b'\rangle$ and $a^2=a'^2$.
			
		\end{itemize}
	\end{theorem}
    \begin{proof}
        By \cite{DEM}, any fine grading on $D(1)$ is equivalent to one of the fine gradings discussed above, with universal groups $\mathbb{Z}^3$, $\mathbb{Z}_4\times\mathbb{Z}_2^2$, $\mathbb{Z}\times\mathbb{Z}_2^2$,  $\mathbb{Z}\times\mathbb{Z}_2^3$, $\mathbb{Z}^2\times\mathbb{Z}_2$ and $\mathbb{Z}_4^2$. Therefore, by \Cref{thm:gradings} and Propositions \ref{prop:IC_fine}, \ref{prop:IP_fine}, \ref{prop:IM_fine2}, \ref{prop:IIM_fine}, \ref{prop:IIC_fine} and \ref{prop:IIP_fine}, any $G$ grading is isomorphic to one of the types $\IC,\IP,\IM{1},\IIM,\IIC,\IIP$, and gradings of different types are not isomorphic. By the same theorem and Propositions \ref{prop:IC_weyl2}, \ref{prop:IP_weyl2}, \ref{prop:IM_weyl}, \ref{prop:IIM_weyl}, \ref{prop:IIC_weyl} and \ref{prop:IIP_weyl}, the action of the Weyl group in each type leads to the stated conditions.
    \end{proof}

    \begin{table}
    \begin{tabular}{|c|c|c|c|}
    \hline
    Family & Subgroups & Combinatorial data & Equivalence\\ \hline 
    $\gr{\IC}_D$ 
    & trivial & \vtop{\hbox{\strut $(h_1,h_2,h_3,h_4)\in G^4$}\hbox{\strut with $h_1h_2h_3h_4=e$}} & \vtop{\hbox{\strut inversion and}\hbox{\strut permutation} \hbox{\strut by $D_4\leq S_4$}}  \\ \hline 
    $\gr{\IP}_D$ 
    &  $H=\langle s\rangle\times\langle a,b\rangle\cong\mathbb{Z}_4\times\mathbb{Z}_2^2$ & \vtop{\hbox{\strut $\{aS,bS\}$ in $H_{[2]}/S$}\hbox{\strut where $S=\langle s^2\rangle$}} & \\ \hline 
    $\gr{\IM{1}}_D$ 
    & $T\cong \mathbb{Z}_2^2$ & $gT\in G/T$ & inversion \\ \hline 
    $\gr{\IIM}_D$ 
    & \vtop{\hbox{\strut $T\cong\mathbb{Z}_2^2$, $H=\langle h\rangle$}\hbox{\strut with $T\cap H=1$}} & $gH\in G/H$ & inversion \\ \hline
    $\gr{\IIC}_D$ 
    & $H=\langle h\rangle\cong\mathbb{Z}_2$ & $\{g_1^{\pm 1},g_2^{\pm 1}\}$ in $G$ & $H$-translation \\ \hline
    $\gr{\IIP}_D$ 
    & $T=\langle a,b\rangle\cong\mathbb{Z}_4^2$, $H=\langle a^2\rangle\leq T_{[2]}$ & &  \\ \hline
    \end{tabular}
    \vspace{5pt} 
    \caption{Classification of $G$-gradings on the Lie superalgebra $D(1)$}\label{tab:D(1)}
    \end{table}

		
\section{Gradings on $A(1,1)$}\label{sec:a}

    The Lie superalgebra $A(1,1)=\mathfrak{psl}(2\,|\,2)$ has an exceptionally large automorphism group among the members of the series $A(n,n)$. Recall that 
    \[ 
    \mathfrak{gl}(2\,|\,2)=\left\{ \begin{pmatrix}
        X_{11} & X_{12} \\
        X_{21} & X_{22}
    \end{pmatrix} \mid X_{ij} \in M_2(\mathbb{F}) \right\} 
    \]
    is a Lie superalgebra with even part corresponding to the blocks $X_{11}$ and $X_{22}$, odd part to the blocks $X_{12}$ and $X_{21}$, and bracket $[X,Y]=XY+(-1)^{|X||Y|}YX$. Its derived subalgebra   $\mathfrak{sl}(2\,|\,2)=\{ (X_{ij})\in \mathfrak{gl}(2\,|\,2)\mid \tr(X_{11})=\tr(X_{22})\}$ has a nontrivial center $\mathfrak{z}$ consisting of scalar matrices, and the quotient $\mathfrak{psl}(2\,|\,2)=\mathfrak{sl}(2\,|\,2)/\mathfrak{z}$ is simple.  

    It turns out that $\mathfrak{psl}(2\,|\,2)$ is isomorphic to the unique nontrivial ideal $(\mathfrak{s}_1\oplus\mathfrak{s}_2)\oplus \mathcal{L}_1$ of $D(-1)$. To define an explicit isomorphism, recall the alternating invariant form $\psi$ on the simple $2$-dimensional $\mathfrak{sl}_2$-module $V=\langle u,v\rangle$ such that $\psi(u,v)=1$. This form defines an isomorphism $V\cong V^*$ as $\mathfrak{sl}_2$-modules and, hence, an isomorphism $V\otimes V\cong\mathrm{End}(V)$, $x\otimes y\mapsto x\psi(y,\cdot)$, as $\mathfrak{sl}_2\oplus\mathfrak{sl}_2$ modules, where the action on $\mathrm{End}(V)\cong M_2(\mathbb{F})$ is given by $(X,Y).M=XM-MY$ for all $X,Y\in\mathfrak{sl}_2$ and $M\in M_2(\mathbb{F})$. 
    Under this isomorphism, the flip $x\otimes y\mapsto y\otimes x$ on $V\otimes V$ corresponds to the map $X\mapsto-\bar{X}$ on $M_2(\mathbb{F})$ 
    where bar denotes the adjunction with respect to $\psi$, which is the standard involution of $M_2(\mathbb{F})$ as a quaternion algebra (cf. Definitions \ref{def:IM} and \ref{def:IIM}). In particular, for any $X\in M_2(\mathbb{F})$, we have $X-\bar{X}=2X-\mathrm{tr}(X)I$. It follows that we can define
    an isomorphism $\Omega:(\mathfrak{s}_1\oplus\mathfrak{s}_2)\oplus\mathcal{L}_1\rightarrow \mathfrak{psl}(2\,|\,2)$ by
    \[ \Omega(X,Y)=  \begin{pmatrix}
        X & 0 \\
        0 & Y
    \end{pmatrix}+\mathfrak{z}, \ 
    \Omega(x\otimes y\otimes u + x'\otimes y'\otimes v)= \sqrt{2} \begin{pmatrix}
        0 & x\psi(y,\cdot) \\
        y'\psi(x',\cdot) & 0
    \end{pmatrix} \]
    for all $X,Y\in\mathfrak{sl}_2$ and $x,y,x',y'\in V=\langle u,v\rangle$.

    Restricting automorphisms of $\mathcal{L}=D(-1)$ to its unique nontrivial ideal and transporting via the isomorphism $\Omega$, we obtain a group homomorphism $\theta:\aut(\mathcal{L})\rightarrow \aut(\mathfrak{psl}(2\,|\,2))$. From \cite{S}, we see that $\theta$ is in fact an isomorphism. Explicitly, for any $f,g\in \slg_2(\mathbb{F})$, we have
    \[ 
    \theta(\iota(f,g,I)) \begin{pmatrix}
        X_{11} & X_{12} \\
        X_{21} & X_{22}
    \end{pmatrix} = \begin{pmatrix}
        fX_{11}f^{-1} & fX_{12}g^{-1} \\
        gX_{21}f^{-1} & gX_{22}g^{-1}
    \end{pmatrix}, 
    \]
    and, for any $h=(h_{ij})_{1\le i,j\le 2}\in \slg_2(\mathbb{F})$, we have
    \[ 
    \theta(\iota(I,I,h)) \begin{pmatrix}
        X_{11} & X_{12} \\
        X_{21} & X_{22}
    \end{pmatrix} = \begin{pmatrix}
        X_{11} & h_{11}X_{12}-h_{12}\overline{X}_{21} \\
        -h_{21}\overline{X}_{12}+h_{22}X_{21} & X_{22}
    \end{pmatrix}. 
    \]
    As to the outer automorphism $\varphi:(X_1,X_2,X_3)\in\mathcal{L}_0\mapsto (X_2,X_1,X_3)$, $x_1\otimes x_2\otimes x_3\in\mathcal{L}_1\mapsto ix_2\otimes x_1\otimes x_3$, we have
    \[ 
    \theta(\varphi) \begin{pmatrix}
        X_{11} & X_{12} \\
        X_{21} & X_{22}
    \end{pmatrix} = \begin{pmatrix}
        X_{22} & -i\overline{X}_{12} \\
        -i\overline{X}_{21} & X_{11}
    \end{pmatrix}. 
    \]

    Recall that a group $\tilde{G}$ is a \emph{central extension} of a group $G$ by an abelian group $A$ if there exists a short exact sequence
    \[ 
    1\rightarrow A \rightarrow \tilde{G} \rightarrow G \rightarrow 1 
    \]
    such that the image of $A$ is in the center of $\tilde{G}$. The isomorphism classes of central extensions of $G$ by $A$ (with $G$ and $A$ fixed point-wise) are in bijection with the second cohomology group $\mathrm{H}^2(G,A)$ for the trivial action of $G$ on $A$.

    \begin{definition}\label{def:perturbation}
        Let $\tilde{G}$ be a central extension of $G$ by $A$ and let $\varepsilon:G\times G\rightarrow A$ be a 2-cocycle (for trivial action). Let $\tilde{\varepsilon}:\tilde{G}\times \tilde{G}\rightarrow A$ be the composition of $\varepsilon$ with the projection $\tilde{G}\rightarrow G$.
        We define $\tilde{G}_\varepsilon$ 
        as the group with the same underlying set as $\tilde{G}$ and with multiplication 
        perturbed by $\varepsilon$ as follows: $g*h=\tilde{\varepsilon}(g,h)gh$ for all $g,h\in \tilde{G}_\varepsilon$.
    \end{definition}

    \begin{remark}\label{rem:coboundary}
        The group $\tilde{G}_\varepsilon$ of \Cref{def:perturbation} depends, up to isomorphism, only on the cohomology class $[\varepsilon]\in\mathrm{H}^2(G,A)$. In fact, let $\varepsilon$ be a coboundary, i.e. there exists a function $\tau:G\rightarrow A$ such that $\varepsilon(g,h)=d\tau(g,h)=\tau(gh)\tau(g)^{-1}\tau(h)^{-1}$. Then the map $T:\tilde{G}\rightarrow\tilde{G}_\varepsilon$ given by $T(g)=\tilde{\tau}(g)g$, for all $g\in \tilde{G}$, is a group isomorphism.
    \end{remark}

    If $\tilde{H}$ is a subgroup of $\tilde{G}$ containing $A$ and $H$ is the projection of $\tilde{H}$ to $G$, then $\tilde{H}$ is a central extension of $H$ by $A$, so we can define $\tilde{H}_\varepsilon$ using the restriction of $\varepsilon$ to $H$ and identify $\tilde{H}_\varepsilon$ with a subgroup of $\tilde{G}_\varepsilon$.

    \begin{lemma}\label{lm:quasitori}
        Suppose that an algebraic group $\tilde{G}$ is a central extension of $G$ by a quasitorus $A$ and that $\varepsilon:G\times G\to A$ is a symmetric $2$-cocycle. Then $\tilde{Q}\mapsto\tilde{Q}_\varepsilon$ is a bijection between the quasitori of $\tilde{G}$ and $\tilde{G}_\varepsilon$ that contain $A$. Moreover, this bijection preserves conjugacy classes and Weyl groups.
    \end{lemma}

    \begin{proof}
    The mapping $\tilde{H}\mapsto\tilde{H}_\varepsilon$ is a bijection between the subgroups of $\tilde{G}$ and $\tilde{G}_\varepsilon$ that contain $A$, since $\tilde{H}$ and $\tilde{H}_\varepsilon$ correspond to the same subgroup of $G$. It also follows that this bijection preserves the conjugacy classes of such subgroups and that  $N(\tilde{H}_\varepsilon)=N(\tilde{H})_\varepsilon$. Since $\varepsilon$ is symmetric, we have $gh=hg$ if and only if $g*h=h*g$. Therefore, the bijection preserves commutativity,  $C(\tilde{H}_\varepsilon)=C(\tilde{H})_\varepsilon$, and $N(\tilde{H}_\varepsilon)/C(\tilde{H}_\varepsilon)\cong N(\tilde{H})/C(\tilde{H})$. The result now follows from the fact that, over an algebraically closed field, a commutative extension of a quasitorus by a quasitorus is itself a quasitorus.
    \end{proof}

    We will apply these considerations to the groups $\aut(D(-1))$ and $\aut(D(1))$, which are both central extensions of $\psl_2(\mathbb{F})^3\rtimes C_2$ by $\langle\delta\rangle\cong C_2$, where $\delta$ is the parity automorphism and the projection on $\psl_2(\mathbb{F})^3\rtimes C_2$ is given by the restriction to $\mathcal{L}_0$.
     
    \begin{proposition}\label{prop:perturbation}
        Let $G=\psl_2(\mathbb{F})^3\rtimes C_2$ and let $\pi:G\to\mathbb{Z}_2$ be the quotient map modulo  $\psl_2(\mathbb{F})^3$. Define  $\varepsilon:G\times G\rightarrow \langle \delta\rangle$ by $\varepsilon(g,h)=\delta^{\pi(g)\pi(h)}$, for all $g,h\in G$. 
        Then $\varepsilon$ is a $2$-cocycle, $\aut(D(-1))\cong\aut(D(1))_\varepsilon$, and there is a bijective correspondence between maximal quasitori of $\aut(D(1))$ and $\aut(D(-1))$, which preserves conjugacy classes and Weyl groups.
    \end{proposition}
    \begin{proof}
        Consider the operator $\sqrt{\delta}:\mathcal{L}\to \mathcal{L}$ that acts as the identity on $\mathcal{L}_0$ and as multiplication by $i$ on $\mathcal{L}_1$. Then define the function $\tau:G\to \langle \sqrt{\delta}\rangle$ by $\tau(g)=\mathrm{id}$ if $\pi(g)=0$ and $\tau(g)=\sqrt{\delta}$ if $\pi(g)=1$. Formally, we have $\varepsilon=d\tau$, which shows that $\varepsilon$ is a $2$-cocycle, but note that it is not a coboundary since $\tau$ does not take values in $\langle\delta\rangle$. Nevertheless, similarly to \Cref{rem:coboundary}, for any $\alpha\in\aut(D(-1))$, we can define $T(\alpha)=\tilde{\tau}(\alpha)\alpha$, and it is easy to check that $T(\alpha)$ is an automorphism of $D(1)$. Therefore, the map $T:\aut(D(-1))\rightarrow \aut(D(1))_\varepsilon$ is a well defined group isomorphism. Since any maximal quasitorus of $D(1)$ or $D(-1)$ contains $\langle\delta\rangle$, the proof is complete by \Cref{lm:quasitori}.
        %
        %
    \end{proof}

    \begin{corollary}
        The fine gradings on $\mathfrak{psl}(2\,|\,2)$ or $D(-1)$ are in bijective correspondence with those on $D(1)$. Up to equivalence, they are the following: 
        \begin{itemize}
            \item $\mathbb{Z}^3 $-grading, with Weyl group isomorphic to $C_2^3\rtimes C_2$, of type $(14,0,1)$ on $D(-1)$ and of type $(12,1)$ on $\mathfrak{psl}(2\,|\,2)$;
            \item $\mathbb{Z}_4\times\mathbb{Z}_2^2 $-grading, with Weyl group isomorphic to $C_2^5\rtimes C_2$, of type $(14,0,1)$ on $D(-1)$ and of type $(12,1)$ on $\mathfrak{psl}(2\,|\,2)$;
            \item $\mathbb{Z}\times\mathbb{Z}_2^2 $-grading, with Weyl group isomorphic to $C_2\times(C_2^2 \rtimes S_3)$, of type $(11,3)$ on $D(-1)$ and of type $(14)$ on $\mathfrak{psl}(2\,|\,2)$;
            \item $\mathbb{Z}\times\mathbb{Z}_2^3 $-grading, with Weyl group isomorphic to $C_2\times(C_2^2 \rtimes S_3)$, of type $(17)$ on $D(-1)$ and of type $(14)$ on $\mathfrak{psl}(2\,|\,2)$;
            \item $\mathbb{Z}^2\times\mathbb{Z}_2 $-grading, with Weyl group isomorphic to $C_2^3$, of type $(15,1)$ on $D(-1)$ and of type $(14)$ on $\mathfrak{psl}(2\,|\,2)$;
            \item $\mathbb{Z}_4^2 $-grading, with Weyl group isomorphic to $C_2^4\rtimes C_2$, of type $(13,2)$ on $D(-1)$ and of type $(14)$ on $\mathfrak{psl}(2\,|\,2)$.
        \end{itemize}
    \end{corollary}
    \begin{proof}
        Using \Cref{prop:perturbation} for each maximal quasitorus corresponding to a fine grading on $D(1)$, we obtain a fine grading on $D(-1)$, which has the same subspaces as homogeneous components. By direct computation, it is easy to see that the universal groups are as stated. Restricting the grading to the unique nontrivial ideal of $D(-1)$ isomorphic to $\mathfrak{psl}(2\,|\,2)$, we complete the proof.
    \end{proof}

    It follows that the classification of $G$-gradings on $D(-1)$ up to isomorphism is the same as on $D(1)$. As for $\mathfrak{psl}(2\,|\,2)$, restricting even derivations gives an isomorphism $\mathrm{Der}(D(-1))_0\cong\mathrm{Der}(\mathfrak{psl}(2\,|\,2))_0$, so the almost fine gradings on $\mathfrak{psl}(2\,|\,2)$ are the restrictions of those on $D(-1)$: the fine gradings and one coarsening of the fine $\mathbb{Z}\times\mathbb{Z}_2^3$-grading (see \Cref{prop:IIM_fine}), which remains proper for $\mathfrak{psl}(2\,|\,2)$. Moreover, since $D(-1)$ and $\mathfrak{psl}(2\,|\,2)$ have the same odd part, it is easy to check that the admissible homomorphisms remain the same, namely, injective on the torsion subgroup of the universal group. We will present the resulting classification of $G$-gradings on $\mathfrak{psl}(2\,|\,2)$ in terms of block matrices. To simplify notation, it will be convenient to identify $M_4(\mathbb{F})$ with $M_2(\mathbb{F})\otimes M_2(\mathbb{F})$ where the superalgebra structure is carried by the first factor:
    a matrix unit $E_{jk}\in M_2(\mathbb{F})$ is even if $j=k$ and odd if $j\ne k$. We also use the $2\times 2$ matrices from \Cref{examples}.

    \begin{definition}\label{def:A}
        We have the following $G$-gradings on $\mathfrak{psl}(2\,\,|2)$: 
        \begin{itemize}
            \item $\Gamma^{(1)}_{A(1,1)}(g_1,g_2,g_3,g_4)$, 
            with any $g_j\in G$, is the grading induced by the elementary grading (see \Cref{def:elemM}) on $M_4(\mathbb{F})$ defined by $\gamma=(g_1,g_2,g_3,g_4)$;
            
            \item $\Gamma^{(2)}_{A(1,1)}(s,a,b)$, 
            where $s,a,b\in G$ such that $s^4=a^2=b^2=e$ and $\langle s,a,b\rangle\cong \mathbb{Z}_4\times\mathbb{Z}_2^2$, is defined by $\deg(E_{11}\otimes C)=\deg(E_{22}\otimes C)=s^2$, $\deg(E_{11}\otimes A)=a$, $\deg(E_{22}\otimes A)=b$, $\deg(E_{11}\otimes B)=s^2a$, $\deg(E_{22}\otimes B)=s^2b$ and, for $r\in\{0,1\}$ and $\{j,k\}=\{1,2\}$, $\deg(E_{jk}\otimes A+(-1)^ri E_{kj}\otimes B)=s^{k-j}a^rb^r$, $\deg(E_{12}\otimes I+(-1)^ri E_{21}\otimes C)=sa^rb^{r+1}$ and $\deg(E_{21}\otimes I+(-1)^ri E_{12}\otimes C)=s^{-1}a^{r+1}b^r$;
            
            \item $\Gamma^{(3)}_{A(1,1)}(g,T)$, 
            where $g\in G$ and $T=\{e,a,b,c\}\cong \mathbb{Z}_2^2$ is a subgroup of $G$, is defined by $\deg(E_{11}\otimes H)=e$, $\deg(E_{11}\otimes E)=g^2$, $\deg(E_{11}\otimes F)=g^{-2}$, $\deg(E_{22}\otimes A)=a$, $\deg(E_{22}\otimes B)=b$, $\deg(E_{22}\otimes C)=c$, 
            and, for $r\in\{0,1\}$, $\deg(E_{13} +(-1)^r E_{32})=ga^r$, $\deg(E_{14} +(-1)^r E_{42})=ga^rb$, $\deg(E_{23} +(-1)^r E_{31})=g^{-1}a^r$ and $\deg(E_{24} +(-1)^r E_{41})=g^{-1}a^rc$;
            
            \item $\Gamma^{(4)}_{A(1,1)}(g,h,T)$, 
            where $T=\{e,a,b,c\}\cong \mathbb{Z}_2^2$ is a subgroup of $G$ and $g,h\in G$ such that $h^2=e$ and $h\notin\{a,b,c\}$, is defined by setting, for each $r\in\{0,1\}$, $\deg(H^r\otimes A)=ah^r$, $\deg(H^r\otimes B)=bh^r$, $\deg(H^r\otimes C)=ch^r$, $\deg(E\otimes I)=g$, $\deg(E\otimes A)=gah$, $\deg(E\otimes B)=gbh$, $\deg(E\otimes C)=gch$, $\deg(F\otimes I)=g^{-1}h$, $\deg(F\otimes A)=g^{-1}a$, $\deg(F\otimes B)=g^{-1}b$ and $\deg(F\otimes C)=g^{-1}c$;
            
            \item $\Gamma^{(5)}_{A(1,1)}(g_1,g_2,h)$, 
            where $g_1,g_2,h\in G$ with $h$ of order $2$, is defined setting, for $r\in\{0,1\}$, $\deg(H^r\otimes H)=h^r$, $\deg(H^r\otimes E)=g_2h^r$, $\deg(H^r\otimes F)=g_2^{-1}h^r$, $\deg(E\otimes I)=g_1$, $\deg(E\otimes H)=g_1h$, $\deg(E\otimes E)=g_1g_2h$, $\deg(E\otimes F)=g_1g_2^{-1}h$, $\deg(F\otimes I)=g_1^{-1}h$, $\deg(F\otimes H)=g_1^{-1}$, $\deg(F\otimes E)=g_1^{-1}g_2$ and $\deg(F\otimes F)=g_1^{-1}g_2^{-1}$;
            
            \item $\Gamma^{(6)}_{A(1,1)}(a,b)$, 
            where $a^4=b^4=e$ and $T:=\langle a,b\rangle \cong \mathbb{Z}_4^2$, is the grading induced by the Pauli grading on $M_4(\mathbb{F})$ from \Cref{def:PauliM} with superalgebra structure defined by the quotient map $T\to T/\langle a,b^2\rangle\cong\mathbb{Z}_2$.
        \end{itemize}
    \end{definition}

    \begin{theorem}\label{thm:psl}
        Let $G$ be an abelian group. Then every $G$-gradings on $A(1,1)$ is 
        isomorphic to one in \Cref{def:A}. Moreover, gradings from different families are not isomorphic, and the isomorphism classes within each family are parametrized by certain subgroups of $G$ and equivalence classes of certain multisets as indicated in Table \ref{tab:A}. \qed 
    \end{theorem}

    \begin{table}
    \begin{tabular}{|c|c|c|c|}
    \hline
    Family & Subgroups & Combinatorial data & Equivalence\\ \hline 
    $\gr{1}_{A(1,1)}$ 
    & trivial & \vtop{\hbox{\strut unordered pair of}\hbox{\strut multisets in $G$:}\hbox{\strut $\{g_1,g_2\}$, $\{g_3,g_4\}$}} & \vtop{\hbox{\strut simultaneous}\hbox{\strut $G$-translation,}\hbox{\strut simultaneous}\hbox{\strut inversion}}  \\ \hline 
    $\gr{2}_{A(1,1)}$ 
    &  $H=\langle s\rangle\times\langle a,b\rangle\cong\mathbb{Z}_4\times\mathbb{Z}_2^2$ & \vtop{\hbox{\strut $\{aS,bS\}$ in $H_{[2]}/S$}\hbox{\strut where $S=\langle s^2\rangle$}} & \\ \hline 
    $\gr{3}_{A(1,1)}$ 
    & $T\cong \mathbb{Z}_2^2$ & $gT\in G/T$ & inversion \\ \hline 
    $\gr{4}_{A(1,1)}$ 
    & \vtop{\hbox{\strut $T\cong\mathbb{Z}_2^2$, $H=\langle h\rangle$}\hbox{\strut with $T\cap H=1$}} & $gH\in G/H$ & inversion \\ \hline
    $\gr{5}_{A(1,1)}$ 
    & $H=\langle h\rangle\cong\mathbb{Z}_2$ & $\{g_1^{\pm 1},g_2^{\pm 1}\}$ in $G$ & $H$-translation \\ \hline
    $\gr{6}_{A(1,1)}$ 
    & $T=\langle a,b\rangle\cong\mathbb{Z}_4^2$, $H=\langle a^2\rangle\leq T_{[2]}$ & &  \\ \hline
    \end{tabular}
    \vspace{5pt} 
    \caption{Classification of $G$-gradings on the Lie superalgebra $A(1,1)$}\label{tab:A}
    \end{table}

\end{document}